%% file: main.tex
\documentclass[a4paper, 10pt]{article}
\usepackage[a4paper,margin=1in]{geometry}
\usepackage{graphicx}
\usepackage{verbatim}
\usepackage{stmaryrd}
\usepackage{amsmath}
\usepackage{comment}
\usepackage{listings}
\usepackage[table,xcdraw]{xcolor}
\usepackage{color}
\usepackage{amsthm}
\usepackage{times}
\usepackage[hidelinks]{hyperref}
\usepackage{multirow}
\usepackage{makecell}
\usepackage{booktabs}
\usepackage{array}
\usepackage{enumitem}
\usepackage{amssymb}
\usepackage{subcaption}
\usepackage{siunitx}
\usepackage{pifont}
\usepackage[normalem]{ulem}

\definecolor{dkgreen}{rgb}{0,0.6,0}
\definecolor{gray}{rgb}{0.5,0.5,0.5}
\definecolor{mauve}{rgb}{0.58,0,0.82}

\newcommand*{\ldblbrace}{\{\mskip-5mu \left \{}
\newcommand*{\rdblbrace}{ \}\mskip-5mu \right \}}
\newcommand{\mean}[1]{\left\ldblbrace #1 \right\rdblbrace}
\newcommand{\mtx}[1]{\boldsymbol{#1}}
\renewcommand{\vec}[1]{\boldsymbol{#1}}
\newcommand{\norm}[1]{\left\lVert #1 \right\rVert}
\newcommand{\jump}[1]{\left\llbracket #1  \right\rrbracket}
\newtheorem{theorem}{Theorem}[section]
\newtheorem{corollary}[theorem]{Corollary}
\newtheorem{lemma}[theorem]{Lemma}
\newtheorem*{remark}{Remark}
\theoremstyle{definition}
\newtheorem{definition}[theorem]{Definition}

\newcommand{\revA}[1]{#1}
\newcommand{\revB}[1]{#1}

\newcommand{\orcid}[1]{ORCID:~\href{https://orcid.org/#1}{#1}}
\usepackage{authblk}

\newenvironment{keywords}{\par\textbf{Key words.}}{\par}
\newenvironment{AMS}{\par\textbf{AMS subject classification.}}{\par}

\title{Structure-Preserving High-Order Methods for the Compressible Euler Equations in Potential Temperature Formulation for Atmospheric Flows}

\date{February 23, 2026} 

\author[1]{Marco Artiano\thanks{\orcid{0009-0009-5872-702X}}}
\affil[1]{Institute of Mathematics, Johannes Gutenberg University Mainz, Germany}

\author[2]{Oswald Knoth}
\affil[2]{Leibniz Institute for Tropospheric Research (TROPOS), Leipzig, Germany}

\author[3]{Peter Spichtinger\thanks{\orcid{0000-0003-4008-4977}}}
\affil[3]{Institute for Atmospheric Physics, Johannes Gutenberg University Mainz, Germany}

\author[1]{Hendrik Ranocha\thanks{\orcid{0000-0002-3456-2277}}}
\begin{document}
\maketitle

\begin{abstract}
\noindent
   We develop structure-preserving numerical methods for the compressible Euler equations, employing potential temperature as a prognostic variable.
   We construct three numerical fluxes designed to ensure the conservation of entropy and total energy within the discontinuous Galerkin framework on general curvilinear meshes.
   Furthermore, we introduce a generalization for the kinetic energy preservation property and total energy conservation in the presence of a gravitational potential term.
   To this end, we adopt a flux-differencing approach for the discretization of the source term, treated as non-conservative product.
   We present well-balanced schemes for different constant background states for both formulations (total energy and potential temperature) on curvilinear meshes.
   Finally, we validate the methods by comparing the potential temperature formulation with the traditional Euler equations formulation across a range of classical atmospheric scenarios.
\end{abstract}



\input{introduction.tex}
\input{section1.tex}
\input{section2.tex}
\input{section3.tex}
\input{section4.tex}

\input{section5.tex}
\input{section6.tex}

\appendix

\input{acknowledgments.tex}

\section*{Data availability statement}

All Julia source code and data needed to reproduce the numerical results
presented in this paper are available in our reproducibility repository
\cite{artiano2025structureRepro}.

\bibliographystyle{elsarticle-num}
\bibliography{refs}

\end{document}

%% file: introduction.tex
\section{Introduction}
Discontinuous Galerkin (DG) methods have gained significant attention for hyperbolic equations due to their ability to accurately capture complex fluid and wave phenomena in applications such as atmospheric physics. Their compact formulation enables excellent parallel scalability, making them suitable for high-resolution simulations \cite{NUMA}.
However, DG methods often require additional stabilization when handling discontinuities or under-resolved features. Especially in atmospheric simulations, to keep stability for long-time simulations, filters are applied \cite{GIRALDO20083849}, as well as other classical techniques such as over-integration (dealiasing) \cite{MENGALDO201556} and artificial viscosity \cite{ULLRICH2018427,CiCP-27-5}. While effective, these approaches can reduce accuracy and require parameter tuning.
To address these issues, entropy-stable (ES) DG schemes have been developed, offering high-order accuracy with guaranteed nonlinear stability without requiring additional parameters \cite{WARUSZEWSKI2022111507,Souza,LIU2025114095,winters2018comparative}.
An essential ingredient of these high-order DG schemes is the implementation of carefully constructed two-point fluxes for the volume terms, an approach pioneered in the context of DG methods by Carpenter et al.\ \cite{Carpenter2014, CarpenterBook}, Fisher et al.\ \cite{FISHER2013518}, and Gassner et al.\ \cite{gassner2016split}.

The compressible Euler equations admit several formulations, and it is not entirely clear whether one formulation offers significant advantages over the others.
In the computational fluid dynamics (CFD) community, considerable effort has been devoted to the development of structure-preserving schemes using the total energy as state variable \cite{Jameson2008,ismail2009affordable,ECChandra,Ranocha2018,Ranocha2022,WARUSZEWSKI2022111507,Souza,LIU2025114095,sjogreen2018high}.
In contrast, the use of potential temperature is especially popular in atmospheric models, where it also plays an important theoretical role in meteorology. Giraldo et al.\ \cite{GIRALDO20083849} and more recently Girfoglio et al.\ \cite{GIRFOGLIO2025106510} have analyzed different formulations for DG and spectral element methods (SEMs) and different pressure-based solvers for the finite volume framework, respectively.
In atmospheric applications, the presence of the gravitational potential is crucial for designing well-balanced schemes that preserve the hydrostatic equilibrium~\cite{ChandrashekarWB,WARUSZEWSKI2022111507,LIU2025114095} \revA{or even more general states~\cite{michel2025towards}}.
Additional desirable properties in the flux-differencing framework include the preservation of kinetic and potential energy (KPEP), as introduced by Souza et al.\ \cite{Souza}.

\revA{The potential temperature formulation is widely used in atmospheric applications~\cite{ABlendedSoundprooftoCompressibleNumericalModelforSmalltoMesoscaleAtmosphericDynamics,GIRALDO20083849, ICON2015,BALDAUF2021110635,wrf_version}.
Thus, developing structure-preserving methods for this formulation is of great interest for the atmospheric modeling community to simplify the implementation in existing models.
Moreover, a particular characteristic of the potential temperature formulation is that the thermodynamic equation (in conservative form) reduces to a simple transport equation.
This property makes this formulation also particularly suitable for (asymptotic-preserving) implicit-explicit (IMEX) time integration method \cite{GiraldoIMEX2013,BISPEN2017222,gmd-11-1497-2018}.
In the literature, several efforts have already been made to design structure-preserving methods with potential temperature as prognostic variable for different classes of spatial discretizations, including finite volume methods~\cite{Gassmann2013}, continuous spectral element methods combined with finite differences in the vertical direction~\cite{Mark2020}, and finite element methods~\cite{WimmerGolo2021}.}

In this work, \revA{first} we focus on the potential temperature formulation to investigate structure-preserving properties, in particular pressure equilibrium preserving (PEP). We design new two-point fluxes that guarantee conservation of either total energy (TEC) or thermodynamic entropy (EC). Building on the ideas in \cite{WARUSZEWSKI2022111507,Souza}, we generalize KPEP, introduce the condition for total energy conservation in presence of a geopotential, and propose a new well-balanced scheme by discretizing the source term as non-conservative product; all these properties are established for both the potential temperature and the total energy formulations. \revA{Furthermore, all the properties aforementioned are extended} to the DGSEM~\cite{Kopriva2009} using arbitrary curvilinear coordinates\revA{, including well-balancedness for isothermal and constant potential temperature background states for both formulation on curvilinear meshes.}

\revA{This paper is inspired by~\cite{GIRALDO20083849, WARUSZEWSKI2022111507,Souza}.
The main differences to~\cite{WARUSZEWSKI2022111507} concern the formulation of the energy equation and the set of equations considered.
We study two sets of equations: the potential temperature formulation and a total energy–type formulation in which the prognostic variable includes only internal and kinetic energy, while the gravitational potential is treated as a source term, in contrast to Waruszewski et al.~\cite{WARUSZEWSKI2022111507}, who evolve the full total energy in conservative form.
Moreover, Waruszewski et al.~\cite{WARUSZEWSKI2022111507} prove the well-balancedness of their scheme for the isothermal background state on a mesh aligned with the gravitational force, while we present a more general well-balanced scheme that can preserve both the isothermal and the constant potential temperature background states on curved meshes.}

The paper is organized as follows.
We introduce the Euler equations in the potential temperature formulation in Section~\ref{sec:fluxes} and derive novel EC and TEC two-point fluxes.
Section~\ref{sec:two_point_flux_properties} analyzes their structure-preserving properties, and discusses positivity of pressure and density.
Section~\ref{sec:additional_conservation_properties} presents the generalization of KPEP and TEC for flux differencing, and the new well-balanced scheme.
Section~\ref{sec:dgsem} considers the DGSEM in one dimension on Cartesian grids, while Section~\ref{sec:dgsem_curved} generalizes these flux differencing properties to three-dimensional DGSEM on curvilinear meshes.
Finally, in Section~\ref{sec:numerical_results} we report the numerical results comparing different formulations and two-point fluxes.
We summarize our results in Section~\ref{sec:conclusions}.

%% file: section1.tex
\section{Total energy and entropy-conservative fluxes}
\label{sec:fluxes}
The compressible Euler equations with the total energy as conservative variable read
\begin{equation}\label{EulerTotalEnergy}
	\begin{aligned}
		 & \partial_t \varrho + \nabla \cdot \left (\varrho \vec{V} \right ) = 0,                                            \\
		 & \partial_t (\varrho \vec{V}) + \nabla \cdot ( \varrho \vec{V} \otimes \vec{V}) + \nabla p = -\varrho \nabla \phi,    \\
		 & \partial_t (\varrho E) + \nabla \cdot \left ( (\varrho E + p) \vec{V} \right ) = -\varrho \vec{V} \cdot \nabla \phi,
	\end{aligned}
\end{equation}
where $\varrho$ is the density, $\vec{V}$ the velocity, $\varrho E$ is the sum of kinetic and internal energy, $p$ the pressure, and $\phi = \phi(x,y,z)$ is the gravitational potential.
We use an ideal gas equation of state, i.e.,
$\varrho E = \frac{p}{\gamma - 1} + \frac{1}{2} \varrho \norm{\vec{V}}^2$,
where $\gamma$ is the ratio of specific heats.
In the context of local weather prediction and numerical methods for atmospheric applications, the evolution of $\varrho E$ is often substituted with the evolution of the potential temperature $\theta$, i.e.,
\begin{equation}\label{EulerPotentialTemperature}
	\begin{aligned}
		 & \partial_t \varrho + \nabla \cdot \left (\varrho \vec{V} \right ) = 0,                                         \\
		 & \partial_t (\varrho \vec{V}) + \nabla \cdot ( \varrho \vec{V} \otimes \vec{V}) + \nabla p = -\varrho \nabla \phi, \\
		 & \partial_t (\varrho \theta) + \nabla \cdot \left (\varrho \theta \vec{V} \right ) = 0,
	\end{aligned}
\end{equation}
where $p = p_0 \left ( R \varrho \theta / p_0 \right )^\gamma$, $R$ is the ideal gas law constant, and $p_0$ is the atmospheric pressure, which we assume to be $\SI{100000}{Pa}$ throughout this study.
For smooth solutions, the conservative parts of the systems of equations~\eqref{EulerTotalEnergy} and~\eqref{EulerPotentialTemperature} induce additional scalar conservation laws, which are of the form
\begin{equation}
	\partial_t U + \nabla \cdot F \left ( U \right ) = 0,
	\label{EntropyEuler}
\end{equation}
where $U$ is a mathematical entropy functional and $F$ the corresponding flux. For the total energy formulation \eqref{EulerTotalEnergy}, we use the thermodynamic entropy
\begin{equation}
	U_{\varrho s} = \varrho s = \varrho \log \left ( \frac{p}{\varrho^\gamma} \right ).
\end{equation}
Usually, this quantity is set as an entropy functional for the derivation of entropy-conservative (EC) or entropy-stable (ES) two-point fluxes.
However, considering the conservative part of the system of our interest~\eqref{EulerPotentialTemperature}, $\varrho E$ (which coincides with the total energy in the absence of the geopotential term) is not a linear invariant. Therefore, we are interested in the construction of two-point numerical fluxes when either the thermodynamic entropy or the total energy are treated as entropy functional, i.e., we consider the mathematical entropies
\begin{equation}
U_{\varrho s} = \varrho s \quad \text{or} \quad U_{\varrho E} = \varrho E.
\end{equation}

\subsection{Entropy analysis}\label{entropy_analysis_section}
We consider the systems in 1D without gravity (the generalization to multiple dimensions is straightforward), i.e., hyperbolic conservation laws of the form
\begin{equation}
	\partial_t \vec{u} + \partial_x \vec{f}(\vec{u}) = 0.
	\label{HyperbolicConservationLaw1D}
\end{equation}
A semi-discrete finite volume (FV) method for \eqref{HyperbolicConservationLaw1D} reads
\begin{equation}
	\partial_t \vec{u}_i + \frac{1}{\Delta x} \big( \vec{f}^\mathrm{num}(\vec{u}_i,\vec{u}_{i+1}) - \vec{f}^\mathrm{num}(\vec{u}_{i-1},\vec{u}_{i}) \big) = 0.
    \label{FVmethod1D}
\end{equation}
From hereinafter, if not stated otherwise, we always consider the right interface at the $i$-th cell, i.e., $\vec{f}^\mathrm{num} := \vec{f}^\mathrm{num}(\vec{u}_i,\vec{u}_{i+1})$. Moreover, we use the classical operators
\begin{equation}
\label{properties_discrete_mean}
\begin{aligned}
	\jump{a} &=               a_\revA{+} - a_{\revA{-}},                 &\qquad& (\text{jump})\\
	\mean{a} &=               \frac{a_\revA{+} + a_\revA{-}}{2},       &\qquad& (\text{arithmetic mean})\\
	\mean{a}_{\mathrm{log}} &=  \frac{\jump{a}}{\jump{\log{a}}}, &\qquad& (\text{logarithmic mean, cf.\ \cite{ismail2009affordable}})\\
	\mean{a}_{\mathrm{geo}} &=  \sqrt{a_{i+1} a_{i}},            &\qquad& (\text{geometric mean})\\
	\mean{a}_\gamma &= \frac{\gamma-1}{\gamma}\frac{\jump{a^{\gamma}}}{\jump{a^{\gamma-1}}}, &\qquad& (\text{Stolarsky mean, cf.\ \cite{Winters2020}})
\end{aligned}
\end{equation}
\revA{where $a_\pm$ denote the values of $a$ at the right/left side of an interface, and} with the well-known properties
\begin{equation}
\label{discrete_chain_rules}
	\jump{ab}  = \mean{a}\jump{b} + \mean{b}\jump{a},
	\quad
	\jump{a^2} = 2 \jump{a}\mean{a},
	\quad
	\jump{1/a} = - \frac{\jump{a}}{\mean{a}_{\mathrm{geo}}^2}.
\end{equation}
Let $(U, F)$ denote a mathematical entropy-entropy flux pair for \eqref{HyperbolicConservationLaw1D}.

\begin{definition}[Tadmor \cite{Tadmor1987,Tadmor2003}]
A two-point numerical flux $\vec{f}^\mathrm{num}$ is EC for a given entropy $U$ if
\begin{equation}
\label{eq:TadmorDef}
	\jump{\vec{\omega}^T}\vec{f}^\mathrm{num} - \jump{\psi} = 0,
\end{equation}
where $\vec{\omega} = U'$ are the entropy variables and $\psi = \vec{\omega}^{T} \vec{f} - F$ is the flux potential.
\label{TadmorDef}
\end{definition}
\revA{The one-dimensional compressible Euler equations with the potential temperature in conservative form are
\begin{equation}
    \partial_t \begin{pmatrix}
           \varrho \\
           \varrho v\\
           \varrho \theta
         \end{pmatrix} + \partial_x \begin{pmatrix}
           \varrho v \\
           \varrho v^2 + p \\
           \varrho \theta v
         \end{pmatrix}
 =  \vec{0}.
 \label{EulerPotentialTemperature1D}
\end{equation}}%
It is of particular interest that the Hessian of the total energy and of the thermodynamic entropy, when expressed with respect to $\vec{u} = \left (\varrho, \varrho v, \varrho \theta \right )$, are given by
\begin{equation}
	 U_{\varrho E}''(\vec{u}) = \begin{bmatrix}
	    m^2/\varrho^3         & -m/\varrho^2 & 0           \\
		-m/\varrho^2                              & 1/\varrho & 0                                         \\
		0 & 0 & K \gamma (\varrho \theta)^{\gamma - 2}
	\end{bmatrix}, \quad  U_{\varrho s}''(\vec{u}) = \begin{bmatrix}
		\frac{\gamma}{\varrho}         & 0 & -\frac{\gamma}{\varrho \theta}            \\
		0                              & 0 & 0                                         \\
		-\frac{\gamma}{\varrho \theta} & 0 & \frac{\gamma \varrho}{(\varrho \theta)^2}
	\end{bmatrix},
\end{equation}
\revA{where $K = p_0 (R / p_0)^\gamma$ and $m = \varrho v$},
with eigenvalues
\begin{equation}
         \lambda_{\varrho E, 1} = 0, \quad \lambda_{\varrho E,2} = \frac{m^2 + \varrho^2}{\varrho^3}, \quad
         \lambda_{\varrho E, 3} = K \gamma (\varrho \theta)^{\gamma - 2},
\end{equation}
\begin{equation}
	 \lambda_{\varrho s, 1} = \lambda_{\varrho s, 2} = 0, \quad \lambda_{\varrho s, 3} = \frac{\gamma}{\varrho} + \gamma \frac{\varrho}{(\varrho \theta)^2}.
\end{equation}
Both Hessians are positive semidefinite, which implies that the entropy functionals are merely convex. As a consequence, the mapping between the conserved and the entropy variables is no longer one-to-one \cite{godlewski2021numerical}, in both cases. Nonetheless, strict convexity is not required for Tadmor's analysis to hold, as the conditions stated in Definition~\ref{TadmorDef} are still valid.
Therefore, this does not preclude the development of total energy-conservative (TEC) and thermodynamic entropy-conservative (EC) numerical fluxes, presented in the next sections. On the other hand, the conditions of Barth's Theorem \cite{Barth1999} cannot be directly applied to uniquely determine a dissipation operator \cite{WINTERS2017274}; as a result the construction of such operators is not straightforward.
\subsection{Derivation of a TEC flux}
Next we use the total energy $U_{\varrho E} = \varrho E$ as mathematical entropy for the system~\eqref{EulerPotentialTemperature} without gravity. Thus, given the set  of conservative variables $\vec{u} = (\varrho, \varrho v, \varrho \theta)$, we look for a numerical flux mimicking at the discrete level the conservation of the total energy, given in \eqref{EulerTotalEnergy}.
By introducing the conserved variables
\begin{equation}
	\vec{u} = (\varrho, \varrho v, \varrho \theta) = \left (\varrho, m, \varrho \theta \right ),
\end{equation}
we can rewrite the mathematical entropy (total energy) and the corresponding flux as
\begin{equation}
	 U_{\varrho E} = \varrho E = K\frac{(\varrho \theta)^\gamma}{\gamma - 1} + \frac{1}{2} \frac{m^2}{\varrho},
	 \qquad
	 F_{\varrho E} = \left (U_{\varrho E} + K(\varrho \theta)^\gamma\right ) \frac{m}{\varrho}.
\end{equation}
Thus, the mathematical entropy variables and the associated flux potential are
\begin{equation}
	 \vec{\omega}^T_\revA{{\varrho E}} = \left(-\frac{1}{2}\frac{m^2}{\varrho^2}, \frac{m}{\varrho}, \frac{\gamma}{\gamma -1}K (\varrho \theta)^{\gamma-1}\right),
	 \qquad
	 \psi_\revA{{\varrho E}} = K (\varrho \theta)^\gamma \frac{m}{\varrho}.
\end{equation}
Applying the general Tadmor EC condition~\eqref{eq:TadmorDef} to $U_{\varrho E} = \varrho E$, we obtain
\begin{equation}
    -\frac{1}{2}\jump{\frac{m^2}{\varrho^2}} f^\mathrm{num}_{\varrho} + \jump{\frac{m}{\varrho}} f^\mathrm{num}_{\varrho v} + \frac{\gamma}{\gamma -1} K \jump{(\varrho \theta)^{\gamma-1}} f^\mathrm{num}_{\varrho \theta} - K \jump{(\varrho \theta)^\gamma \frac{m}{\varrho}} = 0,
    \label{TEC}
\end{equation}
and we refer to~\eqref{TEC} as the total energy-conservative (TEC) condition.
To derive a TEC flux, we compute the jumps of the entropy variables in~\eqref{TEC} by recursively applying the discrete operators~\eqref{properties_discrete_mean} and \eqref{discrete_chain_rules} as follows:
\begin{align}
	\jump{\psi_\revA{{\varrho E}}} = & \jump{K (\varrho \theta)^\gamma \frac{m}{\varrho}} = \jump{K (\varrho \theta)^\gamma}\mean{\frac{m}{\varrho}} + \mean{K (\varrho \theta)^\gamma} \jump{\frac{m}{\varrho}} \nonumber \\ = &K\jump{ (\varrho \theta)^\gamma}\mean{\frac{m}{\varrho}} + K\mean{ (\varrho \theta)^\gamma}\left ( \jump{m}\mean{\frac{1}{\varrho}} - \frac{\mean{m}\jump{\varrho}}{ \mean{\varrho}_{\mathrm{geo}}^2 } \right ),
\end{align}
\begin{align}
	\jump{\omega_{\revA{\varrho E,1}}}  = & -\mean{\frac{m}{\varrho}} \left ( \jump{m}\mean{\frac{1}{\varrho}} - \frac{\mean{m}\jump{\varrho}}{ \mean{\varrho}_\mathrm{geo}^2 }  \right ), \\
	\jump{\omega_{\revA{\varrho E,2}}} =  & \jump{m}\mean{\frac{1}{\varrho}} - \frac{\mean{m}\jump{\varrho}}{ \mean{\varrho}_\mathrm{geo}^2 },                                            \\
	\jump{\omega_{\revA{\varrho E,3}}} =  & \frac{\gamma}{\gamma-1}K \jump{(\varrho \theta)^{\gamma-1}}.
\end{align}
By substituting all terms into the entropy conservation condition and requiring that each coefficient of the jumps vanishes, we obtain
\begin{equation}
\begin{gathered}
	\mean{\frac{m}{\varrho}} \mean{m}f^\mathrm{num}_\varrho   - \mean{m} f^\mathrm{num}_{\varrho  v}  + K\mean{ (\varrho \theta)^\gamma} \mean{m} = 0, \\
	f^\mathrm{num}_{\varrho v} = \mean{\frac{m}{\varrho}} f^\mathrm{num}_\varrho + K\mean{ (\varrho \theta)^\gamma}, \\
	\left ( \frac{\gamma}{\gamma-1} \jump{(\varrho \theta)^{\gamma-1}} \right ) f^\mathrm{num}_{\varrho \theta} =   \jump{ (\varrho \theta)^\gamma}\mean{\frac{m}{\varrho}}.
\end{gathered}
\end{equation}
By substituting the second into the first equation, one can verify that it is always satisfied, leading to a degree of freedom in the density flux.
Hence, we can state the following result:
\begin{theorem}
	A numerical flux for the compressible Euler equations in potential temperature formulation satisfying
	\begin{equation}
		f^{\mathrm{num}}_{\varrho v} = \mean{\frac{m}{\varrho}} f^{\mathrm{num}}_\varrho + \mean{p},
		\qquad
		f^{\mathrm{num}}_{\varrho \theta} = \mean{\varrho \theta}_\gamma \mean{\frac{m}{\varrho}},
	\end{equation}
	where $f^{\mathrm{num}}_{\varrho}$ is any consistent and symmetric discretization of $m = \varrho v$, is TEC.
\label{TECfluxTheorem}
\end{theorem}
\revB{We use the classical notations of consistency and symmetry of (two-point) numerical fluxes: $f^\mathrm{num}$ is symmetric if $f^\mathrm{num}(u_L, u_R) = f^\mathrm{num}(u_R, u_L)$ for all $u_L, u_R$; it is consistent if $f^\mathrm{num}(u,u) = f(u)$ for all $u$.}
We now turn our attention to the derivation of thermodynamic entropy-conservative numerical fluxes and, in particular, how to choose the degree of freedom given by the density flux.

\subsection{Conservation of the thermodynamic entropy}
In the development of EC fluxes for compressible Euler equations, a common approach involves deriving these fluxes so that $\varrho s$ is conserved at a semi-discrete level \cite{Ranocha2018,ECChandra,ismail2009affordable}.
This is the first step towards the development of ES fluxes \cite{WINTERS2017274,WARUSZEWSKI2022111507}.

In contrast to the previous section, where the total energy has been treated as a mathematical entropy, we focus on the thermodynamic entropy and derive numerical fluxes that conserve it at the semi-discrete level. Since the procedure is analogous to the previous section, we highlight only the main steps in the derivation. We start considering the set of variables $\vec{u} = (\varrho, \varrho v, \varrho \theta)$ and the entropy functional $U_{\varrho s}$.
The associated entropy variables are
\begin{equation}
	\vec{\omega}_{\varrho s}^T = \left(
		\log\bigl(\frac{K(\varrho \theta)^\gamma}{\varrho ^\gamma}\bigr) - \gamma,
		0,
		\gamma \frac{\varrho}{\varrho \theta}
	\right).
\end{equation}
 Following the same steps for the TEC flux, the potential flux for this case reduces to $\psi_{\revA{{\varrho s}}} = 0$. Consequently, we have that a numerical flux $f^\mathrm{num} = (f^\mathrm{num}_\varrho, f^\mathrm{num}_{\varrho v}, f^\mathrm{num}_{\varrho \theta})$ is EC if
   \begin{equation}
	\jump{\log \left(\frac{ \varrho}{\varrho \theta } \right)} f^\mathrm{num}_{\varrho} - \jump{ \frac{\varrho}{\varrho \theta}} f^\mathrm{num}_{\varrho \theta} = 0.
	\label{ECCondition}
\end{equation}
We refer to~\eqref{ECCondition} as the EC condition.
Recalling that in the numerical TEC fluxes, the numerical density flux constituted a degree of freedom,  we can employ this EC condition to derive two new numerical fluxes. The first is to use the density flux coming from \eqref{ECCondition} to obtain a flux which is TEC and EC, i.e.,
\begin{equation}
	f^{\mathrm{num}}_\varrho
	=
	\frac{\jump{\frac{\varrho}{\varrho \theta}}}{\jump{\log \left (\frac{\varrho }{\varrho \theta} \right )}} f^{\mathrm{num}}_{\varrho \theta}
	=
	\mean{1 / \theta}_{\mathrm{log}} f^{\mathrm{num}}_{\varrho \theta}.
\end{equation}
However, this leads to an influence of the pressure in the density flux, which may lead to loss of positivity of pressure values \cite{DERIGS2017624}.  The second choice is to use the potential temperature flux coming from~\eqref{ECCondition}, i.e.,
\begin{equation}
	f^{\mathrm{num}}_{\varrho \theta}
	=
	\frac{\jump{\log \left (\frac{\varrho }{\varrho \theta} \right )}}{\jump{\frac{\varrho}{\varrho \theta}}} f_{\varrho}^{\mathrm{num}}
	=
	\frac{1}{\mean{1 / \theta}_{\mathrm{log}}} f^{\mathrm{num}}_{\varrho}.
\end{equation}
This numerical flux is EC by construction and presents a degree of freedom in the density flux. In summary, we consider the numerical fluxes
\begin{itemize}
	\item \textbf{TEC Flux}
	      \begin{equation}
		      \begin{aligned}
			      f^\mathrm{num}_{\varrho} &= \text{any consistent and symmetric flux of $\varrho v$}, \\
			      f^\mathrm{num}_{\varrho v} &=  f^\mathrm{num}_{\varrho} \mean{v} + \mean{p},                                \\
			      f^\mathrm{num}_{\varrho \theta} &= \mean{\varrho \theta}_\gamma \mean{v},
		      \end{aligned}
		      \label{TECflux}
	      \end{equation}
	\item \textbf{EC Flux}
	      \begin{equation}
		      \begin{aligned}
			      f^\mathrm{num}_{\varrho} &= \text{any consistent and symmetric flux of $\varrho v$}, \\
			      f^\mathrm{num}_{\varrho v} &=  f^\mathrm{num}_{\varrho} \mean{v} + \mean{p},                                \\
			      f^\mathrm{num}_{\varrho \theta} &= f^\mathrm{num}_{\varrho} \frac{1}{\mean{1 / \theta}_{\mathrm{log}}}
		      \end{aligned}
		      \label{ECflux}
	      \end{equation}
	\item \textbf{ETEC Flux} (Entropy and Total Energy Conservative)
	      \begin{equation}
		      \centering
		      \begin{aligned}
			      f^\mathrm{num}_{\varrho} &=   f^\mathrm{num}_{\varrho \theta} \mean{1 / \theta}_{\mathrm{log}}, \\
			      f^\mathrm{num}_{\varrho v} &=  f^\mathrm{num}_{\varrho} \mean{v} + \mean{p},                          \\
			      f^\mathrm{num}_{\varrho \theta} &= \mean{\varrho \theta}_\gamma \mean{v}.
		      \end{aligned}
		      \label{ETECflux}
	      \end{equation}
\end{itemize}

The first two fluxes allow for some degree of freedom in the choice of the density flux, which can be used to enforce additional structure-preserving properties.
For sake of completeness, here we report the TEC flux in the $x$-direction for the 2D case with velocity $\vec{V} = (u,v)$:
	      \begin{equation}
		      \begin{aligned}
			      f^\mathrm{num,x}_{\varrho} &= \text{any consistent and symmetric flux of $\varrho u$}, \\
			      f^\mathrm{num,x}_{\varrho u} &=  f^\mathrm{num,x}_{\varrho} \mean{u} + \mean{p},                                \\
			      f^\mathrm{num,x}_{\varrho v} &=  f^\mathrm{num,x}_{\varrho} \mean{v},                                \\
			      f^\mathrm{num,x}_{\varrho \theta} &= \mean{\varrho \theta}_\gamma \mean{u}.
		      \end{aligned}
		      \label{TECflux2D}
	      \end{equation}
The other fluxes are generalized analogously to multiple dimensions.
In the next section, we further study and characterize the properties for all the fluxes.

%% file: section2.tex
\section{Two-point flux structure-preserving properties of $(\varrho, \varrho v, \varrho \theta)$ }
\label{sec:two_point_flux_properties}

In this section we keep our analysis to the one-dimensional case
and analyze in detail the different properties arising for the compressible Euler equations with the potential temperature and the derivation of our new fluxes \eqref{ECflux}, \eqref{TECflux}, and \eqref{ETECflux}.
First we recall the definition of fundamental and desirable properties of two-point numerical fluxes for the compressible Euler equations.

\begin{definition}[Kinetic energy preservation \cite{Jameson2008,Kuya2018,Ranocha2020,ranocha2018thesis,Ranocha2022}]
A numerical flux $f^\mathrm{num} = (f^{\mathrm{num}}_{\varrho}, f^{\mathrm{num}}_{\varrho v}, f^{\mathrm{num}}_{\varrho \theta})$ is KEP if
\begin{equation}
    f^{\mathrm{num}}_{\varrho v} = \mean{v} f^{\mathrm{num}}_{\varrho} + \mean{p}.
\label{kep}
\end{equation}
\end{definition}

The pressure equilibrium preserving (PEP) property, given the relation between pressure and potential temperature, has to be reformulated in terms of the new conserved variables. Following \cite{Ranocha2022}, we introduce the following definition \revA{of the PEP property in terms} of the potential temperature.
\begin{definition} A numerical flux $f^{\mathrm{num}} = (f^{\mathrm{num}}_\varrho, f^{\mathrm{num}}_{\varrho v}, f^{\mathrm{num}}_{\varrho \theta})$ is PEP if
\begin{equation}
   f^{\mathrm{num}}_{\revB{\varrho v}} = v f^{\mathrm{num}}_\varrho + \text{const}(p v),
   \qquad
   f^{\mathrm{num}}_{\varrho \theta} = \text{const}(\varrho \theta, v),
\label{PEPproperty}
\end{equation}
provided that velocity $v$ and pressure $p$ are constant throughout the domain.
\end{definition}

It is straightforward to prove that the TEC and ETEC numerical fluxes derived in the previous section are PEP.
The above definition is motivated by the following
\begin{lemma}
    Pressure equilibrium, i.e., $p \equiv \text{const}$ and $v \equiv \text{const}$ is preserved by the FV semi-discretization if and only if the numerical flux is PEP.
\end{lemma}
\begin{proof}
The pressure equation written in terms of the conserved variables $(\varrho, \varrho v, \varrho \theta)$ is
\begin{equation}
   \partial_t p =  \partial_t (K (\varrho \theta)^\gamma) = \gamma (\varrho \theta)^{\gamma - 1} K \partial_t (\varrho \theta).
\end{equation}
Recalling that $\partial_t(\varrho \theta) = - \partial_x (\varrho \theta v)$ and plugging into the previous equation,
\begin{equation}
    \partial_t(p) = - \gamma K(\varrho \theta)^{\gamma -1} \partial_x (\varrho \theta v).
    \label{eq_pressure_theta}
\end{equation}
A semi-discretization of \eqref{eq_pressure_theta} leads to
\begin{equation}
    \partial_t p_i = -\gamma K (\varrho \theta)_i^{\gamma -1} \frac{1}{\Delta x} \left (f^{\mathrm{num}}_{\varrho \theta}(\vec{u}_{i+1},\vec{u}_{i}) - f^{\mathrm{num}}_{\varrho \theta}(\vec{u}_{i},\vec{u}_{i-1})\right ).
\end{equation}
On the other hand, the semidiscrete evolution equation for the velocity is \cite{Ranocha2022}
\begin{multline}
    \varrho_i \partial_t v_i = \partial_t \varrho_i v_i \revB{- v_i \partial_t \varrho_i} = -\frac{1}{\Delta x} \Bigl (f^{\mathrm{num}}_{\varrho v}(\vec{u}_{i+1},\vec{u}_i) - f^{\mathrm{num}}_{\varrho v}(\vec{u}_{i},\vec{u}_{i-1}) \\- v_i  ( f^{\mathrm{num}}_{\varrho}(\vec{u}_{i+1},\vec{u}_i) - f^{\mathrm{num}}_{\varrho}(\vec{u}_{i},\vec{u}_{i-1}) )  \Bigr ).
\end{multline}
Thus, $\partial_t v_i = 0$ and $\partial_t p_i = 0$ if and only if \eqref{PEPproperty} is satisfied.
\end{proof}

\subsection{Characterization of fluxes}

Our first goal is to show that it is impossible to construct a numerical flux that is EC, TEC, KEP, PEP, and has no influence of the pressure in the density flux for the set of variables $(\varrho, \varrho v, \varrho \theta)$. In other words, when dealing with the potential temperature as primary invariant, the influence of the pressure term in the density flux cannot be avoided if TEC and EC are desired property for the numerical scheme, which can be formulated with the following
\begin{theorem}
\label{theorem1}
For the compressible Euler equations in potential temperature formulation \eqref{EulerPotentialTemperature1D}, the ETEC flux \eqref{ETECflux}
is EC, TEC, KEP and PEP. Moreover, it is the only numerical flux with these
properties for $v \equiv \text{const}$.
\end{theorem}

\begin{remark}
    Note that the numerical flux of Theorem~\ref{theorem1} has a density flux
    influenced by the pressure, which may cause the density to assume negative values \cite{DERIGS2017624}.
\end{remark}

To prove Theorem~\ref{theorem1}, we derive the necessary conditions for EC, TEC, and PEP.
\begin{lemma}
\label{TECconstantvp}
    For $p = \text{const}$, $v = \text{const}$, any consistent and symmetric numerical flux $f^\mathrm{num} = (f^\mathrm{num}_{\varrho}, f^\mathrm{num}_{\varrho v}, f^\mathrm{num}_{\varrho \theta} )$ is TEC and KEP.
\end{lemma}
\begin{proof}
If $p = \text{const}$, then $\varrho \theta = \text{const}$. Thus, the TEC condition becomes
\begin{equation}
    -\frac{1}{2}\jump{v^2}f^\mathrm{num}_\varrho + \jump{v}f^\mathrm{num}_{\varrho v} + \frac{\gamma K }{\gamma - 1} \jump{(\varrho \theta)^{\gamma -1}} f^\mathrm{num}_{\varrho \theta} - K\jump{(\varrho \theta)^\gamma v} = 0,
\end{equation}
and it is always satisfied for any numerical flux, given $p = \text{const}$ and $v = \text{const}$, since all the jumps vanish.
\end{proof}

\begin{lemma}
\label{ECconstantvp}
For $p = \text{const}$, $v = \text{const}$, an EC and PEP numerical flux $f^\mathrm{num}$ satisfies
\begin{equation}
    f^\mathrm{num}_{\varrho \theta} = f^\mathrm{num}_{\varrho} \mean{\varrho}_{\mathrm{log}}^{-1} \varrho \theta.
\end{equation}
\end{lemma}
\begin{proof}
    It follows directly from the EC condition \eqref{ECCondition}, when $p = \text{const}$.
\end{proof}

\begin{lemma}
\label{Lemma55}
    For $p = const$, $v = const$, an EC numerical flux that is also PEP or KEP must be of the form
\begin{equation}
 \left\{
\begin{aligned}
    &f^\mathrm{num}_{\varrho} = \mean{\varrho}_\mathrm{log} v\\
    &f^\mathrm{num}_{\varrho v} = f^\mathrm{num}_{\varrho} v + p\\
    &f^\mathrm{num}_{\varrho \theta} = \varrho \theta v
\end{aligned}
\right.
\label{formecpepkep}
\end{equation}
\end{lemma}
\begin{proof}
    Comparing the EC condition \eqref{ECCondition} with the PEP \eqref{PEPproperty} or KEP \eqref{kep}, the density flux $f^\mathrm{num}_{\varrho}$ is independent of the potential temperature, hence
\begin{equation}
    f^\mathrm{num}_{\varrho} = \mean{\varrho}_\mathrm{log} v,
\end{equation}
and inserting the EC and KEP property the numerical flux has to be of the form \eqref{formecpepkep}.
\end{proof}
Note that since for $p = \text{const}$ and $v = \text{const}$, TEC requires only a consistent and symmetric numerical flux, these numerical fluxes also satisfy the TEC condition \eqref{TEC}.

\begin{lemma}
     For $v = \text{const}$, an EC and PEP numerical flux for which the density flux does not depend on the pressure must be of the form
\begin{equation}
 \left\{
\begin{aligned}
    &f^\mathrm{num}_{\varrho} = \mean{\varrho}_\mathrm{log} v\\
    &f^\mathrm{num}_{\varrho v} = f^\mathrm{num}_{\varrho} v + \mean{p}\\
    &f^\mathrm{num}_{\varrho \theta} = \frac{\mean{\varrho}_\mathrm{log}}{\mean{1 / \theta}_\mathrm{log}} v
\end{aligned}
\right.
\end{equation}
\end{lemma}
\begin{proof}
   Due to Lemma~\ref{Lemma55}, the general form of the dependencies on $\varrho$ for $p \equiv \text{const}$ is already determined.  The remaining degree of freedom for non-constant pressure $p$ can be described by two functions $\varphi_{1,2}$, resulting in the numerical fluxes
\begin{equation}
 \left\{
\begin{aligned}
    &f^\mathrm{num}_{\varrho} = \mean{\varrho}_\mathrm{log} v\\
    &f^\mathrm{num}_{\varrho v} = f^\mathrm{num}_{\varrho} v + \varphi_1(\varrho_\pm, p_\pm)\\
    &f^\mathrm{num}_{\varrho \theta} = \varphi_2(\varrho_\pm, \varrho \theta_\pm) v
\end{aligned}
\right.
\label{fluxesphi}
\end{equation}
where $\varphi_1(\varrho_\pm, p_\pm)$ and $\varphi_2(\varrho_\pm, \varrho \theta_\pm)$ are some kind of mean values depending on $\varrho_\pm$, $p_\pm$, $\varrho \theta_\pm$ such that
\begin{equation}
\begin{aligned}
    \forall \varrho_\pm, p > 0&\colon&  \varphi_{1}(\varrho_+, \varrho_-, p, p) &= p, \\
    \forall \varrho_\pm, \varrho \theta > 0&\colon&  \varphi_{2}(\varrho_+, \varrho_-, \varrho \theta, \varrho \theta) &= \varrho \theta.
\end{aligned}
\label{conditionphi}
\end{equation}
\revB{The function $\varphi_1$ is not constrained by the EC condition~\eqref{ECCondition}. Thus, we adopt the form given by the} KEP condition, $ \varphi_1(\varrho_\pm, p)= \mean{p}$. Substituting the fluxes \eqref{fluxesphi} into the EC condition \eqref{ECCondition},
\begin{equation}
    \varphi_2(\varrho_\pm, \varrho \theta_\pm) = \frac{\mean{\varrho}_\mathrm{log}}{\mean{1 / \theta}_\mathrm{log}},
\end{equation}
which satisfies the condition \eqref{conditionphi}.
\end{proof}
\begin{lemma}
\label{lemma57}
    For $v = \text{const}$, a TEC flux for the potential temperature must be of the form
\begin{equation}
    f_{\varrho \theta}^{\mathrm{num}} = \mean{\varrho \theta}_\gamma v.
    \label{formTECv}
\end{equation}
\end{lemma}
\begin{proof}
The potential temperature numerical flux can be described for constant velocities as
\begin{equation}
    f_{\varrho \theta}^{\mathrm{num}} = \varphi_2(\varrho_\pm, \varrho \theta_\pm) v,
\end{equation}
where $\varphi_2$ is some kind of mean values depending on $\varrho_\pm$, $\varrho \theta_\pm$ such that
\begin{equation}
  \forall \varrho_\pm, \varrho \theta > 0\colon \qquad  \varphi_{2}(\varrho_+, \varrho_-, \varrho \theta, \varrho \theta) = \varrho \theta.
\end{equation}
    The TEC condition \eqref{TEC} for $v = \text{const}$ becomes
\begin{equation}
    \frac{\gamma K }{\gamma - 1} \jump{(\varrho \theta)^{\gamma -1}} f_{\varrho \theta} - K\jump{(\varrho \theta)^\gamma} v = 0,
\end{equation}
which results in the final form \eqref{formTECv}.
\end{proof}

Thus, for $v = \text{const}$ a numerical flux that is PEP, EC, TEC, and has a density flux that does not depend on the pressure cannot be constructed, since in the relation
\begin{equation}
    \mean{\varrho \theta}_\gamma = \frac{\mean{\varrho}_\mathrm{log}}{\mean{1/\theta}_\mathrm{log}},
\end{equation}
for a fixed $\gamma$, there exists at least one tuple $(\varrho_\pm, \theta_\pm)$ such that the relation above is not satisfied.
We showed that a numerical flux that is EC, PEP, TEC, and does not have the influence of the pressure in the density flux cannot exist for $v = \text{const}$. Thus, we are ready to introduce the dependence of the pressure in the density flux.

\begin{lemma}
\label{lemma58}
    For $v = \text{const}$, a numerical flux that is PEP, EC, and TEC must be of the form
\begin{equation}
 \left\{
\begin{aligned}
    &f^\mathrm{num}_{\varrho} =  \mean{\varrho}_\gamma \mean{1 / \theta}_{\mathrm{log}}v\\
    &f^\mathrm{num}_{\varrho v} = f^{\mathrm{num}}_{\varrho} v + \mean{p}\\
    & f^\mathrm{num}_{\varrho \theta} = \mean{\varrho}_\gamma v
\end{aligned}
\right.
\label{tecformv}
\end{equation}
with the influence of the pressure in the numerical density flux.
\end{lemma}
\begin{proof}
   The form of the potential temperature flux has already been determined by Lemma~\ref{lemma57}. The remaining dependencies in the density and momentum flux can be again described by the functions $\varphi_{1,2}$  resulting in the numerical fluxes
\begin{equation}
 \left\{
\begin{aligned}
    &f^\mathrm{num}_{\varrho} = \varphi_{2}(\varrho_\pm, \varrho \theta_\pm) v\\
    &f^\mathrm{num}_{\varrho v} = f^\mathrm{num}_{\varrho} v + \varphi_1(\varrho_\pm, p_\pm)\\
    &f^\mathrm{num}_{\varrho \theta} = \frac{\gamma}{\gamma-1} \frac{\jump{(\varrho \theta)^\gamma}}{\jump{(\varrho \theta)^{\gamma-1}}}v
\end{aligned}
\right.
\label{fluxesphi_tec}
\end{equation}
where $\varphi_1(\varrho_\pm, p_\pm)$ and $\varphi_2(\varrho_\pm, \varrho \theta_\pm)$ are some kind of mean values depending on $\varrho_\pm$, $p_\pm$, $\varrho \theta_\pm$ such that
\begin{equation}
\begin{aligned}
    \forall \varrho_\pm, p > 0&\colon& \varphi_{1}(\varrho_+, \varrho_-, p, p) &= p, \\
    \forall \varrho_\pm, \varrho \theta > 0&\colon&  \varphi_{2}(\varrho_+, \varrho_-, \varrho \theta, \varrho \theta) &= \varrho.
\end{aligned}
\label{conditionphi2}
\end{equation}
The final form \eqref{tecformv} is given by the substituting the numerical fluxes \eqref{fluxesphi_tec} into the conditions EC \eqref{ECCondition} and PEP \eqref{PEPproperty}.
\end{proof}

We can now prove Theorem~\ref{theorem1}.
\begin{proof}[Proof of Theorem~\ref{theorem1}]
The EC \eqref{ECCondition} and TEC \eqref{TEC} conditions are verified by substituting the numerical fluxes into \eqref{ECCondition} and \eqref{TEC}. The KEP \eqref{kep} property is verified by construction and the characterization for $v = \text{const}$ has been proven in Lemma~\ref{lemma58}.
\end{proof}

The characterization of the numerical fluxes for the compressible Euler equations with potential temperature as first invariant has been completed. The numerical fluxes are characterized by the influence of the pressure in the density flux, which is necessary to satisfy the EC, TEC, and PEP properties. In the following lemma we show the characterization for $ p = \text{const}$.
\begin{lemma}
    For fixed $p = \text{const}$, a (symmetric) EC, TEC, KEP, and PEP numerical flux must be of the form
    \begin{equation}
 \left\{
\begin{aligned}
    &f^\mathrm{num}_{\varrho} =  \mean{\varrho}_\gamma \mean{v} + \chi (\varrho_\pm, v_\pm) \\
    &f^\mathrm{num}_{\varrho v} = f^\mathrm{num}_{\varrho} \mean{v} + \mean{p} + \mean{v}\chi (\varrho_\pm, v_\pm)\\
    &f^\mathrm{num}_{\varrho \theta} = \varrho \theta \mean{v} + \frac{\varrho \theta}{\mean{\varrho}_\mathrm{log}}\chi (\varrho_\pm, v_\pm)
\end{aligned}
\right.
\end{equation}
where $\chi$ is a function depending on $\varrho_\pm, v_\pm$ such that $\forall \varrho_\pm, v:$ $\chi(\varrho_+, \varrho_-, v, v) = 0$.
\end{lemma}
\begin{proof}
    The proof is analogous to the one in \cite{Ranocha2022}. Indeed, due to consistency, a numerical flux can always be written as the sum of a given numerical flux and a perturbation $\chi$ that is consistent with zero. Therefore, we add the perturbation $\chi$ to the numerical fluxes from Theorem \ref{theorem1}, considering that $p = \text{const}$, so that
\begin{equation} \left\{
\begin{aligned}
    &f^\mathrm{num}_{\varrho} =  \mean{\varrho}_\gamma \mean{v} + \chi_\varrho (\varrho_\pm, v_\pm)\\
    &f^\mathrm{num}_{\varrho v} = f^\mathrm{num}_{\varrho} \mean{v} + \mean{p} + \chi_{\varrho v} (\varrho_\pm, v_\pm)\\
    &f^\mathrm{num}_{\varrho \theta} = \varrho \theta \mean{v} + \chi_{\varrho \theta} (\varrho_\pm, v_\pm)
\end{aligned}
\right.
\label{ETECperturbed}
\end{equation}
where $\forall \varrho, v\colon \chi_\varrho(\varrho_+, \varrho_-, v, v) = 0$, $\chi_{\varrho v}(\varrho_+, \varrho_-, v, v) = 0$ and $\chi_{\varrho \theta}(\varrho_+, \varrho_-, v, v) = 0$. The perturbation $\chi$ is then chosen such that the EC, TEC, KEP, and PEP conditions are satisfied. The KEP property leads to the condition $\chi_{\varrho v} = \mean{v} \chi_\varrho$. The TEC condition~\eqref{TEC} for $p = \text{const}$ yields
\begin{multline}
        -\frac{1}{2}\jump{\frac{m^2}{\varrho^2}} \left( \frac{\gamma - 1}{\gamma}\frac{\jump{\varrho^\gamma}}{\jump{\varrho^{\gamma - 1}}}  \mean{v} + \chi_\varrho (\varrho_\pm, v_\pm) \right) \\
        + \jump{\frac{m}{\varrho}} \left ( \frac{\gamma - 1}{\gamma}\frac{\jump{\varrho^\gamma}}{\jump{\varrho^{\gamma - 1}}}  \mean{v}^2 + p + \chi_\varrho (\varrho_\pm, v_\pm) \mean{v}  \right ) - p \jump{v}= 0,
\end{multline}
    which is always satisfied, for any perturbation $\chi_\varrho$. The EC condition~\eqref{ECCondition} leads to
\begin{equation}
   \chi_{\varrho \theta} = \frac{\varrho \theta}{\mean{\varrho}_\mathrm{log}} \chi_{\varrho},
\end{equation}
which is satisfied for any perturbation $\chi_\varrho$ that satisfies the condition $\forall \varrho, v\colon \chi_\varrho(\varrho_+, \varrho_-, v, v) = 0$. The PEP is satisfied if and only if the perturbation term is such that $\forall \varrho, v: \chi_\varrho(\varrho_+, \varrho_-, v, v) = 0$.
\end{proof}
The extension of the numerical fluxes~\eqref{ETECperturbed} to an arbitrary non-constant pressure is straightforward and due to the form of the fluxes in Theorem~\ref{theorem1}, the numerical flux is unique also for general velocities.

\subsection{On the pressure positivity for $(\varrho, \varrho v, \varrho \theta)$}
Density and pressure must be positive to avoid nonphysical numerical solutions and eventually instabilities. The numerical fluxes should preserve this property when an explicit Euler time marching method is applied to a first-order FV method. This is a well-known and important property, since SSP Runge-Kutta methods can be written as a convex combination of Euler steps method \cite{Gottliebs2011}. Therefore, given the convex set of admissible states
\begin{equation}
    G = \left \{ \vec{u} = \left (\varrho, \varrho v, \varrho \theta \right ) | \varrho > 0, p = K(\varrho \theta)^\gamma > 0 \right \},
\end{equation}
positivity of density and pressure is inherited by any convex combination of Euler steps, as soon as it is preserved by the internal Euler steps.
Note that the density positivity remains unchanged with respect to the results already described in \cite{Ranocha2018}. However, the advantage of the potential temperature formulation results in a straightforward way to preserve the pressure positivity.
\begin{corollary}
    If the numerical potential temperature flux $f^\mathrm{num}_{\varrho \theta} = \mean{\varrho \theta}_\gamma \mean{v} - \frac{\lambda}{2}\jump{\varrho \theta}$ is used with $\lambda \geq \max \{|v_i|, |v_{i+1}| \}$, the first-order FV scheme preserves the non-negativity of the pressure \revB{$p$} under the CFL condition
\begin{equation}
    \Delta t \leq \frac{\Delta x}{2 \lambda}.
\end{equation}
\end{corollary}
\begin{proof}
    The mean appearing in the flux is given by the integral for $t = 2- \gamma$ (see \cite{MeansGenerated})
 \begin{equation}
            f(t) = \frac{\int_{\revB{\varrho \theta_-}}^{\revB{\varrho \theta_+}} x^{t+1} dx}{\int_{\revB{\varrho \theta_-}}^{\revB{\varrho \theta_+}}x^t dx}.
\end{equation}
 Thus,
\begin{equation}
   f(2-\gamma)  =  \mean{\varrho \theta}_\gamma,
        \end{equation}
and due to the monotonicity of $f$, we have
 \begin{equation}
    f(2-\gamma)  = \mean{\varrho \theta}_\gamma \leq f(0) = \mean{\varrho \theta}.
\end{equation}
Hence, this mean satisfies the conditions of Theorem~6.1 in \cite{Ranocha2018} and the flux preserves the positivity of the potential temperature. Pressure positivity follows, since
\begin{equation}
    p_i^{n+1} = K (\varrho \theta)_i^{n+1} \geq 0 \hspace{1cm} \forall n \in \mathbb{N}.
    \qedhere
\end{equation}
\end{proof}

In Table~\ref{summary_fluxes} we provide a summary overview of the properties of the derived fluxes. It is important to emphasize that the PEP property can only be achieved for an EC flux if the density flux is constructed with logarithmic mean, thereby reducing the degree of freedom in its derivation. On the other hand, a TEC flux does not present this limitation.
\begin{table}[h]
    \centering
    \renewcommand{\arraystretch}{1.3}
    \setlength{\tabcolsep}{10pt}
    \caption{Properties and conserved variables for the different derived numerical fluxes. $\checkmark$: variable conserved or property satisfied. $\times$: variable not conserved or property is not satisfied. $\overline{\varrho} = \mean{\varrho}_\mathrm{log}$: property satisfied only if the density in the $f_{\varrho}$ flux is discretized with a logarithmic mean.}
    \begin{tabular}{l|ccc}
        \toprule
        \textbf{Conserved Variables and Properties} & \multicolumn{3}{c}{\textbf{Numerical Fluxes}} \\
        \cmidrule(lr){2-4}
        & \textbf{EC} \eqref{ECflux} & \textbf{TEC} \eqref{TECflux} & \textbf{ETEC} \eqref{ETECflux} \\
        \midrule
        $\varrho E$  & $\times$ & $\checkmark$ & $\checkmark$ \\
        $\varrho s$  & $\checkmark$  & $\times$ & $\checkmark$ \\
        $\varrho \theta$  & $\checkmark$  & $\checkmark$ & $\checkmark$ \\
        KEP & $\checkmark$  & $\checkmark$ & $\checkmark$ \\
        PEP & $ \overline{\varrho} = \mean{\varrho}_\mathrm{log}$  & $\checkmark$ & $\checkmark$ \\
        \textbf{No Pressure Term in the Density Flux} $f_\revB{\varrho}$ & $\checkmark$ & $\checkmark$ & $\times$ \\
        \bottomrule
    \end{tabular}
    \label{summary_fluxes}
\end{table}

%% file: section3.tex
\section{On the conservation properties in presence of a geopotential term}
\label{sec:additional_conservation_properties}
To include gravity effects, typically a source term with a geopotential term $\vec{\phi} = \vec{\phi}(\vec{x})$ is added to the right-hand side of the compressible Euler equations, as in Eqs~\eqref{EulerTotalEnergy} and~\eqref{EulerPotentialTemperature}.
The momentum and the total energy are not conserved anymore and other invariants and properties play a significant role in geophysical flows.
Souza et al.\ \cite{Souza} have developed a kinetic and potential energy preserving (KPEP) numerical flux with the total energy as conserved variable using the Kennedy-Gruber flux~\cite{KENNEDY20081676} for the conservative part of the Euler equations.
In this section, we extend the KPEP property introduced in~\cite{Souza}. We first present a general condition for KPEP fluxes within the FV framework. Furthermore, we introduce the condition for total energy conservation in the presence of gravity, considering two sets of conserved variables: $(\varrho, \varrho v, \varrho E)$ and $(\varrho, \varrho v, \varrho \theta)$. To do that, we employ a non-conservative discretization of the source terms.

\subsection{On the kinetic and total energy}
In presence of gravity source terms given by a generic geopotential $\phi$, the compressible Euler equations in 1D read
\begin{equation}
    \partial_t \begin{pmatrix}
           \varrho \\
           \varrho v \\
           \varrho E
         \end{pmatrix} + \partial_x \begin{pmatrix}
           \varrho v \\
           \varrho v^2 + p \\
           (\varrho E + p ) v
         \end{pmatrix}  =  \begin{pmatrix}
           0 \\
           - \varrho \partial_x \phi \\
           - \varrho v \partial_x \phi \\

         \end{pmatrix}
         \label{eq2}
\end{equation}
for $(\varrho, \varrho v, \varrho E)$ as conserved variables; for the potential temperature formulation, we have
\begin{equation}
    \partial_t \begin{pmatrix}
           \varrho \\
           \varrho v \\
           \varrho \theta
         \end{pmatrix} + \partial_x \begin{pmatrix}
           \varrho v \\
           \varrho v^2 + p \\
           \varrho \theta v
         \end{pmatrix}  =  \begin{pmatrix}
           0 \\
            -\varrho \partial_x \phi \\
           0 \\

         \end{pmatrix}.
         \label{eq3}
\end{equation}
The FV method in presence of a source term is typically written as
\begin{equation}
 	\partial_t \vec{u}_i + \frac{1}{\Delta x} \Bigg ( \vec{f}^\mathrm{num}(\vec{u}_i,\vec{u}_{i+1}) - \vec{f}^\mathrm{num}(\vec{u}_{i-1},\vec{u}_{i}) \Bigg ) = \revA{\vec{\mathcal{S}}}_i.
    \label{FVmethod1Dsource}
\end{equation}
We seek for a particular discretization of the source term $\vec{S}_i$ and conditions such that for both semi-discretization the KPEP and TEC property are satisfied, where the total energy for the systems~\eqref{eq2} and~\eqref{eq3} is $U = \varrho E + \varrho \phi$, which is a conserved quantity. In particular, following the work of \cite{FJORDHOLM20115587,Ranocha2017}, we introduce the general form for the discretization of source terms
\begin{equation}
    \revA{\mathcal{S}}_i^{\varrho v} =\frac{S_{i+1/2}^{\varrho v} + S_{i-1/2}^{\varrho v}}{2 \Delta x} = \frac{\overline{\varrho}_{i+1/2}\jump{\phi}_{i+1/2} + \overline{\varrho}_{i-1/2} \jump{\phi}_{i-1/2}}{2 \Delta x},
\end{equation}
\begin{equation}
    \revA{\mathcal{S}}_i^{\varrho E} =\frac{S_{i+1/2}^{\varrho E} + S_{i-1/2}^{\varrho E}}{2 \Delta x}  = \frac{\overline{v}_{i+1/2}\overline{\varrho}_{i+1/2}\jump{\phi}_{i+1/2} + \overline{v}_{i-1/2}\overline{\varrho}_{i-1/2} \jump{\phi}_{i-1/2}}{2 \Delta x},
\end{equation}
where $\overline{\varrho}_{i+1/2}$ and $\overline{v}_{i+1/2}$ are consistent mean values, such that
\begin{equation}
\begin{aligned}
    \forall \varrho_{i+1} = \varrho = \varrho_{i} > 0&\colon&  \overline{\varrho}_{i+1/2} &= \varrho, \\
    \forall v_{i+1} = v = v_{i}&\colon&  \overline{v}_{i+1/2} &= v.
\end{aligned}
\label{conditionsource}
\end{equation}
We therefore investigate the necessary conditions for ensuring KPEP and TEC property.
\begin{lemma}
\label{lemmathetatpec}
Given the set of variables $(\varrho, \varrho v, \varrho \theta)$, a numerical flux $(f^\mathrm{num}_\varrho, f^\mathrm{num}_{\varrho v}, f^\mathrm{num}_{\varrho \theta})$, the source term $S^{\varrho v}$ and the corresponding FV method \eqref{FVmethod1D} are TEC if
\begin{enumerate}[label=(\roman*)]
    \item the density flux and the source term are of the form
    \begin{equation}
  f^\mathrm{num}_\varrho = \overline{\varrho} \mean{v}, \quad S^{\varrho v}   = \overline{\varrho} \jump{\phi}
\end{equation}
    \item and the numerical flux $(f^\mathrm{num}_\varrho, f^\mathrm{num}_{\varrho v}, f^\mathrm{num}_{\varrho \theta})$ is of the form~\eqref{TECflux}.
\end{enumerate}
\end{lemma}
\begin{proof}
    The proof requires the contraction of the right-hand side in entropy space. Therefore, we consider the total energy
\begin{equation}
   U =  \varrho E + \varrho \phi = \frac{p}{\gamma-1} + \frac{1}{2} \frac{m^2}{\varrho} + \varrho \phi.
\label{TECP}
\end{equation}
The entropy variables for the new entropy functional $U$ can be written as
\begin{equation}
    \vec{\omega} = \vec{\omega}_{\varrho E} + \vec{\omega}_{\phi}= \begin{pmatrix}
        -\frac{1}{2}\frac{m^2}{\varrho^2}\\
        \frac{m}{\varrho}\\
        \frac{\gamma}{\gamma-1}K (\varrho \theta)^{\gamma-1}
    \end{pmatrix} + \begin{pmatrix}
        \phi\\
        0\\
        0\\
    \end{pmatrix}.
    \label{entropytecphivariables}
\end{equation}
Hence, we can use this linear relation between the total energy including the potential.
    Therefore, performing a left multiplication of \eqref{FVmethod1Dsource} by the entropy variables \eqref{entropytecphivariables} yields
    \begin{equation}
    \frac{d U_i}{dt} = -\frac{1}{\Delta x}\Bigg ( \langle \vec{\omega}_i, \vec{f}_{i+1/2} \rangle - \langle \vec{\omega}_i, \vec{f}_{i-1/2} \rangle \Bigg) - \langle \vec{\omega}_i , \revA{\vec{\mathcal{S}}}_i \rangle,
\end{equation}
where we introduced the abbreviation $\vec{f}_{i+1/2} = f^\mathrm{num}(\vec{u}_i, \vec{u}_{i+1})$.
By expressing the first two terms as $\vec{\omega}_i = \mean{\vec{\omega}}_{i \pm 1/2} \mp \frac{1}{2}\jump{\vec{\omega}}_{i \pm 1/2}$, we obtain
\begin{align}
   \frac{d U_i}{dt}  = & -\frac{1}{\Delta x} \left ( \langle \mean{\vec{\omega}}_{i+1/2},\vec{f}_{i+1/2} \rangle - \frac{1}{2}\langle \jump{\vec{\omega}}_{i+1/2},\vec{f}_{i+1/2}\rangle \right  )  \nonumber \\ &-\frac{1}{\Delta x} \left ( \langle \mean{\vec{\omega}}_{i-1/2}, \vec{f}_{i-1/2}\rangle + \frac{1}{2} \langle \jump{\vec{\omega}}_{i-1/2}, \vec{f}_{i-1/2} \rangle\right ) - \langle \vec{\omega}_i, \revA{\vec{\mathcal{S}}}_i \rangle.
\end{align}
Therefore, we can split the different contributions and employing the TEC definition and the definition of the source term, we are left with
\begin{align}
    &\frac{d U_i}{dt} =-\frac{1}{\Delta x}(F_{i+1/2} - F_{i-1/2}) -\frac{\phi_i}{\Delta x} (f_{i+1/2}^{\varrho} - f_{i-1/2}^{\varrho} ) - v_i \frac{\overline{\varrho}_{i + 1/2} \jump{\phi}_{i+1/2} + \overline{\varrho}_{i - 1/2} \jump{\phi}_{i-1/2}}{2 \Delta x},
\end{align}
where $F_{i+1/2}$ is the conservative and symmetric total energy flux. Following a similar argument as before,
\begin{align}
    \frac{d U_i}{dt} = &- \frac{1}{\Delta x} (F_{i+1/2} - F_{i-1/2}) -\frac{1}{\Delta x}(f_{i+1/2}^{\varrho} \mean{\phi}_{i+1/2}- f_{i-1/2}^{\varrho}\mean{\phi}_{i-1/2}) + \nonumber \\
    &+ \frac{\overline{\varrho}_{i + 1/2} \jump{\phi}_{i+1/2} \jump{v}_{i+1/2}- \overline{\varrho}_{i - 1/2} \jump{\phi}_{i-1/2} \jump{\revA{v}}_{i-1/2}}{4 \Delta x} + \nonumber \\
    & +\frac{\jump{\phi}_{i+1/2}f^{\varrho}_{i+1/2} + \jump{\phi}_{i-1/2}f^{\varrho}_{i-1/2}}{2\Delta x} - \frac{\overline{\varrho}_{i+1/2}\jump{\phi}_{i+1/2}\mean{v}_{i+1/2} + \overline{\varrho}_{i-1/2}\jump{\phi}_{i-1/2}\mean{v}_{i-1/2}}{2 \Delta x}.
\end{align}
If $f_{i+1/2}^\varrho = \overline{\varrho}_{i+1/2} \mean{v}_{i+1/2}$ then the last two terms vanish
\begin{align}
    \frac{d U_i}{dt} = &- \frac{1}{\Delta x} (F_{i+1/2} - F_{i-1/2}) -\frac{1}{\Delta x}(f_{i+1/2}^{\varrho} \mean{\phi}_{i+1/2}- f_{i-1/2}^{\varrho}\mean{\phi}_{i-1/2}) \nonumber \\
    &+ \frac{\overline{\varrho}_{i + 1/2} \jump{\phi}_{i+1/2} \jump{v}_{i+1/2}- \overline{\varrho}_{i - 1/2} \jump{\phi}_{i-1/2} \jump{v}_{i-1/2}}{4 \Delta x},
\end{align}
which can be written in a conservative form as
 \begin{equation}
     \frac{d U_i}{ dt} = - \frac{1}{\Delta x} (H_{i+1/2} - H_{i-1/2}),
 \end{equation}
 where
 \begin{equation}
     H_{i+1/2} = F_{i+1/2} + f_{i+1/2}^\varrho \mean{\phi}_{i+1/2} - \frac{\overline{\varrho}_{i+1/2}}{4} \jump{\phi}_{i+1/2} \jump{v}_{i+1/2}.
    \qedhere
 \end{equation}
\end{proof}

Note that for the set of conserved variables with the potential temperature we had to impose the TEC flux derived in the previous section. That is not required when working with $(\varrho, \varrho v, \varrho E)$ as primary variables, since $\varrho E$ is a first invariant when $\phi = 0$.
\begin{lemma}

    Given the set of conserved variables $(\varrho, \varrho v, \varrho E)$, a numerical flux $(f^\mathrm{num}_\varrho, f^\mathrm{num}_{\varrho v}, f^\mathrm{num}_{\varrho E})$, the source terms $S^{\varrho v}$, $S^{\varrho E}$ and the corresponding FV method \eqref{FVmethod1D} are TEC if the density flux and the source terms are of the form
    \begin{equation}
        f^{\mathrm{num}}_\varrho = \overline{\varrho} \mean{v}, \quad
        S^{\varrho v} = \overline{\varrho} \jump{\phi},
        \quad
        S^{\varrho E} = f^\mathrm{num}_{\varrho} \jump{\phi}.
    \end{equation}
\end{lemma}
\begin{proof}
    A semi-discretization of the total energy starting with the given set of conserved variables can be written as
    \begin{equation}
        \partial_t (\varrho_i E_i + \phi_i \varrho_i) = \partial_t ( \varrho_i E_i) + \phi_i \partial_t (\varrho_i),
    \end{equation}
    since the geopotential depends only on the space coordinates. Following the same steps as shown before, and considering $f^\varrho_{i+1/2} = \overline{v}_{i+1/2} \overline{\varrho}_{i+1/2}$ we obtain
   \begin{equation}
    \begin{aligned}
        \frac{d ( \varrho_i E_i + \phi_i \varrho_i)}{dt} = & - \frac{1}{\Delta x} (f_{i+1/2}^{\varrho E} - f_{i-1/2}^{\varrho E}) - \frac{\phi_i}{\Delta x} (f_{i+1/2}^\varrho - f_{i-1/2}^{\varrho}) \\
        & - \frac{ \overline{v}_{i+1/2}\overline{\varrho}_{i + 1/2} \jump{\phi}_{i+1/2} + \overline{v}_{i-1/2}\overline{\varrho}_{i - 1/2} \jump{\phi}_{i-1/2}}{2 \Delta x},
    \end{aligned}
\end{equation}
    which can be further developed into
    \begin{multline}
        \frac{d ( \varrho_i E_i + \phi_i \varrho_i)}{dt}
        =
        - \frac{1}{\Delta x} (f_{i+1/2}^{\varrho E} - f_{i-1/2}^{\varrho E}) -\frac{1}{\Delta x}(f_{i+1/2}^{\varrho} \mean{\phi}_{i+1/2}- f_{i-1/2}^{\varrho}\mean{\phi}_{i-1/2})
        \\
        +\frac{\jump{\phi}_{i+1/2}f^{\varrho}_{i+1/2} + \jump{\phi}_{i-1/2}f^{\varrho}_{i-1/2}}{2\Delta x}  - \frac{ \overline{v}_{i+1/2}\overline{\varrho}_{i + 1/2} \jump{\phi}_{i+1/2} + \overline{v}_{i-1/2}\overline{\varrho}_{i - 1/2} \jump{\phi}_{i-1/2}}{2 \Delta x}.
    \end{multline}
    The last two terms vanishes if $ f_{i+1/2}^\varrho= \overline{v}_{i+1/2} \overline{\varrho}_{i+1/2} $; its conservative form is
    \begin{equation}
        \frac{d}{dt} ( \varrho_i E_i + \phi_i \varrho_i) = -\frac{1}{\Delta x} (H_{i+1/2} - H_{i-1/2}),
    \end{equation}
    where $H$ is the numerical flux of the total energy
    \begin{equation}
        H_{i+1/2} = f_{i+1/2}^{\varrho E} + f_{i+1/2}^{\varrho} \mean{\phi}_{i+1/2}.
        \qedhere
    \end{equation}
\end{proof}

Following the same approach as above, we investigate the condition for the numerical fluxes to achieve kinetic and potential energy preserving (KPEP) property. Since the momentum equation and the KEP property is not influenced by the set of variables, we can formulate a single lemma, that accounts for both.
\begin{lemma}
   Given the set of conserved variables $(\varrho, \varrho v, \varrho E)$ or $(\varrho, \varrho v, \varrho \theta)$, a numerical flux $(f^\mathrm{num}_{\varrho}, f^\mathrm{num}_{\varrho v}, f^\mathrm{num}_{\varrho E})$ or $(f^\mathrm{num}_{\varrho}, f^\mathrm{num}_{\varrho v}, f^\mathrm{num}_{\varrho \theta})$, a source term $S^{\varrho v}$ and the corresponding FV method \eqref{FVmethod1D} are KPEP if
\begin{enumerate}[label=(\roman*)]
    \item the density flux and the source term are of the form
    \begin{equation}
  f^\mathrm{num}_\varrho = \overline{\varrho} \mean{v}, \quad S^{\varrho v}   = \overline{\varrho} \jump{\phi}
\end{equation}
    \item and the numerical flux is KEP.
\end{enumerate}
\end{lemma}
\begin{proof}
    First we note that
    \begin{equation}
    \partial_t \left (\frac{1}{2}k + \varrho \phi \right ) = v \partial_t (\varrho v) - \frac{v^2}{2}\partial_t \varrho + \phi \partial_t \varrho = - \partial_x \left ( v \left( \frac{1}{2}\varrho v^2 + p + \varrho \phi  \right) \right ) + p \partial_x v.
\end{equation}
A semi-discretization leads to
\begin{equation}
    \partial_t \left ( \frac{1}{2}\varrho_i v_i^2 + \varrho_i \phi_i \right ) = - v_i \frac{f^{\varrho v}_{i+1/2} - f^{\varrho v}_{i-1/2}}{\Delta x} + \frac{v_i^2}{2} \frac{f^{\varrho}_{i+1/2} - f^{\varrho}_{i-1/2}}{\Delta x} - v_i S^{\varrho v}_i -\phi_i \frac{f^{\varrho}_{i+1/2} - f^{\varrho}_{i-1/2}}{\Delta x}.
\end{equation}
If $f^\mathrm{num}_{\varrho v}$ is KEP, this proof reduces to finding a condition that results in a conservative form for the last two terms. This has already been done in the previous lemmas, leading to the same condition.
\end{proof}

\subsection{Well-balanced schemes}
Steady states are particular solutions of the compressible Euler equations. In this section, we show how to construct well-balanced schemes by employing the same non-conservative product discretization as in the previous section. In particular, in presence of a potential the scheme should be able to mimic at a discrete level the hydrostatic balance of the Euler equations, i.e.,
\begin{equation}
    \nabla p = -\varrho \nabla \phi.
\end{equation}
The steady background state is typically prescribed either with an isothermal or constant potential temperature state, where then perturbations are added. Waruszewski et al.\ \cite{WARUSZEWSKI2022111507} derived a well-balanced scheme with isothermal background state for DG methods using a generalization of flux differencing for numerical fluxes in fluctuation form.
Based on their results, here we show that this also applies for our FV methods and we present a well-balanced scheme for a constant potential temperature background state.

The only equation which plays a role in the well-balanced scheme is the momentum. Thus, independently of the last closure equation, ($\varrho E$ or $\varrho \theta$, or yet more general), these results are satisfied as long as the momentum $\varrho v$ is employed as primary invariant.
\begin{lemma}
    A source term $S^{\varrho v}$ preserves the hydrostatic balance prescribed by a constant background temperature $T$ of the compressible Euler equations if the source term is of the form
    \begin{equation}
 S^{\varrho v} = \mean{\varrho}_\mathrm{log} \jump{\phi},
 \label{nonconservativeT}
\end{equation}
and the pressure term is discretized as $\mean{p}$.
\end{lemma}
\begin{proof}
Presenting the idea in 1D, we consider
\begin{equation}
    \partial_x p = - \varrho \partial_x \phi.
    \label{balanceeq1d}
\end{equation}
A solution of \eqref{balanceeq1d} must be of the form
\begin{align}
    p = p_0 e^{-\frac{\phi}{RT}}
    \quad\text{and}\quad
    \varrho = \varrho_0 e ^ {-\frac{{\phi}}{RT}} = \frac{p_0}{RT} e ^ {-\frac{{\phi}}{RT}},
    \label{IsothermalSolution}
\end{align}
where $p_0 = p(x = 0)$ and $\varrho_0 = \varrho(x = 0)$. The semi-discretization~\eqref{FVmethod1Dsource} reduces to
\begin{equation}
    \frac{\mean{p}^+ - \mean{p}^-}{\Delta x} + \frac{\mean{\varrho}^+_\mathrm{log}\jump{\phi}^+ + \mean{\varrho}^-_\mathrm{log} \jump{\phi}^-}{2 \Delta x} = 0,
\end{equation}
where we introduced the abbreviation $\mean{a}^{\pm} = \mean{a}_{i\pm 1/2}$.
Inserting~\eqref{IsothermalSolution} results in
\begin{equation}
     RT\frac{\mean{\varrho}^+ - \mean{\varrho}^-}{\Delta x} - RT\frac{\jump{\varrho}^+ + \jump{\varrho}^-}{2 \Delta x} = 0,
\end{equation}
which is always satisfied  $\forall \varrho^{\pm}$.
\end{proof}
For constant potential temperature background state, we consider the hydrostatic balance explicitly expressed in terms of the variable $\varrho \theta$, namely in 1D
\begin{equation}
    K \partial_x ( \varrho \theta) ^ \gamma = -\varrho \partial_x \phi.
\end{equation}
If no perturbations in $\theta$ are present, i.e., $\theta = \text{const}$, the relation simplifies to
\begin{equation}
   K \theta^ \gamma \partial_x  \varrho^\gamma = -\varrho \partial_x \phi,
\end{equation}
and a closed form solution is of the form
\begin{equation}
    \frac{1}{\gamma - 1}\left ( \varrho^ {\gamma - 1} - \varrho_0^{\gamma-1}\right ) = - \frac{1}{K \gamma \theta^\gamma} (\phi - \phi_0),
    \label{closedformtheta}
\end{equation}
where $\varrho_0 = \varrho(x = 0)$ and $\phi_0 = \phi (x = 0)$. At this point we can prove the following lemma.
\begin{lemma}
    A source term $S^{\varrho v}$ preserves the hydrostatic balance prescribed by a constant background potential temperature $\theta$ of the compressible Euler equations if the source term is of the form
    \begin{equation}
 S^{\varrho v} = \mean{\varrho}_\gamma \jump{\phi},
 \label{nonconservativetheta}
\end{equation}
and the pressure term is discretized as $\mean{p}$.
\end{lemma}

\begin{proof}
 The semi-discretization~\eqref{FVmethod1Dsource} at the steady state reads
\begin{equation}
    \frac{\mean{p}^+ - \mean{p}^-}{\Delta x} + \frac{\frac{\gamma - 1}{\gamma} \frac{\jump{\varrho^\gamma}}{\jump{\varrho^{\gamma - 1}}}\jump{\phi}^+ + \frac{\gamma - 1}{\gamma} \frac{\jump{\varrho^\gamma}}{\jump{\varrho^{\gamma - 1}}} \jump{\phi}^-}{2 \Delta x} = 0.
\end{equation}
Using the steady state solution~\eqref{closedformtheta}, we can write the source term as
\begin{equation}
    S = \frac{\gamma - 1}{\gamma} \frac{\jump{\varrho^\gamma}}{\jump{\varrho^{\gamma - 1}}} \jump{\phi} = -K \theta^\gamma \jump{\varrho ^\gamma},
\label{eqnonconservativetheta}
\end{equation}
and consequently the semi-discretization reduces to
\begin{equation}
    K\theta^\gamma \frac{\mean{\varrho^\gamma}^+ -\mean{\varrho^\gamma}^-}{\Delta x} - K \theta^\gamma \frac{\jump{\varrho^\gamma}^+ + \jump{\varrho^\gamma}^-}{2\Delta x} = 0,
\end{equation}
which is always identically zero $\forall \varrho^\pm$, for the general property that $a_i = \mean{a}^\pm \mp \frac{\jump{a}^\pm}{2}$.
\end{proof}

In Tables~\ref{summary_log_fluxes} and \ref{summary_stolarsky_fluxes} we summarize the properties of the different numerical fluxes derived in this section. Note that an EC numerical flux for $(\varrho, \varrho v, \varrho E)$ cannot be TEC, KPEP and preserve hydrostatic balance for $\theta = \text{const}$ due to $f^\mathrm{num}_\varrho = \mean{\varrho}_\mathrm{log} \mean{v}$. On the other hand, for $\varrho \theta$ you can always at least satisfy two conditions among TEC, KPEP and well-balancedness for $\theta = \text{const}$ or $T = \text{const}$.

\begin{table}[h!]
    \centering
    \renewcommand{\arraystretch}{1.3} 
    \setlength{\tabcolsep}{10pt} 
    \caption{Properties and conserved variables for the different derived numerical fluxes. $\checkmark$: variable conserved or property satisfied. $\times$: variable not conserved or property not satisfied.}
    \label{summary_log_fluxes}
    \begin{tabular}{l|ccc}
        \toprule
         \textbf{$S^{\varrho v} = \mean{\varrho}_\mathrm{log} \jump{\phi}$ Conserved Variables and Properties} & \multicolumn{3}{c}{\textbf{Numerical Fluxes}} \\
        \cmidrule(lr){2-4}
        & \textbf{EC} & \textbf{TEC} & \textbf{ETEC} \\
        \midrule
        KPEP  & $\checkmark$ & $\checkmark$ & $\times$ \\
        TEC (including potential)  & $\times$  & $\checkmark$ & $\times$ \\
        Hydrostatic Balance $T = \text{const}$ & $\checkmark$  & $\checkmark$ & $\checkmark$ \\
        Hydrostatic Balance $\theta = \text{const}$ & $\times $  & $\times$ & $\times$ \\
        \bottomrule
    \end{tabular}
\end{table}
\begin{table}[h!]
    \centering
    \renewcommand{\arraystretch}{1.3} 
    \setlength{\tabcolsep}{10pt} 
    \caption{Properties and conserved variables for the different derived numerical fluxes. $\checkmark$: variable conserved or property satisfied. $\times$: variable not conserved or property not satisfied.}
    \label{summary_stolarsky_fluxes}
    \begin{tabular}{l|ccc}
        \toprule
         \textbf{$S^{\varrho v} = \mean{\varrho}_\gamma \jump{\phi}$ Conserved Variables and Properties} & \multicolumn{3}{c}{\textbf{Numerical Fluxes}} \\
        \cmidrule(lr){2-4}
        & \textbf{EC} & \textbf{TEC} & \textbf{ETEC} \\
        \midrule
        KPEP  & $\checkmark$ & $\checkmark$ & $\times$ \\
        TEC  & $\times$  & $\checkmark$ & $\times$ \\
        Hydrostatic Balance $T = \text{const}$ &  $\times $  & $\times$ & $\times$ \\
        Hydrostatic Balance $\theta = \text{const}$ & $\checkmark$  & $\checkmark$ & $\checkmark$ \\
        \bottomrule
    \end{tabular}
\end{table}

%% file: section4.tex
\section{Entropy-conservative DGSEM discretization}
\label{sec:dgsem}
In this section we extend all the previous properties to DGSEM.
To obtain an entropy-stable DGSEM discretization of a conservation law\revA{~\eqref{HyperbolicConservationLaw1D}}
the computational domain is divided into non-overlapping elements.
Within each element, the solution is approximated by Lagrange polynomials of degree $N$
defined on the Lobatto-Gauss-Legendre (LGL) nodes. These basis functions are continuous inside each element but may
be discontinuous across element boundaries. Such discontinuities are resolved with the introduction of two-point numerical fluxes.

The discrete scheme is obtained by multiplying the governing equation by test functions
of degree $N$, integrating by parts on each element, and evaluating the resulting integrals
using an LGL quadrature rule with $N+1$ nodes on the reference interval
$\xi \in [-1,1]$. This procedure leads to the semi-discrete formulation
\begin{equation}
  J\omega_i \vec{\dot{u}}_i(t) + (\vec{f}^*_{(N,R)} - \vec{f}_N)\delta_{iN} - (\vec{f}^*_{(0,L)} - \vec{f}_0)\delta_{i0} + 2 \sum_{j=0}^N\omega_i D_{ij} \vec{f}^\mathrm{vol}_{(\revB{i,j})} = 0,
  \label{dgsem_conservative}
\end{equation}
where $J$ is the element Jacobian, $\omega_i$ are the quadrature weights,
$D_{ij}=l'_j(\xi_i)$ is the differentiation matrix derived from the
Lagrange basis\revA{, $\delta_{ij}$ is the Kronecker delta function,} $\vec{f}^*_{(i,j)}$ is the numerical two-point surface flux\revA{,} $\vec{f}^\mathrm{vol}_{(j,i)}$ is the volume numerical two-point symmetric flux\revB{, $\vec{f}_i = \vec{f}(\vec{u}_i)$} \revA{ and the indexes $R$ and $L$ refer to the right and left neighboring nodes of the adjacent elements, respectively.}
Moreover, the differentiation matrix satisfies the summation-by-parts (SBP) property
\begin{equation}
    (\mtx{M}\mtx{D}) + (\mtx{M}\mtx{D})^T = \mtx{Q} + \mtx{Q}^T = \vec{B},
\end{equation}
where $\mtx{M}$ is the diagonal mass matrix, \revB{$\mtx{B} = \text{diag}(-1,0,\dots,0,1)$ is the boundary matrix} and $Q_{ij} = \omega_i D_{ij}$ \cite{gassner2013skew}.

\revA{A general nonconservative hyperbolic system can be written as
\begin{equation}
    \partial_t \vec{u} + \partial_x \vec{f}(\vec{u}) + \underbrace{\vec{r}(\vec{u}) \circ \partial_x \vec{b}(\vec{u})}_{\vec{g}(u)} = 0,
    \label{nonconservative_hyperbolic_1D}
\end{equation}
where $\circ$ denotes the Hadamard product, $\vec{r}(\vec{u})$ and $\vec{b}(\vec{u})$ are vectors valued functions of the state vector $\vec{u}$, and $\vec{g}(\vec{u})$ is the non-conservative product.}
A natural extension of the semidiscrete non-conservative products in the FV method is the DGSEM discretization of Eq.~\ref{nonconservative_hyperbolic_1D} (see, e.g., \cite{GASSNER2016291,DERIGS2018420})
\begin{align}
    J \revB{\omega_i} \dot{\vec{u}}_i = - \sum_{j=0}^N 2Q_{ij} \vec{f}^\mathrm{vol}_{(i,j)} \revB{+} \delta_{iN}\vec{f_N} \revB{-} \delta_{i0}\vec{f_0} + \sum_{j=0}^N Q_{ij} \vec{g}_{(i,j)}^\mathrm{vol} -\vec{f}^*_{(N,R)}\delta_{iN} \revB{+} \vec{f}^*_{(0,L)}\delta_{i0} \revB{-} \vec{g}^*_{(N,R)}\frac{\delta_{iN}}{2} \revB{+} \vec{g}^*_{(0,L)}\frac{\delta_{i0}}{2},
    \label{DGSEM1D}
\end{align}
where $\vec{g}^*_{(i,j)}$ is the surface numerical non-conservative flux and $\vec{g}^\mathrm{vol}$ is the volume numerical non-conservative flux.
Rueda-Ramirez et al.~\cite{RUEDARAMIREZ2024112607} showed that the DGSEM semi-discretization~\eqref{DGSEM1D} can be rewritten as flux differencing formula, if the non-conservative terms can be written
as the product of a local and a symmetric contribution, i.e.,
\begin{equation}
    \vec{g}^\mathrm{vol}_{(i,j)} = \revA{\vec{r}^{sym}_{(i,j)} \circ \vec{b}_i}.
    \label{form_g_vol}
\end{equation}
Here, we extend this result to non-conservative terms that are anti-symmetric, i.e., terms that can be written as the product of a symmetric contribution and an anti-symmetric contribution.
\begin{lemma}
It is possible to rewrite \eqref{DGSEM1D} as a flux-difference formula,
\begin{equation}
    J \omega_i \vec{\dot{\vec{u}}}_i = \vec{\Gamma}_{(i,i-1)} - \vec{\Gamma}_{(i,i+1)} \hspace{2cm} i= 0,...,N,
\end{equation}
where the indexes $i = -1$ and $i= N$ refer to the outer states (across the left and right boundaries, respectively) and $\Gamma_{i,k}$ is the so-called staggered (or telescoping) “flux” between node $i$ and the adjacent node $k$, if it is possible
to write the volume numerical non-conservative term as a product of an anti-symmetric and a symmetric contribution, i.e.,
\begin{equation}
    \vec{g}^\mathrm{vol}_{(i,j)} = \vec{\revA{r}}^{sym}_{(i,j)} \circ \jump{\vec{\revA{b}}}_{(i,j)}.
    \label{fluxg}
\end{equation}
\end{lemma}
\begin{proof}
The anti-symmetric (jump) term can be rewritten as
\begin{equation}
\jump{\vec{\revA{b}}}_{(i,j)} = 2 \mean{\vec{\revA{b}}}_{(i,j)} - 2 \vec{\revA{b}}_i.
\end{equation}
Thus, the volume term reads
\begin{equation}
    \vec{g}^\mathrm{vol}_{(i,j)} = 2 \vec{\revA{r}}^{sym}_{(i,j)} \circ \mean{\vec{\revA{b}}}_{(i,j)} - 2 \vec{\revA{r}}^{sym}_{(i,j)} \circ \vec{\revA{b}}_i,
\end{equation}
where the first term is purely symmetric, and the second term is the product of a symmetric and a local contribution. Therefore, the rest of proof follows the same steps as the flux-differencing proposition in~\cite[Proposition 1]{RUEDARAMIREZ2024112607}, \revB{with
\begin{multline}
\vec{\Gamma}_{(i,k)} = \sum_{i=0}^{\min(i,k)}\sum_{j=0}^N \tilde{Q}_{ij} \Bigl ( \vec{f}^\mathrm{vol}_{(i,j)} + \vec{r}^{\mathrm{sym}}_{(i,j)} \circ \mean{\vec{b}}_{(i,j)} \Bigr ) \\- \vec{b}_j \circ \sum_{i=0}^{\min(i,k)}\sum_{j=0}^N \tilde{Q}_{ij} \vec{r}^{\mathrm{sym}}_{(i,j)}, \quad i = 0,\dots, N, k \in \{i-1, i+1 \}
\end{multline}
\begin{equation}
    \vec{\Gamma}_{(0,-1)} = \vec{f}^*_{(0,L)} + \frac{1}{2}\vec{g}^*_{(0,L)}, \quad \vec{\Gamma}_{(N,N+1)} = \vec{f}^*_{(N,R)} + \frac{1}{2}\vec{g}^*_{(N,R)},
\end{equation}
where $\vec{\tilde{Q}} = 2 \vec{Q} - \vec{B}$.}
\end{proof}

Note that the anti-symmetric flux \eqref{fluxg} is consistent with zero, i.e. $\vec{g}^*(\vec{u}, \vec{u}) = 0$.
Note that the non-conservative anti-symmetric flux also reduces computational cost in the volume term, as it avoids the diagonal contributions and requires computing only the off-diagonal terms, either from the upper or the lower triangle of the matrix, similar to the contribution of the conservative symmetric volume flux \cite{ranocha2023efficient}.
\revA{For the compressible Euler equations with gravity the nonconservative product can be written as in~\eqref{nonconservative_hyperbolic_1D}
\begin{equation}
\vec{r}_{\varrho \theta} (\vec{u}) = \varrho \vec{1}, \quad \vec{b}_{\varrho \theta}(\vec{u})^T = (0, \phi, 0)
\end{equation}
and
\begin{equation}
\vec{r}_{\varrho E} (\vec{u})^T = (0, \varrho, \varrho v), \quad \vec{b}_{\varrho E}(\vec{u})^T = (0, \phi, \phi)
\end{equation}
respectively for the two formulations of the compressible Euler equations with gravity.
}

\subsection{Conservation properties of DGSEM}
In this section we extend the FV properties of the numerical fluxes for the non-conservative product of the geopotential term to the DGSEM.
First, we extend the result of Lemma~\ref{lemmathetatpec}.
\begin{lemma}
Given the set of variables $(\varrho, \varrho v, \varrho \theta)$, a numerical non-conservative surface and volume flux $g_{\varrho v}$ and the corresponding DGSEM discretization \eqref{DGSEM1D} are TEC if

\begin{enumerate}[label=(\roman*)]
    \item the density and non conservative fluxes are of the form
    \begin{equation}
     f^{\varrho,*} = f^{\varrho, \text{vol}} = \overline{\varrho} \mean{v}, \quad
    g^{\varrho v,*} = g^{\varrho v, \text{vol}} = \overline{\varrho} \jump{\phi},
\end{equation}
    \item \revB{and the volume and the surface fluxes $\vec{f}^\mathrm{vol}$, $\vec{f}^*$ are of the form~\eqref{TECflux} or~\eqref{ETECflux}.}
\end{enumerate}
\end{lemma}
\begin{proof}
Since the proof mimics the steps for the FV framework, we only present the main steps.
We introduce the generic anti-symmetric volume and surface numerical flux
\begin{equation}
  g^{\varrho v, *} =  g^{\varrho v, \mathrm{vol}}  = \overline{\varrho} \jump{\phi},
\end{equation}
and we define the entropy variables
\begin{equation}
    \vec{q} = \begin{pmatrix}
        -\frac{1}{2}\frac{m^2}{\varrho^2} + \phi\\
        \frac{m}{\varrho}\\
        \frac{\gamma}{\gamma-1}K (\varrho \theta)^{\gamma-1}
    \end{pmatrix} = \begin{pmatrix}
        -\frac{1}{2}\frac{m^2}{\varrho^2} \\
        \frac{m}{\varrho}\\
        \frac{\gamma}{\gamma-1}K (\varrho \theta)^{\gamma-1}
    \end{pmatrix} + \begin{pmatrix}
        \phi\\
        0\\
        0
    \end{pmatrix}
   = \vec{v} + \phi \vec{e_1}.
\end{equation}
Since the conservative numerical flux is TEC,
\begin{equation}
    \jump{\vec{v}^T} \vec{f}^\mathrm{num} = \jump{\psi}.
\end{equation}
We can contract in entropy space
\begin{equation}
    \sum_{i=0}^N \vec{v}_i^T \vec{M} J \dot{\vec{u}}_i = \langle \vec{v}, J \vec{\dot{u}}\rangle_{\mtx{M}} = J\dot{\eta}.
\end{equation}
In the following proof we drop the term $\mtx{M}$, and divide the volume contribution $\text{VOL}$ and the surface contribution SURF:
\begin{equation}
\begin{aligned}
    J \dot{\eta} = &-\sum_{i = 0}^N \vec{v}_i^T \left ( \sum_{j=0}^N 2 Q_{ij} \vec{f}^\mathrm{vol}_{(i,j)} - \delta_{iN}\vec{f_N} + \delta_{i0}\vec{f_0} + \sum_{j=0}^N Q_{ij} \vec{g}_{(i,j)}^\mathrm{vol} \right )\\
    &-\sum_{i = 0}^N \vec{v}_i^T  \left (\vec{f}^*_{(N,R)}\delta_{iN} - \vec{f}^*_{(0,L)}\delta_{i0} + \vec{g}^*_{(N,R)}\frac{\delta_{iN}}{2} - \vec{g}^*_{(0,L)}\frac{\delta_{i0}}{2}\right )\\
    = & - \text{VOL} - \text{SURF}.
\end{aligned}
\end{equation}
We first consider the volume term:
\begin{equation}
\begin{aligned}
    \text{VOL} &= \sum_{i = 0}^N \vec{q}_i^T \left ( \sum_{j=0}^N 2 \revB{Q}_{ij} \vec{f}^\mathrm{vol}_{(i,j)} - \delta_{iN}\vec{f_N} + \delta_{i0}\vec{f_0} + \sum_{j=0}^N \revB{Q}_{ij} \vec{g}_{(i,j)}^\mathrm{vol} \right )  \\ &= \sum_{i,j=0}^N \vec{q}_i^T \revB{Q}_{ij} \vec{f}^\mathrm{vol}_{(i,j)} - \vec{q}^T_N\vec{f_N} + \vec{q}^T_0 \vec{f_0} + \sum_{i,j=0}^N \vec{q}_i^T  Q_{ij} \vec{g}_{(i,j)}^\mathrm{vol}.
\end{aligned}
\end{equation}
By applying the SBP property $\revB{2Q}_{ij} = B_{ij} - Q_{ji} + Q_{ij}$,
\begin{align}
    \text{VOL} = \sum_{i,j=0}^N \vec{q}_i^T (\revB{B}_{ij} - Q_{ji} + Q_{ij}) \vec{f}^\mathrm{vol}_{(i,j)} - \vec{q}^T_N\vec{f_N} + \vec{q}^T_0 \vec{f_0} + \frac{1}{2}\sum_{i,j=0}^N \vec{q}_i^T  (\revB{B}_{ij} - Q_{ji} + Q_{ij}) \vec{g}_{(i,j)}^\mathrm{vol}.
\end{align}
Because of the flux consistency on the diagonal elements and the definition of $\vec{B}$, we have
\begin{align}
   \text{VOL} = \sum_{i,j=0}^N \vec{q}_i^T (- Q_{ji} + Q_{ij}) \vec{f}^\mathrm{vol}_{(i,j)}+ \frac{1}{2}\sum_{i,j=0}^N \vec{q}_i^T  ( - Q_{ji} + Q_{ij}) \vec{g}_{(i,j)}^\mathrm{vol}.
\end{align}
Since the $\revB{\vec{f}}_{(i,j)} = \revB{\vec{f}}_{(j,i)}$ is symmetric and $\revB{\vec{g}}_{(i,j)} =  -\revB{\vec{g}}_{(j,i)}$ is anti-symmetric, we can rearrange and reindex
\begin{align}
    \text{VOL} =  \sum_{i,j=0}^N (\vec{q}_i^T - \vec{q}_j^T) Q_{ij} \vec{f}^\mathrm{vol}_{(i,j)}+ \frac{1}{2}\sum_{i,j=0}^N (\vec{q}_i^T + \vec{q}_j^T) Q_{ij} \vec{g}_{(i,j)}^\mathrm{vol}.
\end{align}
Applying the definition of TEC flux, we are left with
\begin{align}
    \text{VOL} =  \sum_{i,j=0}^N Q_{ij}(\psi_i - \psi_j) + \sum_{i,j=0}^N (\phi_i - \phi_j) \revB{Q}_{ij} f^{\varrho}_{(i,j)}
  + \frac{1}{2}\sum_{i,j=0}^N (v_i + v_j) \revB{Q}_{ij} \overline{\varrho}_{i,j} (\phi_j - \phi_i).
\end{align}
Because of the form of the density flux, \revB{i.e., $f^{\varrho}_{(i,j)} = \overline{\varrho}_{i,j} \mean{v}_{i,j}$, and} by applying the properties of the differential operator $\vec{Q}$ we have $\text{VOL} = \psi_0 - \psi_N$.
Now that we simplified the volume term, we can sum also the surface term
\begin{align}
    J \dot{\eta}
    = -\text{VOL} - \text{SURF}
    = \psi_N - \psi_0 - \vec{q}_N^T \vec{f}^*_{(N,R)} +
    \vec{q}_0^T \vec{f}^*_{(0,L)} -\frac{1}{2}\vec{q}_N^T  \vec{g}^*_{(N,R)} + \frac{1}{2}\vec{q}_0^T \vec{g}^*_{(0,L)}.
\end{align}
The proof now proceeds as for the FV method.
\end{proof}

The remaining properties are straightforward following the same steps as in the above lemma.

\subsection{Well-balanced DGSEM}
\begin{lemma}
    The DGSEM discretization \eqref{DGSEM1D} with non-conservative surface and volume fluxes $g^*_{(i,j)}$, $g^\mathrm{vol}_{(i,j)}$ preserves the hydrostatic balance prescribed by a constant background temperature $T$ of the compressible Euler equations if
    \begin{enumerate}[label=(\roman*)]
    \item the numerical non-conservative fluxes are of the form
    \begin{equation}
        g^{*} = g^\mathrm{vol} = \mean{\varrho}_{\log} \jump{\phi},
    \end{equation}
    \item and the pressure term is discretized as $\mean{p}$.
\end{enumerate}
\label{lemmaTDGSEM}
\end{lemma}
\begin{proof}
    The semi-discretization \eqref{DGSEM1D} applied to \revB{the momentum of} the Euler equations and considering a background state reduces for $i = 0$ to
\begin{multline}
    -\mean{p} + p_R - \mean{\varrho}_\mathrm{log} \frac{\phi_L - \phi_R}{2} + 2 \sum_{j=0}^N\omega_i D_{ij} \left (\mean{p}_{i,j}  + \mean{\varrho}_{\text{log},i,j}\frac{\jump{\phi}_{j,i}}{2} \right )
    \\=
    -\mean{p} + p_R + \mean{\varrho}_\mathrm{log} \frac{\jump{\phi}}{2}+ 2 \sum_{j=0}^N\omega_i D_{ij} \left (\mean{p}_{i,j}  + \mean{\varrho}_{\text{log},i,j}\frac{\jump{\phi}_{j,i}}{2} \right ).
\end{multline}
Substituting the solution for the isothermal background balance
$\jump{\log (\varrho)} = - \jump{\phi} / (RT)$ \revB{and the ideal gas law $p = \varrho R T$} yields
\begin{multline}
    -RT \mean{\varrho} + RT \varrho_R +RT\frac{\jump{\varrho}}{2}
    + 2 \sum_{j=0}^N\omega_i D_{ij} \left (\frac{p_j}{2}
    -RT \frac{\jump{\varrho}}{2} \right )
    = 2 \sum_{j=0}^N\omega_i D_{ij} \left (\frac{p_j}{2}  + RT\frac{\varrho_j}{2} \right ) = 0.
\end{multline}
For the case $i = N$ and general $i$, the proof follows the same steps.
\end{proof}

\begin{lemma}
    The DGSEM discretization \eqref{DGSEM1D} non-conservative surface and volume flux $g^*_{(i,j)}$, $g^\mathrm{vol}_{(i,j)}$  preserves the hydrostatic balance prescribed by a constant background potential temperature $\theta$ of the compressible Euler equations if
    \begin{enumerate}[label=(\roman*)]
    \item the numerical non-conservative fluxes are of the form
    \begin{equation}
    g^{*}_{(i,j)} = g^\mathrm{vol}_{(i,j)} = \mean{\varrho}_\gamma \jump{\phi},
\end{equation}

    \item and the pressure term is discretized as $\mean{p}$.
\end{enumerate}
\end{lemma}
\begin{proof}
    The proof follows the same steps as the one of Lemma~\ref{lemmaTDGSEM}.
\end{proof}

%% file: section5.tex
\section{Extension to three-dimensional curvilinear meshes}
\label{sec:dgsem_curved}

In this work, we recall the main ingredients of the DGSEM curvilinear formulation, following closely the notation introduced in \cite{RUEDARAMIREZ2021110580}, see also \cite{GASSNER2016291,DERIGS2018420}.
On a three-dimensional curvilinear element, the semi-discrete DGSEM formulation for the advective and non-conservative contributions can be written in compact form as
\begin{equation}
    J_{ijk}\,\omega_{ijk}\,\dot{\vec{u}}_{ijk}
    = \vec{\mathcal{F}}_{ijk},
    \label{eq:dgsem3d}
\end{equation}
where $J_{ijk}$ denotes the mapping Jacobian, $\omega_{ijk}$ the quadrature weight at node $(i,j,k)$, and $\vec{\mathcal{F}}_{ijk}$ the discrete operator collecting all numerical flux contributions in the three spatial directions.

\medskip
An explicit expression of the operator reads
\begin{align}
    \vec{\mathcal{F}}_{ijk} = &
      \omega_{jk}\Bigg (-2 \sum_{m=0}^N Q_{im} \vec{\tilde{f}}^{1,\text{vol}}_{(i,m)jk}
      - \sum_{m=0}^N Q_{im} \vec{\tilde{g}}^{1}_{(i,m)jk}
      - \delta_{i0} \left [ \vec{\overset{\leftrightarrow}{f}} \cdot J\vec{a}^1 \right ]_{0jk}
      + \delta_{iN} \left [ \vec{\overset{\leftrightarrow}{f}} \cdot J\vec{a}^1 \right ]_{Njk} \nonumber \\
      & \hspace{0.9cm}
      -2 \sum_{m=0}^N Q_{jm} \vec{\tilde{f}}^{2,\text{vol}}_{i(j,m)k}
      - \sum_{m=0}^N Q_{jm}\,\vec{\tilde{g}}^{2}_{i(j,m)k}
      - \delta_{j0} \left [ \vec{\overset{\leftrightarrow}{f}} \cdot J\vec{a}^2 \right ]_{i0k}
      + \delta_{jN}\, \left [ \vec{\overset{\leftrightarrow}{f}} \cdot J\vec{a}^2 \right ]_{iNk} \nonumber \\
      & \hspace{0.9cm}
      -2 \sum_{m=0}^N Q_{km} \vec{\tilde{f}}^{3,\tilde{vol}}_{ij(k,m)}
      - \sum_{m=0}^N Q_{km}\vec{\tilde{g}}^{3}_{ij(k,m)}
      - \delta_{k0} \left [ \vec{\overset{\leftrightarrow}{f}} \cdot J\vec{a}^3\right ]_{ij0}
      + \delta_{kN}\, \left [\vec{\overset{\leftrightarrow}{f}} \cdot J\vec{a}^3\right ]_{ijN}
      \Bigg) \nonumber \\
      &+ \omega_{jk}\left ( \delta_{i0}\left [\vec{\tilde{f}}^*_{(0,L)jk}+\frac{1}{2}\vec{\tilde{g}}^*_{(0,L)jk} \right ]
      - \delta_{iN}\left [\vec{\tilde{f}}^*_{(N,R),jk}+\frac{1}{2}\vec{\tilde{g}}^*_{(N,R)jk} \right ] \right ) \nonumber \\
      &+ \omega_{ik}\left ( \delta_{j0} \left [\vec{\tilde{f}}^*_{i(0,L)k}+\frac{1}{2}\vec{\tilde{g}}^*_{i(0,L)k} \right ]
      - \delta_{jN} \left [\vec{\tilde{f}}^*_{i(N,R)k}+\frac{1}{2}\vec{\tilde{g}}^*_{i(N,R)k} \right ] \right ) \nonumber \\
      &+ \omega_{ij}\left ( \delta_{k0} \left [\vec{\tilde{f}}^*_{ij(0,L)}+\frac{1}{2}\vec{\tilde{g}}^*_{ij(0,L)}\right ]
      - \delta_{kN} \left [\vec{\tilde{f}}^*_{ij(N,R)}+\frac{1}{2}\vec{\tilde{g}}^*_{ij(N,R)} \right ] \right ).
    \label{DGSEM3D}
\end{align}
Here, $Q_{ij}=\omega_i D_{ij}$ is the SBP matrix, and $\vec{a}^m_{ijk}$ are the contravariant basis vectors linking the reference coordinates $(\xi^1,\xi^2,\xi^3)\in[-1,1]^3$ to the physical domain. The symbol $\tilde{\vec{g}}$ accounts for additional non-conservative terms.
The numerical two–point fluxes in the three directions are consistently defined with the metric terms, e.g.,
\begin{align}
    \vec{\tilde{f}}^{1,\text{vol}}_{(i,m)jk} &= \vec{\overset{\leftrightarrow}{f}}^{\text{vol}}(\vec{u}_{ijk}, \vec{u}_{mjk}) \cdot \mean{J\vec{a}^1}_{(i,m)jk}, \\
    \vec{\tilde{f}}^{2,\text{vol}}_{i(j,m)k} &= \vec{\overset{\leftrightarrow}{f}}^{\text{vol}}(\vec{u}_{ijk}, \vec{u}_{imk}) \cdot \mean{J\vec{a}^2}_{i(j,m)k}, \\
    \vec{\tilde{f}}^{3,\text{vol}}_{ij(k,m)} &= \vec{\overset{\leftrightarrow}{f}}^{\text{vol}}(\vec{u}_{ijk}, \vec{u}_{ijm}) \cdot \mean{J\vec{a}^3}_{ij(k,m)},
\end{align}
where we adopted the common notation $\vec{\overset{\leftrightarrow}{f}}$ to denote block vectors, containing a state vector in each spatial component (see, e.g., \cite{BOHM2020108076,winters2021dgsem})
\begin{equation}
\overset{\leftrightarrow}{\mathbf{f}} =
\begin{bmatrix}
\mathbf{f}_1 \\
\mathbf{f}_2 \\
\mathbf{f}_3
\end{bmatrix}, \quad
\overset{\leftrightarrow}{\mathbf{\vec{g}}} =
\begin{bmatrix}
\mathbf{\vec{g}}_1 \\
\mathbf{\vec{g}}_2 \\
\mathbf{\vec{g}}_3
\end{bmatrix},
\end{equation}
with each component defined as follows
\begin{equation}
\mathbf{\vec{f}}_1 =
\begin{pmatrix}
\varrho u \\
\varrho u^2 + p \\
\varrho u v \\
\varrho u w\\
\varrho \theta u
\end{pmatrix}, \quad
\mathbf{\vec{f}}_2 =
\begin{pmatrix}
\varrho v \\
\varrho v u \\
\varrho v^2 + p \\
\varrho v w \\
\varrho \theta v
\end{pmatrix}, \quad
\mathbf{\vec{f}}_3 =
\begin{pmatrix}
\varrho w \\
\varrho w u \\
\varrho w v \\
\varrho w^2 + p\\
\varrho \theta w
\end{pmatrix},
\end{equation}
\revA{
\begin{equation}
\mathbf{\vec{g}}_i =
\varrho \begin{pmatrix}
0\\
\delta_{i1} \\
\delta_{i2} \\
\delta_{i3} \\
0
\end{pmatrix} \circ\ \partial_{x_i} \begin{pmatrix}
0\\
\phi \\
\phi \\
\phi \\
0
\end{pmatrix} = \vec{r} \circ\ \partial_{x_i}\vec{\phi}, \, i = 1,2,3.
\end{equation}}
The non-conservative and anti-symmetric volume fluxes are defined as
\begin{align}
    \vec{\tilde{g}}^{1,\text{vol}}_{(i,m)jk} &= \vec{\overset{\leftrightarrow}{g}}^{\text{vol}}(\vec{u}_{ijk}, \vec{u}_{mjk}) \cdot \mean{J\vec{a}^1}_{(i,m)jk}, \\
    \vec{\tilde{g}}^{2,\text{vol}}_{i(j,m)k} &= \vec{\overset{\leftrightarrow}{g}}^{\text{vol}}(\vec{u}_{ijk}, \vec{u}_{imk}) \cdot \mean{J\vec{a}^2}_{i(j,m)k}, \\
    \vec{\tilde{g}}^{3,\text{vol}}_{ij(k,m)} &= \vec{\overset{\leftrightarrow}{g}}^{\text{vol}}(\vec{u}_{ijk}, \vec{u}_{ijm}) \cdot \mean{J\vec{a}^3}_{ij(k,m)}.
\end{align}
The surface numerical flux is defined similarly as
\begin{align}
    \vec{\tilde{f}}^{1,*}_{(i,m)jk} &= \vec{\overset{\leftrightarrow}{f}}^{*}(\vec{u}_{ijk}, \vec{u}_{mjk}) \cdot \mean{J\vec{a}^1}_{(i,m)jk}, \\
    \vec{\tilde{f}}^{2,*}_{i(j,m)k} &= \vec{\overset{\leftrightarrow}{f}}^{*}(\vec{u}_{ijk}, \vec{u}_{imk}) \cdot \mean{J\vec{a}^2}_{i(j,m)k}, \\
    \vec{\tilde{f}}^{3,*}_{ij(k,m)} &= \vec{\overset{\leftrightarrow}{f}}^{*}(\vec{u}_{ijk}, \vec{u}_{ijm}) \cdot \mean{J\vec{a}^3}_{ij(k,m)}.
\end{align}
\subsection{TEC and KPEP with curvilinear coordinates}
In this section, we consider the compressible Euler equations with the potential temperature and extend the property for TEC. The extension of the other results presented in the previous 1D Cartesian semi-discretization is straightforward, following the same approach, therefore we will only state them.
\begin{theorem}
    Given a semi-discretization of the compressible Euler equations as in~\eqref{DGSEM3D}, for the set of conserved variables $(\varrho, \varrho \vec{V}, \varrho \theta)$, a numerical non-conservative surface and volume fluxes, a numerical conservative surface and volume fluxes are TEC if
    \begin{enumerate}[label=(\roman*)]
    \item the numerical density flux and non-conservative fluxes are of the form
    \begin{equation}
    \vec{f}^{\varrho,*} = \vec{f}^{\varrho,\mathrm{vol}} = \overline{\varrho} \mean{\vec{V}},\quad \vec{g}^{\varrho \vec{V},*} = \vec{g}^{\varrho \vec{V}, \mathrm{vol}} = \overline{\varrho} \jump{\vec{\phi}},
\end{equation}
    \item \revB{and the volume and the surface fluxes $\vec{f}^\mathrm{vol}$, $\vec{f}^*$ are of the form~\eqref{TECflux} or ~\eqref{ETECflux}}.
\end{enumerate}
\end{theorem}
\begin{proof}
    Consider the first volume term in the $\xi_1$ direction and contract it in the entropy variable space, where the entropy variables are here denoted with $\vec{v}$, to avoid confusion with the weights $\omega_{ijk}$.
    \begin{equation}
      \sum_{i=0}^{N}\vec{v}^T_{ijk}\vec{VOL}^1 = \omega_{jk} \sum_{i=0}^{N}\vec{v}^T_{ijk}\left (-2 \sum_{m=0}^N Q_{im} \vec{\tilde{f}}^{1,\text{vol}}_{(i,m)jk}
      - \sum_{m=0}^N Q_{im} \vec{\tilde{g}}^{1}_{(i,m)jk}
      - \delta_{i0} \left [ \vec{\overset{\leftrightarrow}{f}}\cdot J\vec{a}^1 \right ]_{0jk}
      + \delta_{iN} \left [ \vec{\overset{\leftrightarrow}{f}}\cdot J\vec{a}^1 \right ]_{Njk} \right ).
    \end{equation}
    Applying the SBP property $\vec{Q}_{ij} = B_{ij} - Q_{ji} + Q_{ij}$, we can write
    \begin{equation}
        \begin{aligned}
        \sum_{i=0}^{N}\vec{v}^T_{ijk}\vec{VOL}^1 = &\omega_{jk} \sum_{i=0}^{N}\vec{v}^T_{ijk}\Bigg (-2 \sum_{m=0}^N \left ( B_{ij} - Q_{ji} + Q_{ij} \right )  \vec{\tilde{f}}^{1,\text{vol}}_{(i,m)jk}
      - \sum_{m=0}^N \left ( B_{ij} - Q_{ji} + Q_{ij} \right ) \vec{\tilde{g}}^{1}_{(i,m)jk} \\ &
      - \delta_{i0} \left [ \vec{\overset{\leftrightarrow}{f}}\cdot J\vec{a}^1 \right ]_{0jk}
      + \delta_{iN} \left [ \vec{\overset{\leftrightarrow}{f}}\cdot J\vec{a}^1 \right ]_{Njk} \Bigg ).
        \end{aligned}
    \end{equation}
    Because of the flux consistency on the diagonal elements and the definition of $\vec{B}$, we have
    \begin{equation}
        \sum_{i=0}^{N}\vec{v}^T_{ijk}\vec{VOL}^1 = \omega_{jk} \Bigg (- \sum_{i,m =0}^{N}\vec{v}^T_{ijk}\left ( - Q_{ji} + Q_{ij} \right )  \vec{\tilde{f}}^{1,\text{vol}}_{(i,m)jk}
      - \frac{1}{2}\sum_{i,m =0}^{N}\vec{v}^T_{ijk} \left (- Q_{ji} + Q_{ij} \right ) \vec{\tilde{g}}^{1}_{(i,m)jk} \Bigg ).
    \end{equation}
Here the proof closely follows the step for the 1D DGSEM; indeed, due to the symmetry and anti-symmetric properties of the conservative and non-conservative fluxes, after a rearrangement of the index, we are left with
 \begin{equation}
         \sum_{i=0}^{N}\vec{v}^T_{ijk}\vec{VOL}^1 = \omega_{jk} \Bigg (- \sum_{i,m =0}^{N} Q_{im}\left ( \vec{v}^T_{ijk} - \vec{v}^T_{mjk} \right ) \vec{\tilde{f}}^{1,\text{vol}}_{(i,m)jk}
      - \frac{1}{2}\sum_{i,m =0}^{N}Q_{im} \left ( \vec{v}^T_{ijk} + \vec{v}^T_{mjk} \right ) \vec{\tilde{g}}^{1}_{(i,m)jk} \Bigg ).
 \end{equation}
Applying the definition of TEC flux and the non-conservative flux, we can write
\begin{equation}
    \sum_{i=0}^{N}\vec{v}^T_{ijk}\vec{VOL}^1 = \omega_{jk} \Bigg (- \sum_{i,m =0}^{N} Q_{im}\mean{J\vec{a}}_{(i,m)jk} \cdot \left ( \vec{\psi}_{ijk} - \vec{\psi}_{mjk} \right ) \Bigg ).
\end{equation}
The difference now with the 1D Cartesian case is the presence of the metric terms. After some manipulations and re-indexing, the term simplifies to
\begin{equation}
     \sum_{i=0}^{N}\vec{v}^T_{ijk}\vec{VOL}^1 = \omega_{jk} \Bigg (- \sum_{i,m =0}^{N} Q_{im} \left ( J\vec{a}^1 \right )_{mjk} \cdot \vec{\psi}_{ijk} + \left (J\vec{a}^1 \right )_{0jk} \cdot \vec{\psi}_{0jk} -\left (J\vec{a}^1 \right )_{Njk} \cdot \vec{\psi}_{Njk}  \Bigg ).
\end{equation}
The remaining terms can be treated similarly, as well as the surface term, that we do not report here for brevity. We are left with
\begin{equation}
    -\sum_{i,j,k=0}^N\omega_{ijk} \vec{\psi}_{ijk} \cdot \sum_{m = 0}^N \left ( D_{im} \left (J\vec{a}^1 \right )_{mjk} + D_{jm} \left (J\vec{a}^2 \right ) _{imk} + D_{km} \left (J\vec{a}^3 \right )_{ijm} \right )  = 0,
\end{equation}
which is equal to zero if the discrete metric identity \cite{Kopriva2006} is satisfied.
\end{proof}
\begin{theorem}
    Given a semi-discretization of the compressible Euler equations as in~\eqref{DGSEM3D}, for the set of conserved variables $(\varrho, \varrho \vec{V}, \varrho E)$, a numerical non-conservative surface and volume fluxes, a numerical conservative surface and volume fluxes are TEC if the numerical density flux and the non-conservative fluxes are of the form
    \begin{equation}
    \vec{f}^{\varrho, *} = \vec{f}^{\varrho, \mathrm{vol}} = \overline{\varrho} \mean{\vec{V}},\quad \vec{g}^{\varrho \vec{V},*} = \vec{g}^{\varrho \vec{V}, \mathrm{vol}} = \overline{\varrho} \jump{\phi},\quad \vec{g}^{\varrho E,*} = \vec{g}^{\varrho E, \mathrm{vol}} = \vec{f}^\varrho \jump{\vec{\phi}}. 
\end{equation}
\end{theorem}

\begin{theorem}
    Given a semi-discretization of the compressible Euler equations as in~\eqref{DGSEM3D}, for the set of conserved variables $(\varrho, \varrho \vec{V}, \varrho E)$ or $(\varrho, \varrho \vec{V}, \varrho \theta)$, a numerical non-conservative surface and volume fluxes, a numerical conservative surface and volume fluxes are KPEP if \begin{enumerate}[label=(\roman*)]
    \item the numerical density flux and the non-conservative fluxes are of the form
    \begin{equation}
    \vec{f}^{\varrho, *} = \vec{f}^{\varrho, \mathrm{vol}} = \overline{\varrho}\mean{\vec{V}}, \quad\vec{g}^{\varrho \vec{V},*} = \vec{g}^{\varrho \vec{V}, \mathrm{vol}} = \overline{\varrho} \jump{\vec{\phi}},
\end{equation}
    \item and the numerical \revB{volume and surface fluxes are} KEP.
\end{enumerate}

\end{theorem}
\subsection{Well-balancedness with curvilinear coordinates}
\begin{theorem}
    Given a semi-discretization of the compressible Euler equations as in \eqref{DGSEM3D}, for the set of conserved variables $(\varrho, \varrho \vec{V}, X)$, where $X$ represents any conserved variable that closes the system, a non-conservative numerical surface and volume flux $\vec{g}$ preserve the hydrostatic balance prescribed by a constant background temperature $T$ under any coordinate mapping if
    \begin{enumerate}[label=(\roman*)]
        \item the non-conservative volume flux is of the form
        \begin{equation}
    \vec{g}^{\varrho \vec{V},\text{vol}} = \mean{\varrho}_{\log} \jump{\vec{\phi}},
\end{equation}
\item the pressure term is discretized as $\mean{p}$,
        \item and the free-stream preservation property is satisfied at the discrete level.
    \end{enumerate}
\end{theorem}
\begin{proof}
For sake of simplicity and brevity, here we present the proof for 2D case only in the $x$ momentum.

We start by observing that the surface non-conservative fluxes are zeros, because of the continuity of the potential along the interfaces, i.e.,
 \begin{equation}
     \phi_{0j}^R = \phi_{Nj}, \quad \phi_{0j} = \phi_{Nj}^L, \quad \phi_{i0}^R = \phi_{iN}, \quad \phi_{i0} = \phi_{iN}^L.
 \end{equation}
Therefore, the semidiscrete formulation can be rewritten at the hydrostatic balance for the horizontal momentum equation as
  \begin{equation}
     \begin{aligned}
    J_{ij}\dot{\left (\varrho u\right )}_{ij} &+ \frac{1}{\omega_i}\left ( Ja^1_{1,Nj} \left ( \mean{p}_{(N,R)j}-  p_{Nj}\right ) \delta_{iN} - Ja^1_{1,0j}\left ( \mean{p}_{(0,L)} - p_{0j} \right ) \delta_{i0} \right ) \\ &+ \frac{1}{\omega_j} \left ( Ja^2_{1,iN} \left ( \mean{p}_{i,(N,R)} - p_{iN}\right )  \delta_{jN} - Ja^2_{1,i0} \left( \mean{p}_{i,(0,L)} - p_{i0} \right) \delta_{j0} \right )\\
    & + \sum_{k=0}^N 2 D_{ik} \mean{Ja^1_1}_{(i,k)j} \left (\mean{p}_{(i,k)j} + \frac{1}{2} \mean{\varrho}_{\log,(i,k)j} (\phi_{kj} - \phi_{ij}) \right )  \\
    &+\sum_{k=0}^N  2D_{jk} \mean{Ja^2_1}_{i(j,k)} \left (\mean{p}_{(i(j,k))} + \frac{1}{2}\mean{\varrho}_{\log,i(j,k)}(\phi_{ik} - \phi_{ij}) \right ) = 0.
    \end{aligned}
 \end{equation}
 The analytical solution of an isothermal hydrostatic flow are
 \begin{equation}
     p = p_0 e^{-\phi/RT} \quad \varrho = \frac{p_0}{RT}e^{-\phi/RT}.
 \end{equation}
Considering the first surface conservative term, we have
\begin{equation}
    \frac{p_{Nj}}{2} - \frac{p_{0j}^R}{2} = \frac{p_0}{2}\left ( e^{-\phi_{Nj}/RT} - e^{-\phi_{0j}^R/RT} \right ) = 0,
\end{equation}
which is again zero for the continuity of the potential along the interfaces. The same applies to all the remaining surface terms.

For the volume terms we just consider the first one
\begin{equation}
\text{VOL}^1 = \sum_{k=0}^N 2 D_{ik} \mean{Ja^1_1}_{(i,k)j} \left (\frac{p_{kj} + p_{ij}}{2} + \frac{1}{2} \frac{\varrho_{kj} - \varrho_{ij}}{\mathrm{log}(\varrho_{kj}) - \mathrm{log}(\varrho_{ij})}(\phi_{kj} - \phi_{ij}) \right ).
\end{equation}
From the hydrostatic balance, we derive $\jump{\mathrm{log}(\varrho)} = -\frac{\jump{\phi}}{RT}$
thus
\begin{equation}
\text{VOL}^1 = \sum_{k=0}^N 2 D_{ik} \mean{Ja^1_1}_{(i,k)j} \left (RT \frac{\varrho_{kj} + \varrho_{ij}}{2} - RT \frac{1}{2} (\varrho_{kj} - \varrho_{ij}) \right ) = \sum_{k=0}^N 2 D_{ik} \mean{Ja^1_1}_{(i,k)j} RT \varrho_{ij}.
\end{equation}
The second volume term is analogous
\begin{multline}
\text{VOL}^2 = \sum_{k=0}^N  2D_{jk} \mean{Ja^2_1}_{i(j,k)} \left (\frac{p_{ik} + p_{ij}}{2} + \frac{1}{2}\frac{\varrho_{ik} - \varrho_{ij}}{\mathrm{log}(\varrho_{ik}) - \mathrm{log}(\varrho_{ij})}(\phi_{ik} - \phi_{ij}) \right )
\\
=    \sum_{k=0}^N  2D_{jk} \mean{Ja^2_1}_{i(j,k)} RT \varrho_{ij}.
\end{multline}
Summing up all the volume terms
\begin{equation}
    \text{VOL} = \text{VOL}^1 + \text{VOL}^2 = 2RT \varrho_{ij} \sum_{k=0}^N \left ( D_{ik} \mean{Ja_1^1}_{(i,k),j} + D_{jk} \mean{J a_1^2}_{i(j,k)} \right )  =0,
\end{equation}
which is equal to zero, due to the free-stream preservation property.
\end{proof}
\begin{theorem}
    Given a semi-discretization of the compressible Euler equations as in \eqref{DGSEM3D}, for the set of conserved variables $(\varrho, \varrho \vec{V}, X)$, where $X$ represents any conserved variable that closes the system, a non-conservative numerical volume and surface flux $\vec{g}$ preserve the hydrostatic balance prescribed by a constant background potential temperature $\theta$ under any coordinate mapping if
    \begin{enumerate}[label=(\roman*)]
        \item the non-conservative volume flux is of the form
        \begin{equation}
     \vec{g}^{\varrho \vec{V},\mathrm{vol}} = \mean{\varrho}_\gamma \jump{\vec{\phi}},
\end{equation}
\item the pressure term is discretized as $\mean{p}$,
        \item and the free-stream preservation property is satisfied at the discrete level.
    \end{enumerate}
\end{theorem}
\begin{proof}
The proof is conceptually analogous to the previous theorem.
\end{proof}

%% file: section6.tex
\section{Numerical results}
\label{sec:numerical_results}
Next, we present numerical results for both set of equations~\eqref{EulerTotalEnergy},~\eqref{EulerPotentialTemperature}, assessing the theoretical findings described in the previous sections.
We have implemented our novel methods in Julia \cite{bezanson2017julia} using Trixi.jl~\cite{ranocha2022adaptive,schlottkelakemper2021purely}.
The time integration method is SSPRK43~\cite{kraaijevanger1991contractivity} implemented in OrdinaryDiffEq.jl~\cite{rackauckas2017differentialequations} and, unless stated otherwise, the fixed Courant number (CFL) is taken to be 1. \revB{The time step size is computed as
\begin{equation}
\Delta t = \text{CFL} \frac{2}{p + 1} \min_{i} \Bigl ( \frac{\Delta x_i}{\lambda_{\max} ( \vec{u}_i)} \Bigr ),
\label{time_step_dt}
\end{equation}
where the largest-magnitude eigenvalue of the flux Jacobian $\lambda_{\max} ( \vec{u}_i)$ and the local grid spacing $\Delta x_i$ are computed using the local curvilinear grid information, see~\cite{ranocha2025error,ranocha2021optimized,jahdali2021optimized} for further details.}
The conservative surface flux, except for the test cases to assess the conservation of entropy and total energy, is computed with the low Mach number approximate Riemann solver (LMARS) numerical flux~\cite{AControlVolumeModeloftheCompressibleEulerEquationswithaVerticalLagrangianCoordinate}. \revB{The volume conservative fluxes are computed with the EC numerical flux of Ranocha \cite{ranocha2018thesis} for the conserved variables $(\varrho,\varrho v, \varrho E)$ and with the TEC numerical flux for the conserved variables $(\varrho, \varrho v, \varrho \theta)$, with the logarithmic mean for the density in the density flux, if not stated otherwise.}
We used Makie.jl~\cite{danisch2021makie} and matplotlib~\cite{hunter2007matplotlib} for visualization.
All source code required to reproduce the numerical results presented here is available online in our reproducibility repository \cite{artiano2025structureRepro}.

The results are presented as follows: first, we assess the conservation of entropy and total energy for the potential temperature formulation. We then show numerical results for well-balanced schemes on curvilinear meshes, followed by a convergence analysis on Cartesian and curvilinear meshes for the inertia–gravity waves test case. A comparison between the total energy and potential temperature formulations is then carried out on mountain-wave test cases, including the linear hydrostatic, linear non-hydrostatic, and Schär mountain. Finally, the new formulation is tested on the baroclinic instability test case on a cubed-sphere mesh.

\subsection{Conservation of entropy and total energy}
First, we analyze the conservation properties of the newly derived numerical fluxes in the absence of a geopotential term. The semi-discretization is performed via FV method, hence the polynomial degree of DGSEM discretization is zero.
A well-known and famous test case to assess PEP property is the density wave, as in~\cite{Ranocha2022, Shima2021}. To avoid particular symmetries of the $\sin$ function as described in \cite{DEMICHELE2023112439}, we employ the initial condition
\begin{equation}
	(\varrho(x,0), v(x,0), p(x,0)) = (1 + \exp{\sin(2 \pi x)}, 1, 1).
\end{equation}
The solution is computed on the domain $[0, 1]$, over a time interval $[0, 40]$, with a $\text{CFL} = 0.01$. The domain is discretized with 64 cells and periodic boundary conditions are prescribed on the left and right boundaries.
\begin{figure}[h!]
    \centering
    \includegraphics[width=1.0\linewidth]{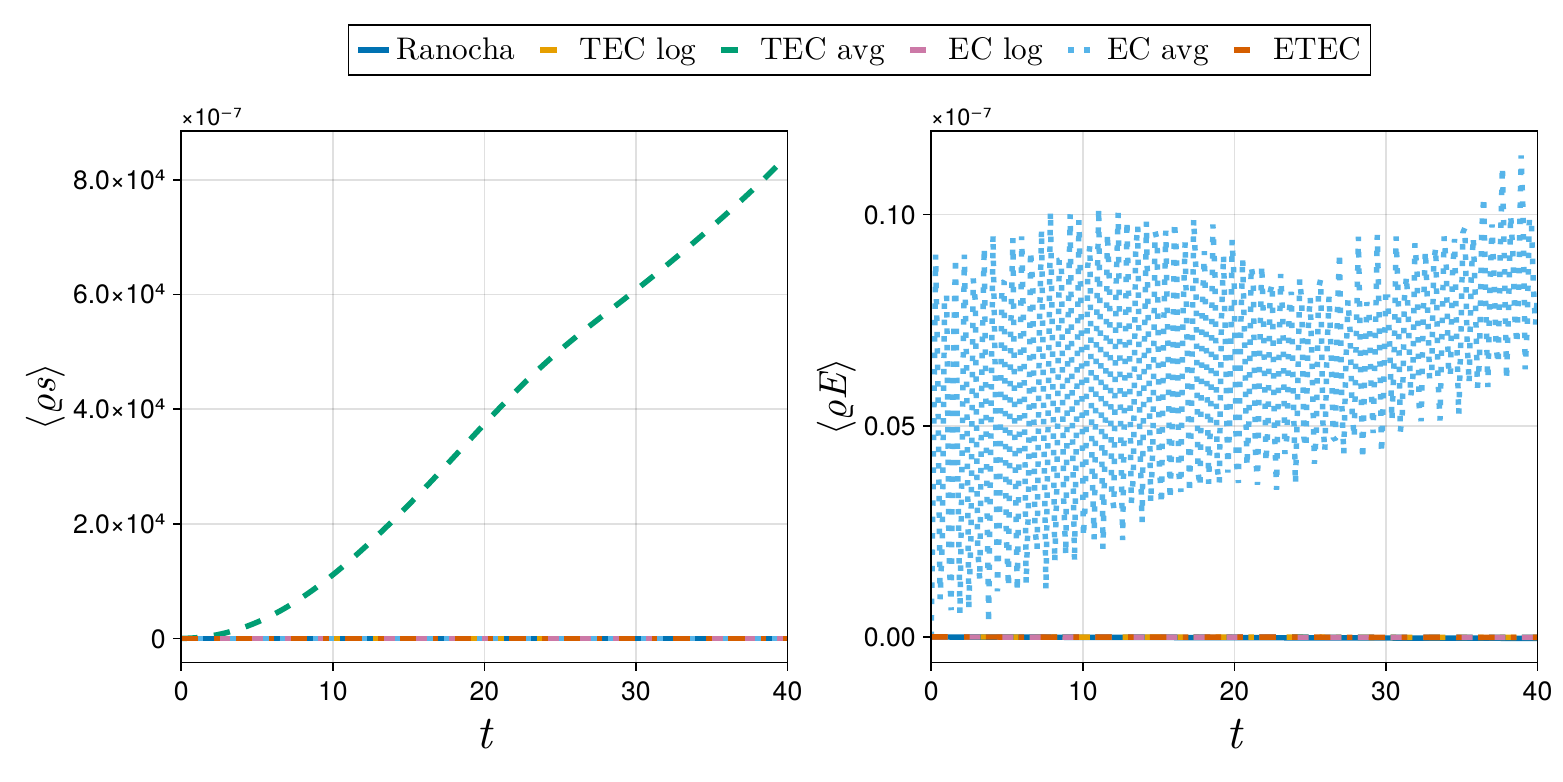}
    \caption{Evolution of the entropy (left) and total energy (right) integrals, normalized with respect to the initial value for the density wave test case using different numerical fluxes and the $\text{CFL} = 0.01$.}
    \label{dw_comparison}
\end{figure}
The EC, TEC and ETEC numerical fluxes for the set of conserved variables $(\varrho, \varrho v, \varrho \theta)$ are compared with different choice of the density mean in the density flux $f^{\varrho}$ with the EC numerical flux of Ranocha \cite{ranocha2018thesis} for the conserved variables $(\varrho,\varrho v, \varrho E)$.
The evolution of the entropy and total energy integrals can be seen in Figure~\ref{dw_comparison}.
The entropy conservation results show that the TEC numerical flux is not EC, when the density mean is discretized with the central mean.
However, when the logarithmic mean is used, the TEC is able to conserve the entropy up to round-off errors.
This can be explained considering that in this test case the velocity $v$ and the pressure $p$ are constants.
In fact, the numerical flux for the potential temperature reduces to $f^{\varrho \theta} = \varrho \theta v$ for $p = \text{const}$ and $v = \text{const}$. Thus, recalling Lemma~\ref{ECconstantvp}, the EC condition \eqref{ECCondition} is satisfied
\begin{equation}
    f^{\varrho} = \mean{\varrho}_{\log} \frac{f_{\varrho \theta}}{\varrho \theta} = \mean{\varrho}_{\log} v.
\end{equation}
For the conservation of the total energy $\varrho E$, as stated by the Lemma~\ref{TECconstantvp}, all the presented numerical fluxes are TEC for constants $p$ and $v$.
The divergent behaviour of the EC flux with the arithmetic mean for the density in the density flux is due to the lack of PEP property.
Indeed, Lemma~\ref{Lemma55} implies that an EC numerical flux with the arithmetic density mean is not PEP. For this reason, in the next sections we will omit the mean value used in the degree freedom of the EC and TEC flux, as it always will be logarithmic mean.

A more involved test case is the three-dimensional inviscid Taylor-Green vortex, which is a canonical benchmark for the transition to turbulence and therefore also the creation of small underesolved scales.
The initial conditions are given by~\cite{gassner2016split}
\begin{equation}
\centering
    \begin{aligned}
            \varrho(\vec{x},0) &= 1,\\
            u(\vec{x},0) &= \sin(x) \cos(y) \cos(z),\\
            v(\vec{x},0) &= -\cos(x) \sin(y) \cos(z),\\
            w(\vec{x},0) &= 0,\\
            p(\vec{x},0) &= 10 + \frac{(\cos(2x) + \cos(2y))(\cos(2x) + 2)-2}{16}.
    \end{aligned}
\end{equation}
The solution is computed on the domain $[0, 2 \pi]^3$, over a time interval $[0, 50]$, with a $\text{CFL} = 0.01$. The domain is discretized with 32 cells in each direction and periodic boundary conditions are prescribed in each direction.
\begin{figure}[h!]
    \centering
    \includegraphics[width=1.0\linewidth]{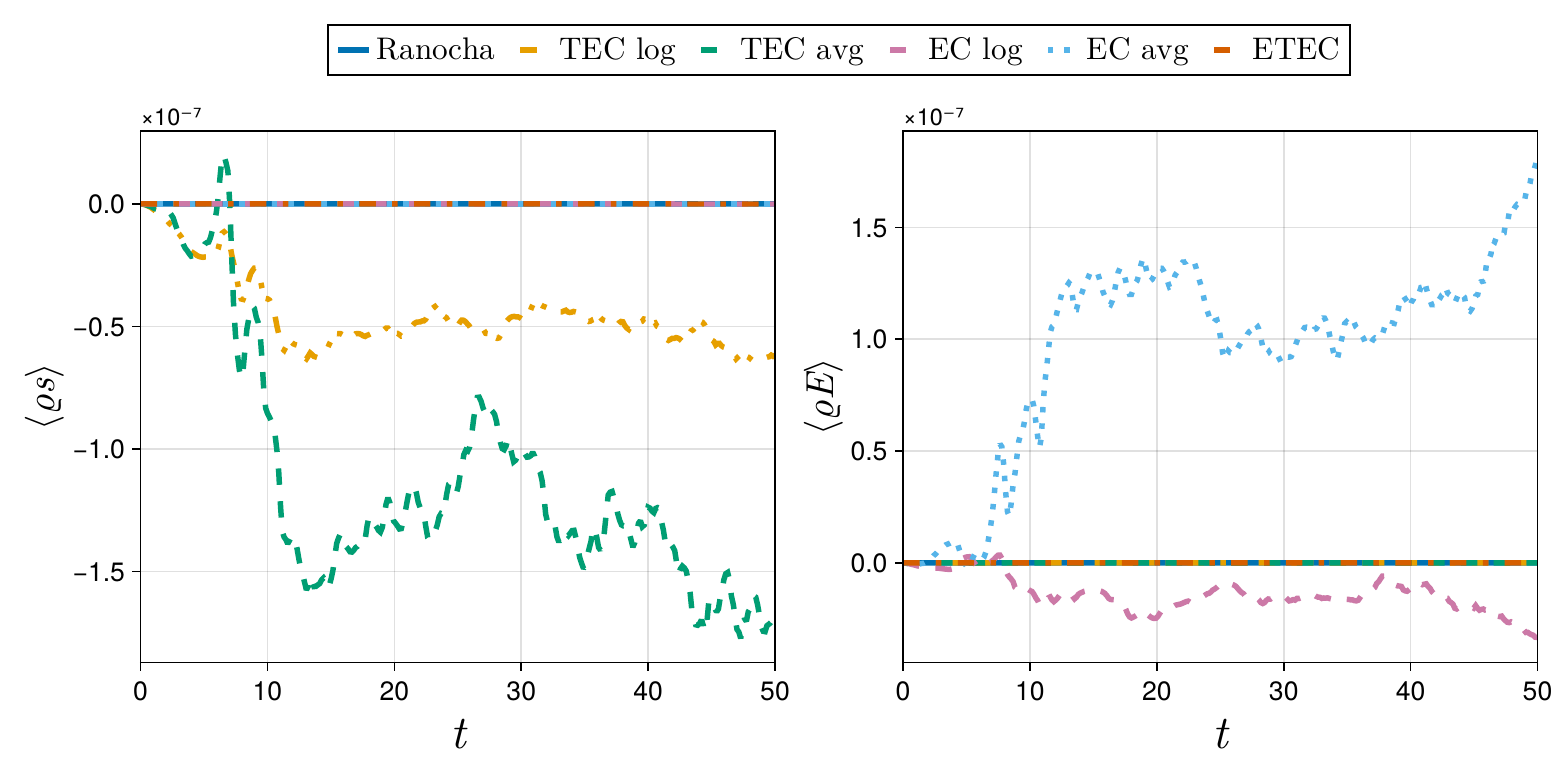}
    \caption{Evolution of the entropy (left) and total energy (right) integrals, normalized with respect to the initial value for the Taylor-Green Vortex test case using different numerical fluxes and the $\text{CFL} = 0.01$. }
    \label{tgv_comparison}
\end{figure}
The numerical fluxes ETEC and EC with arithmetic mean and logarithmic mean for the density show numerical conservation of the entropy up to round-off errors for long-time simulations, as shown in Figure~\ref{tgv_comparison}. 
The TEC fluxes clearly are not able to conserve the entropy and the use of a mean over the other does not present any particular benefits.
For the conservation of the total energy $\varrho E$, the numerical TEC fluxes with different density mean averages and the ETEC flux are able to conserve the total energy over long simulations.

\subsection{Well-balancedness test case}
In this section, we assess the well-balancing properties of the proposed numerical fluxes on curvilinear mesh. To this end, we consider two canonical hydrostatic background states: an isothermal atmosphere and a constant potential temperature (adiabatic) atmosphere.

The computational domain is discretized into $16 \times 16$ elements, with a polynomial degree of $2$. Time integration is carried out with a fixed time step of $\Delta t = 0.01$.

The initial conditions for the isothermal and adiabatic background states are defined as follows. For the isothermal case, the density and pressure fields satisfy the hydrostatic balance under a constant temperature $T_0 = \SI{250}{K}$, while the velocity field is initially at rest. For the adiabatic atmosphere, the potential temperature $\theta = \SI{300}{K}$ is constant, and the pressure and density are computed to satisfy hydrostatic equilibrium.
We employ a smooth curvilinear mapping of the reference coordinates $(\xi, \eta)$ to physical space $(x, y)$:
\begin{equation}
x = 500(1 + \xi + 0.1 \sin(\pi \xi) \sin(\pi \eta)), \quad y = 500(1 + \eta + 0.1 \sin(\pi \xi) \sin(\pi \eta)).
\end{equation}
The warped mesh is shown in Figure~\ref{warped_mesh}.
The simulations were run up to the final time $T = \SI{5000}{s}$.
\begin{figure}[h!]
    \centering
    \includegraphics[width = \textwidth]{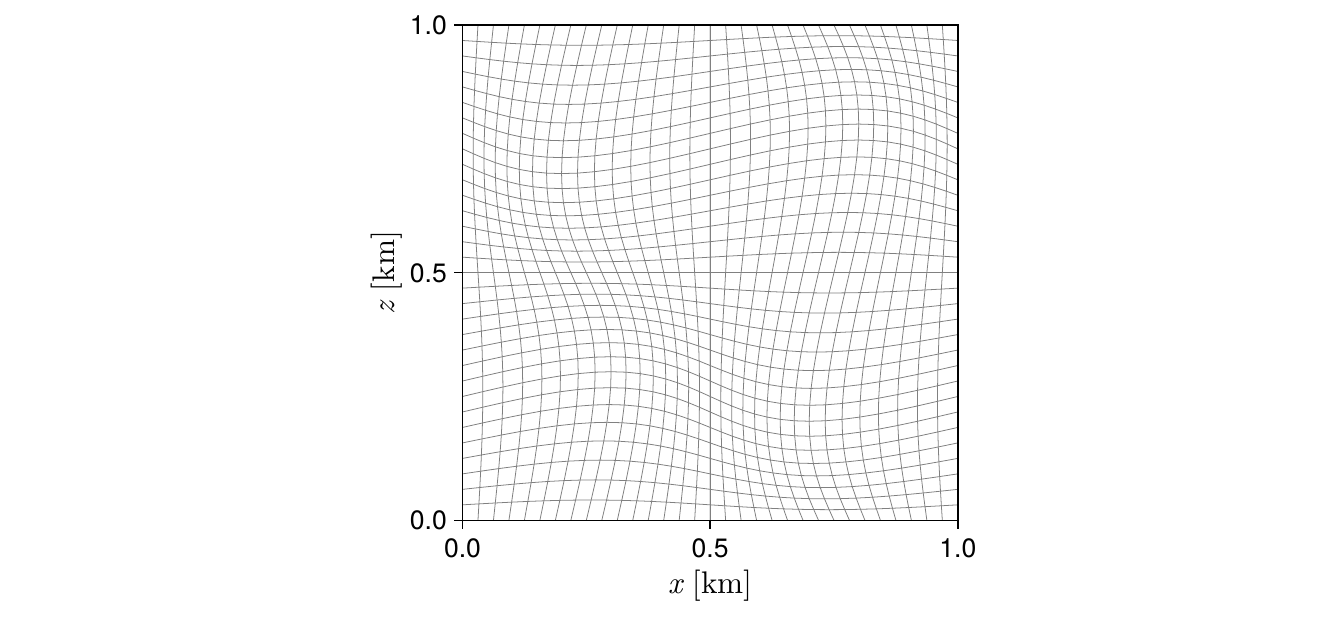}
    \caption{Warped mesh used for well-balancedness test cases.}
    \label{warped_mesh}
\end{figure}
In Figures~\ref{balance_rhoE} and~\ref{balance_rhotheta}, the $L_2$ error of the velocity components is shown over the timesteps. Due to the singularity of the logarithmic and Stolarsky mean, a numerically stable evaluation through Taylor expansion is required~\cite{ismail2009affordable,Winters2020}.
These approximations may introduce floating-point errors that accumulate and cause an increase in the error over long integration intervals. Moreover, for the constant potential temperature background state, the growth may occur because floating-point errors trigger a physical instability, as this hydrostatic balance state is metastable. To verify that the growth was solely due to floating-point inaccuracies, we repeated the simulation using 64-bit double precision. The error in the well-balancedness is preserved up to machine precision, even after $500,000$ timesteps.
\begin{figure}[h!]
    \centering
    \includegraphics[width = \textwidth]{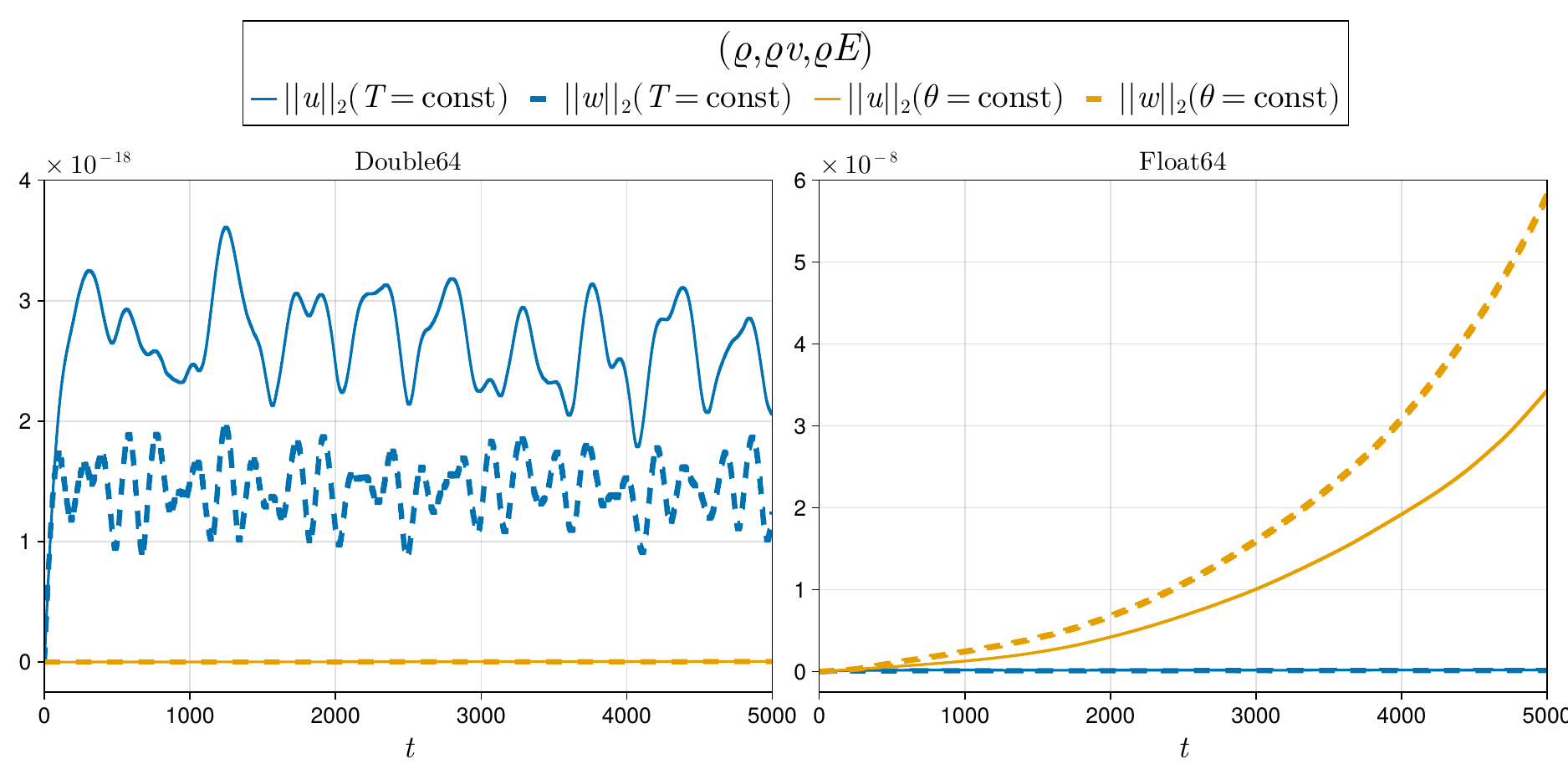}
    \caption{$L_2$ error of the horizontal and vertical velocity for the two different constant background states for the compressible Euler equation with total energy as prognostic variable.}
    \label{balance_rhoE}
\end{figure}

\begin{figure}[h!]
    \centering
    \includegraphics[width = \textwidth]{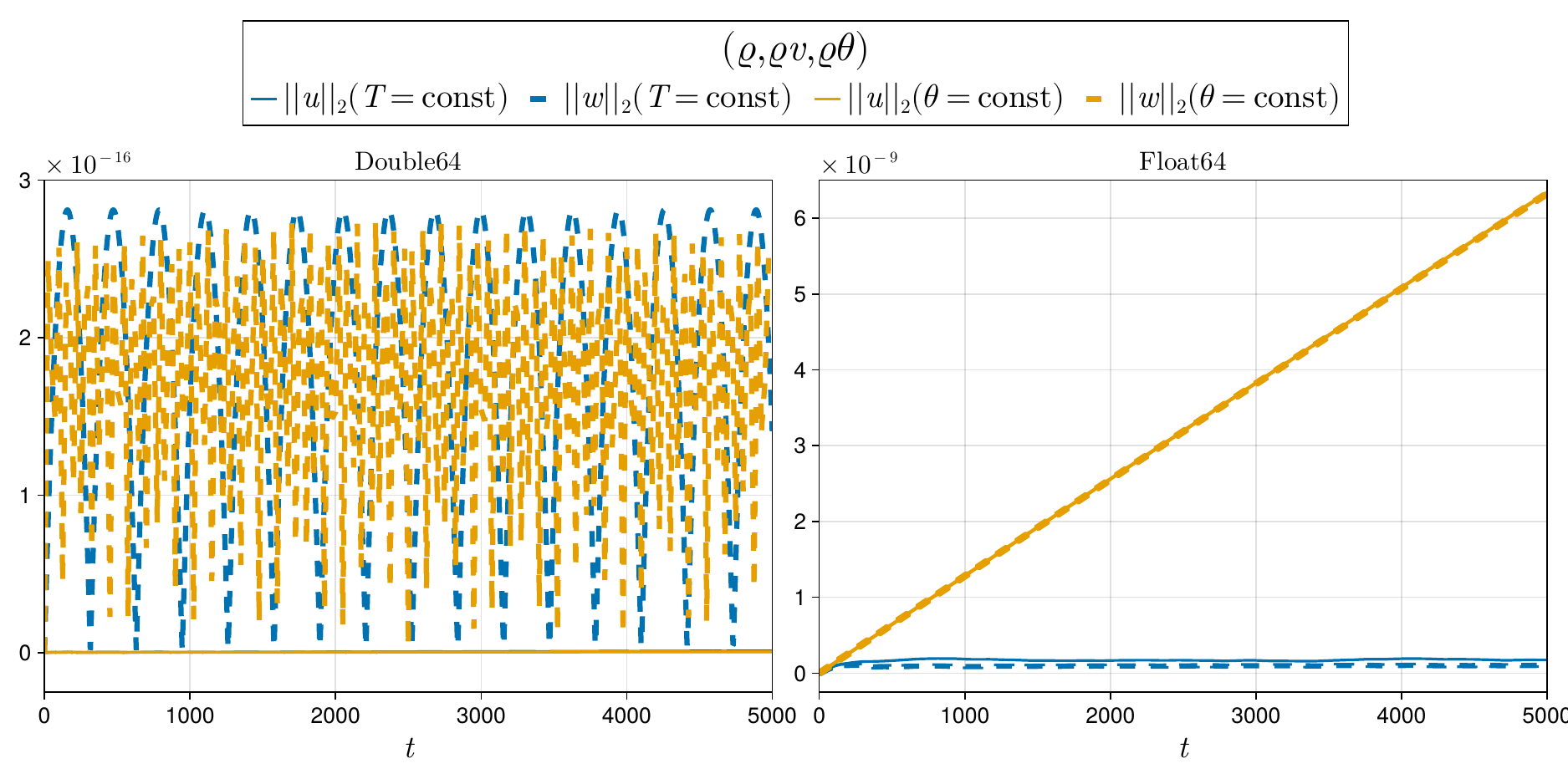}
    \caption{$L_2$ error of the horizontal and vertical velocity for the two different constant background states for the compressible Euler equation with potential temperature as prognostic variable.}
    \label{balance_rhotheta}
\end{figure}

\subsection{Inertia gravity waves}
The numerical convergence of the high-order DGSEM can be assessed via the linearized solution obtained by Baldauf et al.\ \cite{baldauf2013} for the gravity wave in a channel.
The atmosphere is initialized with an isothermal background state with a temperature value of $T_0 = \SI{250}{K}$. The constant background state is then perturbed with a warm bubble of the form
\begin{equation}
    T' = \Delta T \sin \left (\frac{ \pi z}{H} \right ) \exp \left (-\frac{(x - x_c)^2}{a^2}\right ),
\end{equation}
where $\Delta T$ is the maximum temperature perturbation, $H = \SI{10}{km}$  is the height of the channel, $L = \SI{300}{km}$ is the length of the channel, $x_c = \SI{100}{km}$ is the origin of the perturbation, $a = \SI{5}{km}$ is the radius of the bubble perturbation, and $x$/$z$ are the horizontal/vertical coordinates.
The potential $\phi = g z$, where $g = \SI{9.81}{m/s^2}$, and the solution is computed from zero to time $t = \SI{1800}{s}$. The non-conservative flux for both compressible Euler formulations is given by the isothermal well-balanced scheme~\eqref{nonconservativeT}.
Several authors~\cite{blaise2016,BALDAUF2021110635,WARUSZEWSKI2022111507} have shown that high $\Delta T$ leads to saturation error due to the non-linear effects on refined grids and high-order polynomials. Thus, we follow the same approach and set the initial temperature perturbation $\Delta T = 10^{-3}$.
The domain is discretized with a uniform grid with $N_x \times N_z$ elements, where the ratio between $\frac{\Delta x}{\Delta z} = 3$ is held constant and the $\text{CFL} = 0.1$.

\begin{figure}[h!]
    \centering
    \includegraphics[width = \textwidth]{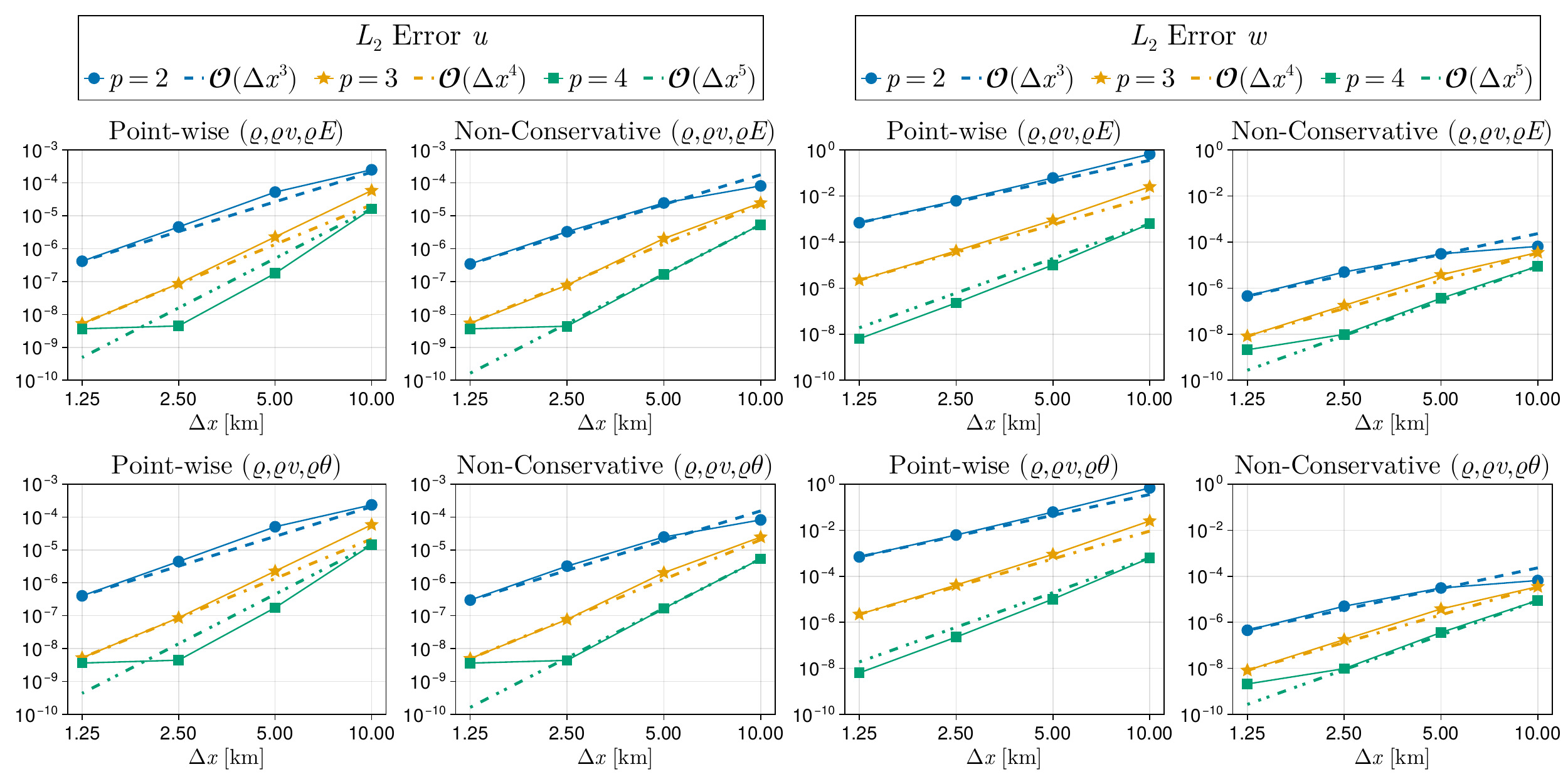}
    \caption{$L_2$ error of the horizontal (left two columns) and vertical (right two columns) velocity with point-wise (first and third column) and non-conservative (second and fourth column) discretization of the source term, for the Euler equations with total energy (top) and potential temperature (bottom).}
    \label{uw_igw}
\end{figure}

\begin{figure}[h!]
    \centering
    \includegraphics[width = \textwidth]{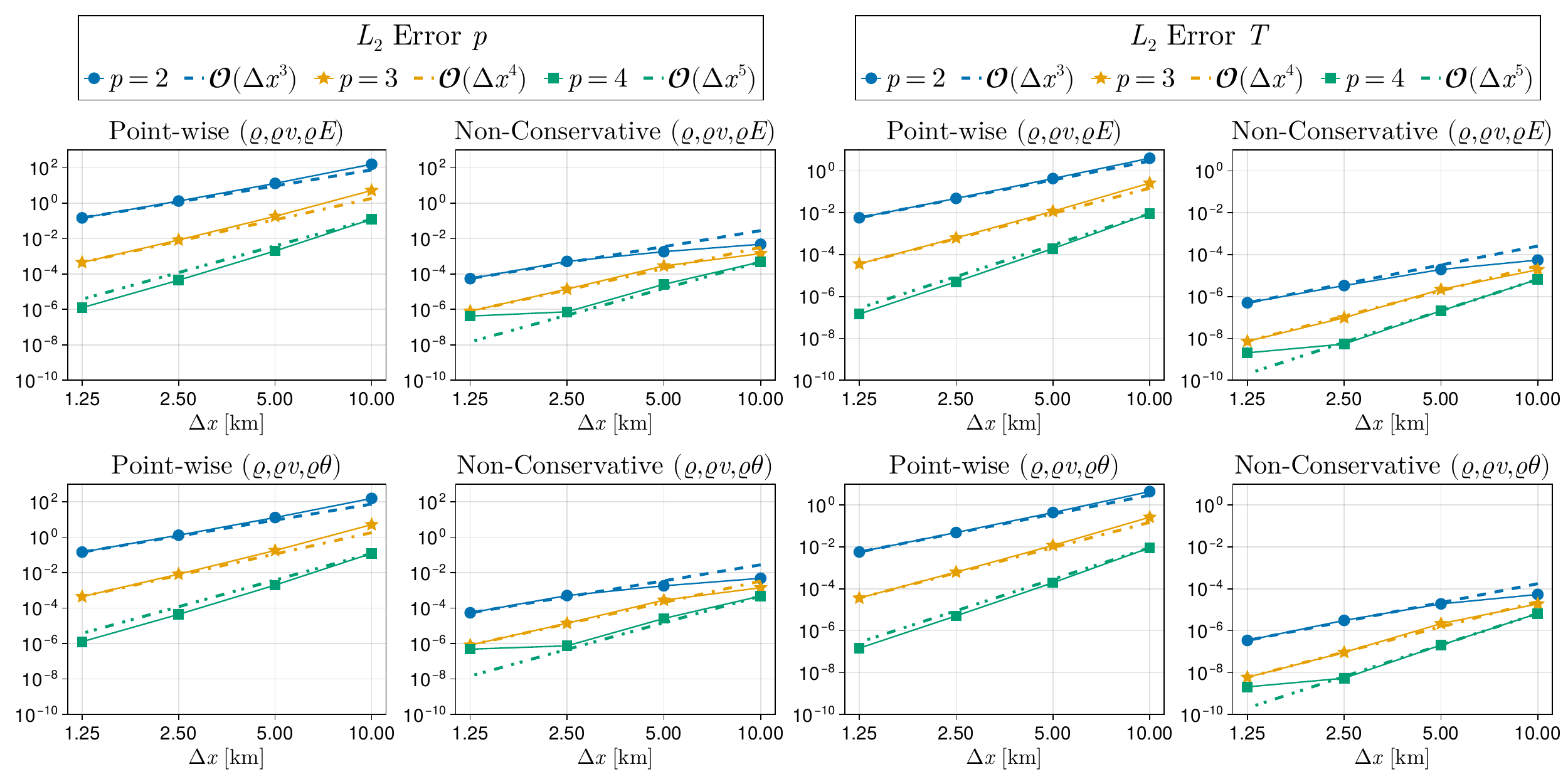}
    \caption{$L_2$ error of the pressure (left two columns) and temperature (right two columns) with point-wise (first and third column) and non-conservative (second and fourth column) discretization of the source term, for the Euler equations with total energy (top) and potential temperature (bottom).}
    \label{pT_igw}
\end{figure}

Figures~\ref{uw_igw} show the $L_2$ errors of the horizontal and vertical velocity components.
Our results demonstrate perfect agreement between the two formulations, and in this particular test case, we do not observe a clear advantage of one formulation over the other.
We would like to emphasize that the point-wise discretization of the source term, as opposed to the non-conservative approach, shows a difference of four orders of magnitude in the vertical component.
This behavior is also observed in the other variables, such as pressure and temperature, reported in Figure~\ref{pT_igw}.
Furthermore, the results are in agreement with \cite{WARUSZEWSKI2022111507}.
Additionally, as in~\cite{blaise2016,BALDAUF2021110635,WARUSZEWSKI2022111507}, saturation errors arising from nonlinear effects are present for polynomial degree 4, and they are particularly more pronounced in the horizontal velocity component, where it seems to suggest that a plateau has been reached.
The analysis has been done also considering the ETEC and EC numerical fluxes for the potential temperature and all the schemes show the same convergence rates.

The numerical fluxes and the two formulations are also compared on curvilinear mesh, with the transformation proposed by~\cite{WARUSZEWSKI2022111507}. The convergence rates are shown in Figure~\ref{wT_igw_warped}.
Compared to the point-wise discretization, the non-conservative discretization shows significantly lower error due to the extension of the well-balanced property to curvilinear meshes.

\begin{figure}[h!]
    \centering
    \includegraphics[width = \textwidth]{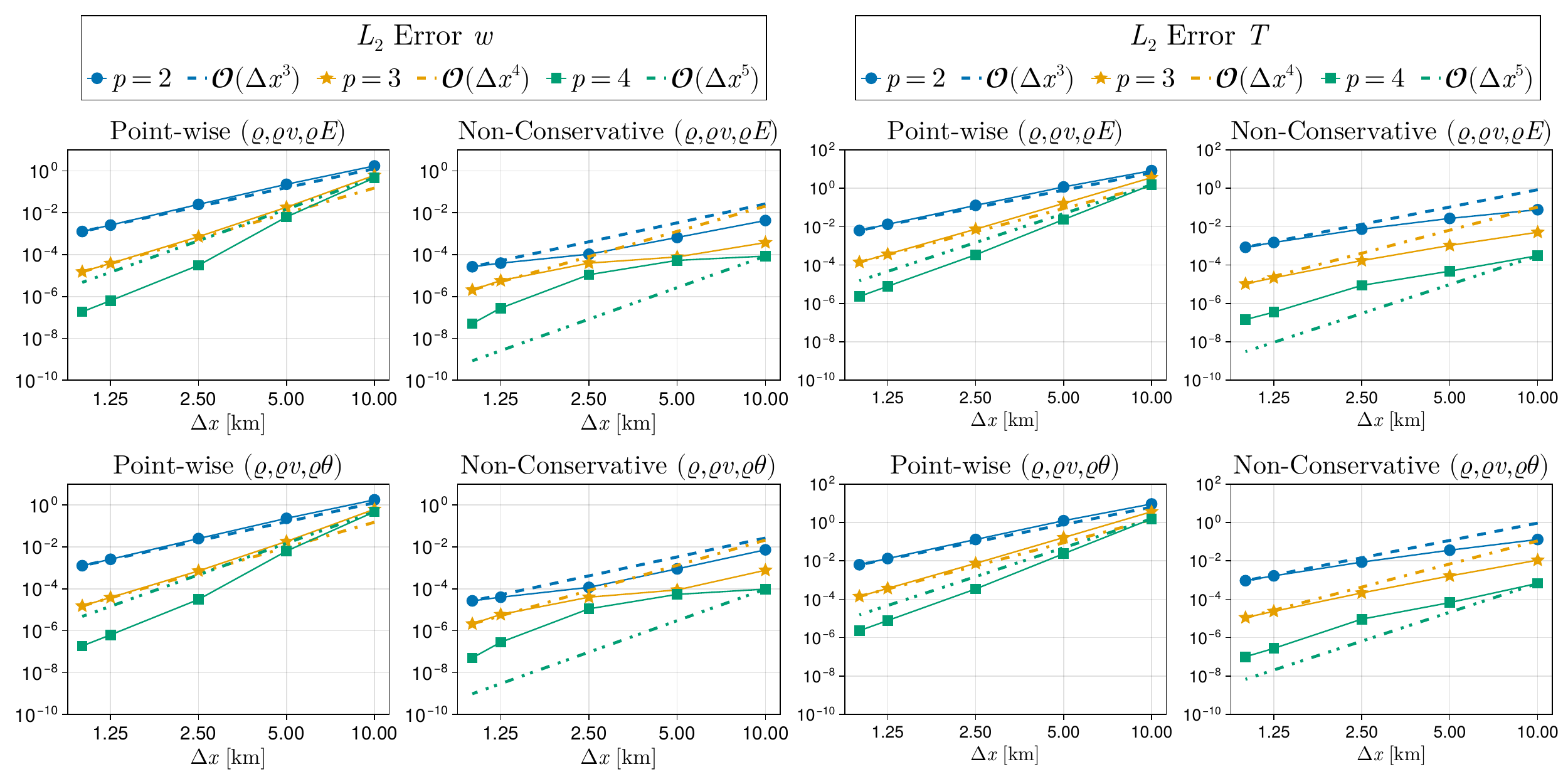}
    \caption{$L_2$ error of the vertical velocity (left two columns) and temperature (right two columns) with point-wise (first and third column) and non-conservative (second and fourth column) discretization of the source term, for the Euler equations with total energy (top) and potential temperature (bottom) on the warped mesh.}
    \label{wT_igw_warped}
\end{figure}

The convergence analysis has also been conducted with different numerical fluxes for the source term and having the well-balanced scheme appears to be crucial for high-order convergence of the solution, in contrast to conservation of TEC and KPEP, which play a marginal role in this particular test case. Figure~\ref{contour_igw} shows a contour comparison between the point-wise and non-conservative discretization. The oscillations observed in the point-wise source term discretization are purely numerical, resulting from the lack of a well-balanced scheme.

\begin{figure}[h!]
    \centering
    \includegraphics[width = \textwidth]{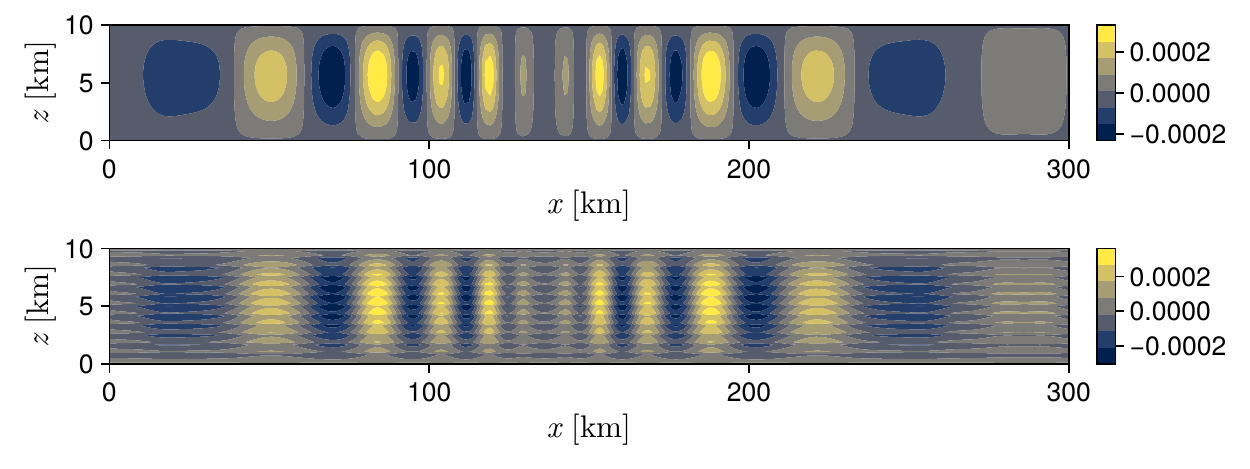}
    \caption{Contour of the vertical velocity with non-conservative (top panel) and point-wise (bottom panel) discretization of the source term, for the Euler equations with potential temperature.}
    \label{contour_igw}
\end{figure}

\subsection{Numerical results with orography}
In this section we present the numerical results for three idealized flows over orography.
To avoid spurious and unphysical solutions due to the reflections on the domain extrema, non-reflecting boundary conditions are prescribed on the top and lateral boundaries, with the introduction of Rayleigh damping profiles as in  \cite{ORLANDO2023115124,GIRFOGLIO2025106510}, which we provide here for completeness.
On the right-hand side of the momentum and either total energy or potential temperature equation, we add the source terms
\begin{equation}
\begin{aligned}
    \vec{q}_{\vec{V}} &= \varrho (\vec{V} - \vec{V}_0)( S_{\vec{V}}  + S_{h_1} + S_{h_2}),
    \\
    q_{\varrho \theta} &= \varrho (\theta - \theta_0)( S_{\vec{V}}  + S_{h_1} + S_{h_2}),
    \\
    q_{\varrho E} &= \varrho (\theta - \theta_0) K \frac{\gamma}{\gamma-1} (\varrho \theta)^{\gamma-1}( S_{\vec{V}}  + S_{h_1} + S_{h_2}) + \vec{V} \cdot \vec{q}_{\vec{V}},
\end{aligned}
\end{equation}
where the Rayleigh damping profiles $S_{\vec{V}}$, $S_{h_1}$, and $S_{h_2}$ are defined as
\begin{equation}
\begin{aligned}
    & S_{\vec{V}}= \begin{cases}0, & \text { if } \quad z<z_B \\
    \alpha \sin ^2\left[\frac{\pi}{2}\left(\frac{z-z_B}{z_T-z_B}\right)\right], & \text { if } \quad z \geq z_B\end{cases} \\
    & S_{h_1}= \begin{cases}0, & \text { if } x<x_{B} \\
    \alpha \sin ^2\left[\frac{\pi}{2}\left(\frac{x-x_{B}}{x_T-x_{B}}\right)\right], & \text { if } x \geq x_{B}\end{cases} \\
    & S_{h_2}= \begin{cases}0, & \text { if } x>-x_{B} \\
    \alpha \sin ^2\left[\frac{\pi}{2}\left(\frac{x+x_{B}}{x_B-x_T}\right)\right], & \text { if } x \leq -x_{B} \end{cases}
    \end{aligned}
\label{RayleighDampingProfile}
\end{equation}
\revB{where the physical domain is bounded by $[-x_B, x_B] \times [h(x), z_T]$, with $h(x)$ denoting the bottom topography, and $z_B$ denoting the height at which the vertical damping layer starts.}

\subsubsection{Linear hydrostatic mountain}
The first classical benchmark test case of a flow over a hill is the linear hydrostatic mountain, see \cite{ACompressibleModelfortheSimulationofMoistMountainWaves, GIRALDO20083849} for a thorough description of the test case. The mountain profile is the well-known \textit{versiera di Agnesi} given by
\begin{equation}
    h(x,z) = \frac{h_c}{1 + \left (\frac{x}{a_c} \right )^2},
    \label{versiera}
\end{equation}
where $h_c$ is the height of the mountain and $a_c$ is its half-width.
The computational domain is \revB{included in} $[-120, 120] \times [0, 30]~\mathrm{km}^2$ \revB{and bounded from below by the mountain profile $h$}; thus the mountain peak is centered at 0 with $h_c = \SI{1}{m}$ and $a_c = \SI{10}{km}$. The atmosphere is initialized with a constant background temperature $T_0 = \SI{250}{K}$ and a constant mean flow $\overline{u} = \revA{\SI{20}{m/s}}$.
The pressure is initialized with the Exner pressure $\overline{\pi}$ with the following $p = p_0 \overline{\pi}^{c_p/R}$, where $\overline{\pi} = \exp(\mathcal{N}^2/g z)$ and $\mathcal{N} = g / \sqrt{c_p T_0}$ is the Brunt-Väisälä frequency.
Since in this case $\frac{\mathcal{N} a_c}{\overline{u}} > 1 $, the flow is in hydrostatic regime \cite{GIRALDO20083849,SimpleTestsofaSemiImplicitSemiLagrangianModelon2DMountainWaveProblems}.
A no-slip boundary condition is used at the bottom and non-reflecting boundary conditions are used at the top and lateral boundaries, with $\alpha = 0.1$, $z_B = \SI{15}{km}$, $z_T = \SI{30}{km}$, $x_B = \SI{80}{km}$, and $x_T = \SI{120}{km}$.
The polynomial degree is $p = 3$, and we use $100 \times 60$ elements (horizontal $\times$ vertical direction). \revA{Here the gravity term is again discretized for both formulations with the isothermal well-balanced scheme~\eqref{nonconservativeT}.}

The quantity of interest for this benchmark is the analytical momentum flux \cite{Smith1979TheIO}
\begin{equation}
    m = \int_{-\infty}^{\infty} \overline{\varrho}(z) u'(x,z) w'(x,z) dx,
\end{equation}
where $u'$ and $w'$ are the velocity perturbations and $\overline{\varrho}$ is the density. The analytical momentum flux is given by
\begin{equation}
    m^H = - \frac{\pi}{4}\overline{\varrho}_s \overline{u}_s N h_m^2,
\end{equation}
where $\overline{\varrho}_s$ and $\overline{u}_s$ are the background surface density and velocity. Thus, the ratio $m/m^H$ has to be as close as possible to 1.
Figure~\ref{comparison_hydrostatic_125} shows the normalized momentum for both formulations with different discretizations of the source terms.
The non-conservative formulations show the same results for both sets of conserved variables and an increased robustness compared to the point-wise discretization of the source term, which shows spurious oscillations.
This behavior can also be seen in Figure~\ref{Evolution12}, where the evolution of the normalized momentum is reported.
It can be noticed that while the momentum approaches its analytical value, for the non-conservative formulations there are no spurious oscillations.
Figure~\ref{Contouruwlinear} shows the contour plots of the horizontal and vertical velocity components.
On a more refined grid the two schemes for both sets of conserved variables agree rather well and the spurious oscillations diminish (here not shown).

\begin{figure}[h!]
    \centering
    \includegraphics[scale = 0.5]{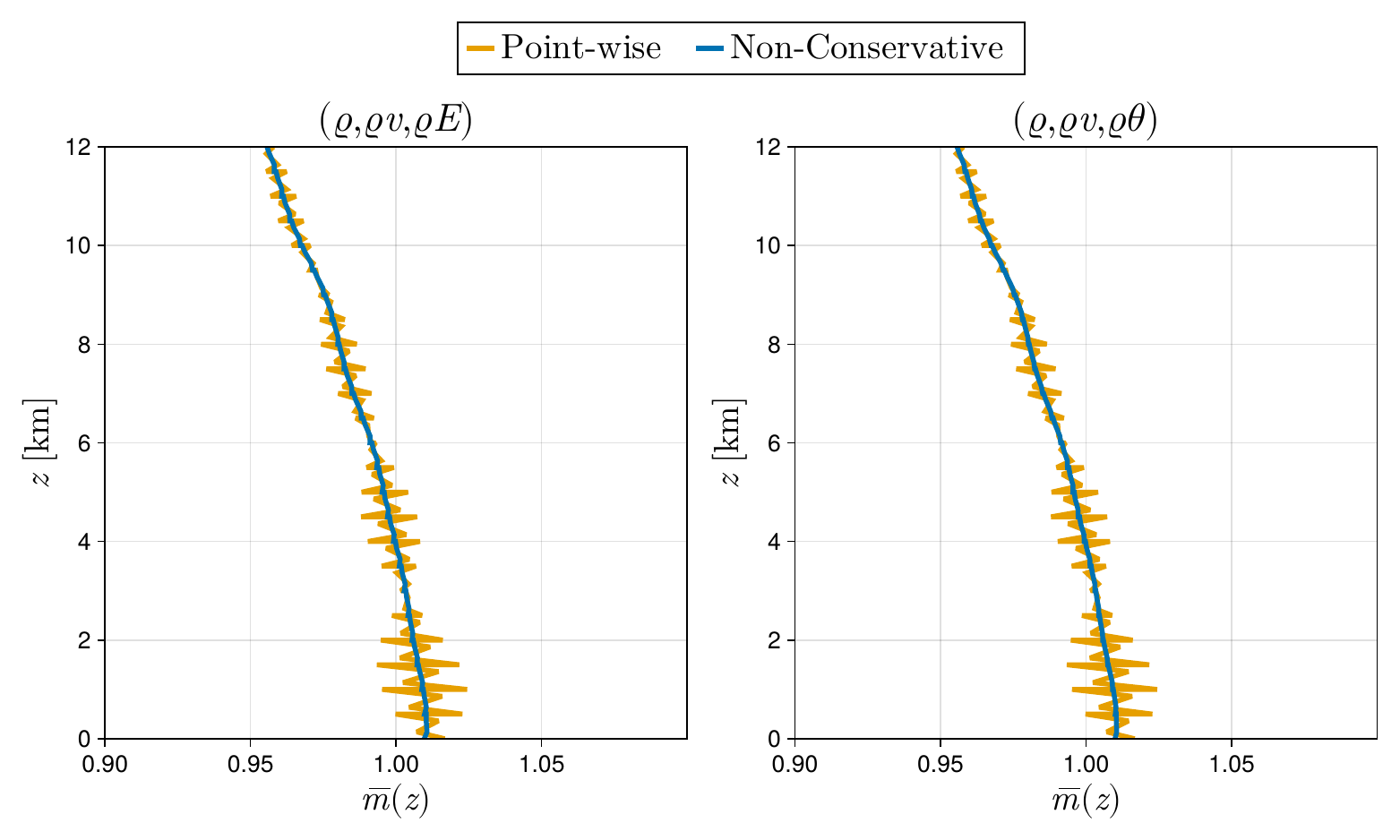}
    \caption{Comparison of the normalized momentum for the linear hydrostatic mountain test case for the 4 different formulations at $T = \SI{12.5}{h}$, with $100\times 60$ elements along the $x$ direction and $z$ direction, respectively, and $p = 3$.}
    \label{comparison_hydrostatic_125}
\end{figure}

\begin{figure}[h!]
    \centering
    \includegraphics[scale = 0.5]{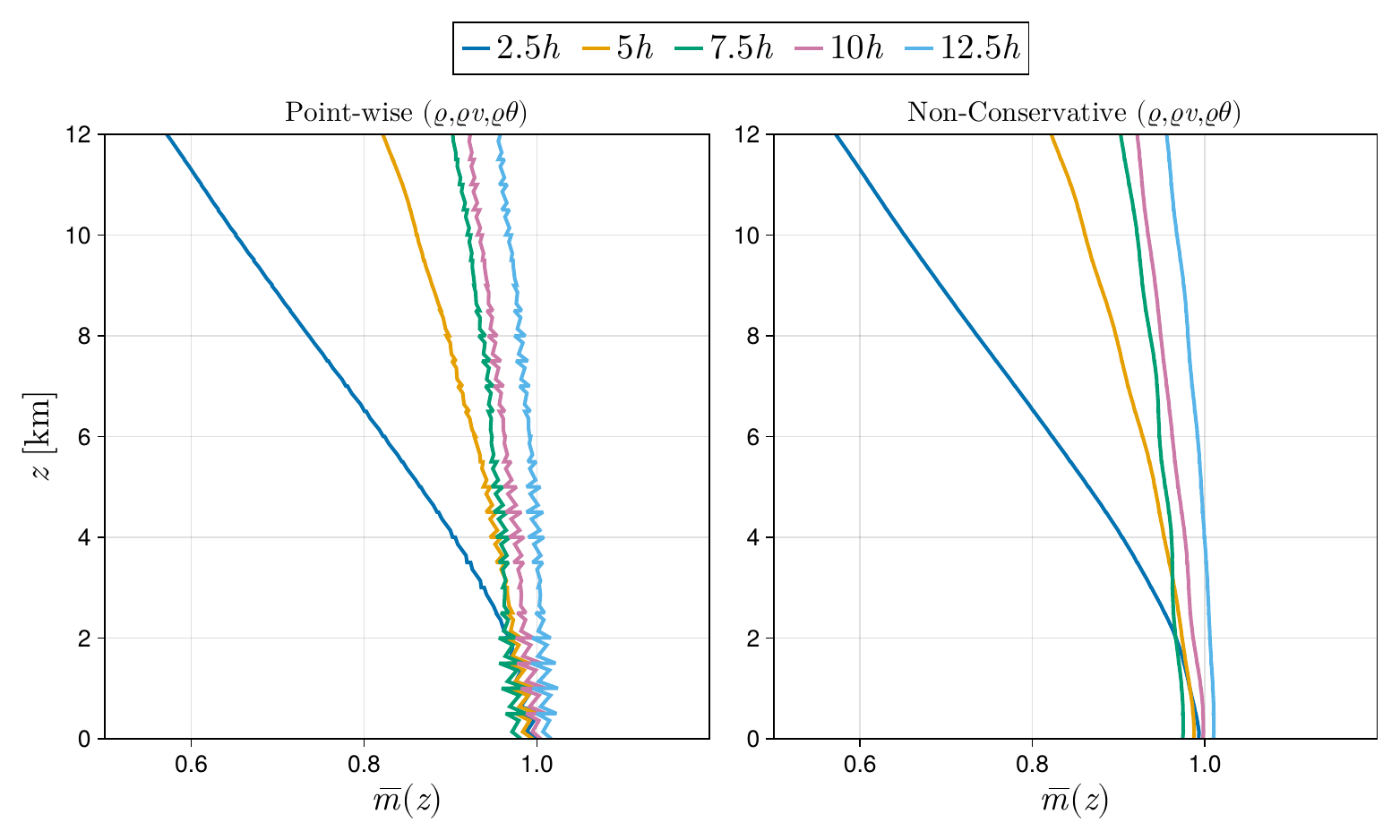}
    \caption{Evolution of the normalized momentum for the linear hydrostatic mountain with $100 \times 60$ elements along the $x$ and $z$ direction respectively and $p = 3$\revB{, at different times from $2.5$ to $12.5$ hours.}}
    \label{Evolution12}
\end{figure}

\begin{figure}[h!]
    \centering
    \includegraphics[width = \textwidth]{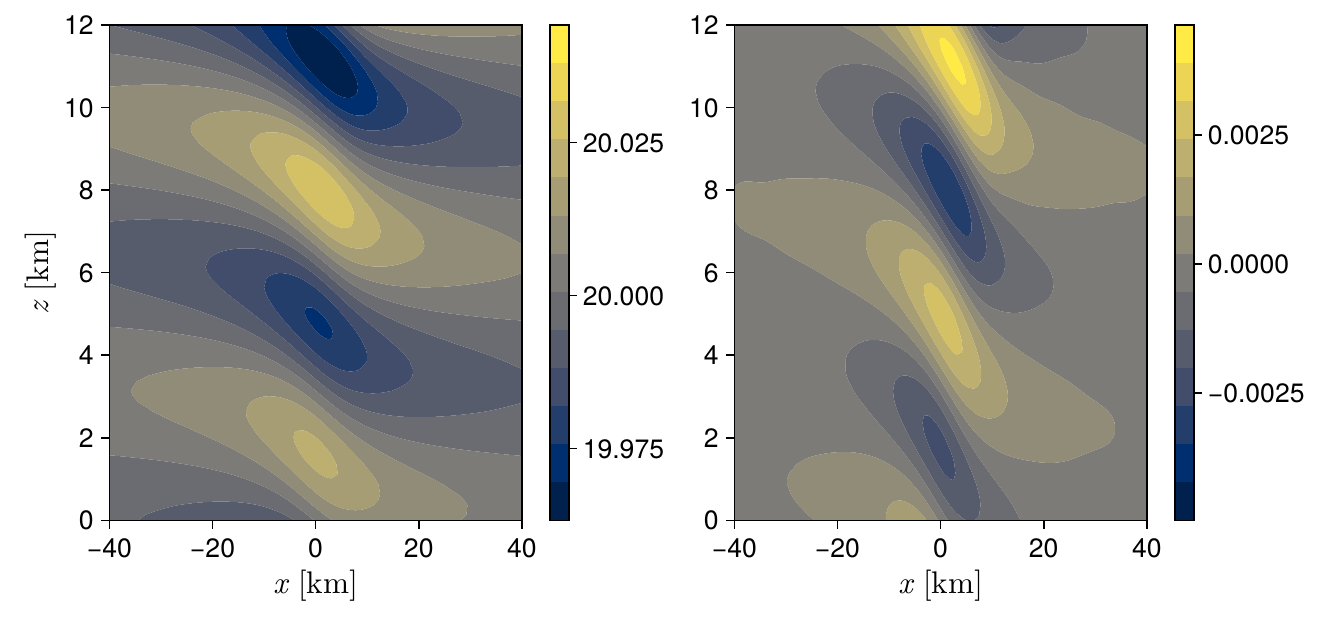}
    \caption{Contours of the horizontal (left) and vertical (right) velocity components for the linear hydrostatic mountain with $100 \times 60$ elements in the $x$ and $z$ directions, respectively and $p = 3$\revB{, for the compressible Euler equations using $\varrho E$ as prognostic variable and the nonconservative formulation of the gravity term, at the final time $T = 12.5$ hours.}}
    \label{Contouruwlinear}
\end{figure}

\subsubsection{Linear nonhydrostatic mountain}
The second classical benchmark test case focuses on linear nonhydrostatic mountain waves, see \cite{ACompressibleModelfortheSimulationofMoistMountainWaves, GIRALDO20083849} for a detailed description of the test case.
Here, the initial atmospheric state consists of a uniform horizontal mean flow $\overline{u} = \SI{10}{m/s}$ within a uniformly stratified atmosphere characterized by a constant Brunt-Väisälä frequency $\mathcal{N} = \SI{0.01}{s^{-1}}$.
The potential temperature and Exner pressure profiles are initialized following the hydrostatic balance:
\begin{equation} \overline{\pi}(z) = 1 + \frac{g^2}{c_p \theta_0 \mathcal{N}^2} \left( \exp\left(-\frac{\mathcal{N}^2}{g} z\right) - 1 \right),
\end{equation}
where $\theta_0 = \SI{280}{K}$ is the surface potential temperature.
The pressure is computed as
\begin{equation}
    p = p_0 \overline{\pi}^{c_p/R},
\end{equation}
where $R = c_p - c_v$ is the gas constant for dry air.
The potential temperature is given by
\begin{equation}
    \theta(z) = \theta_0 \exp\left(\frac{\mathcal{N}^2}{g} z\right),
\end{equation}
and the temperature is recovered by $T = \theta \overline{\pi}$.
The background density is
\begin{equation}
    \overline{\varrho} = \frac{p}{R T}.
\end{equation}
The initial velocity field is set as $u = \overline{u}$ and $w = 0$.
The mountain profile is given by the classical \textit{versiera di Agnesi} \eqref{versiera}, where $h_c = \SI{1}{m}$ is the mountain height and $a_c = \SI{1}{km}$ is the half-width.

The computational domain is $(x,z) \in [0, 144] \times [0, 30]~\mathrm{km}^2$, and the simulation time spans $t \in [0, 8]$ hours. No-flux boundary conditions are imposed at the bottom, while non-reflecting boundary conditions are used at the top and lateral boundaries. Damping layers are applied in the last \SI{15}{km} in the vertical direction and over the last \SI{40}{km} in the horizontal direction, with a damping coefficient $\alpha = 0.03$. \revA{The gravity term is here discretized with the constant potential temperature well-balanced scheme~\eqref{nonconservativetheta}}.

Since in this configuration $\mathcal{N} a_c / \overline{u} = 1$, the flow is in the nonhydrostatic regime~\cite{GIRALDO20083849,SimpleTestsofaSemiImplicitSemiLagrangianModelon2DMountainWaveProblems}. The quantity of interest for this benchmark is again the normalized momentum flux, defined analogously to the hydrostatic test case and normalized by $m^{NH} = 0.457m^H$ \cite{AnUpperBoundaryConditionPermittingInternalGravityWaveRadiationinNumericalMesoscaleModels}.
The results for the linear nonhydrostatic mountain test case are shown in Figures~\ref{comparison_nonhydrostatic_8} and \ref{Evolutionnon12}. The normalized momentum flux, computed at $T = \SI{8}{h}$ and reported in Figure~\ref{comparison_nonhydrostatic_8}, approaches the theoretical value of 1 for all formulations, similarly to what was observed in the hydrostatic case.

Figure~\ref{Evolutionnon12} visualizes the time evolution of the normalized momentum, showing values plotted at times \SI{2}{h}, \SI{4}{h}, \SI{6}{h}, and \SI{8}{h}. It can be observed that the normalized momentum progressively approaches the expected value of 1 over time. However, spurious oscillations are present, especially for the point-wise discretization of the source terms. These oscillations are less pronounced when using the non-conservative formulation, indicating a more robust behavior.
Figure~\ref{Contouruwnonlinear} shows contour plots of the horizontal and vertical velocity components.

\begin{figure}[h!]
    \centering
    \includegraphics[scale = 0.5]{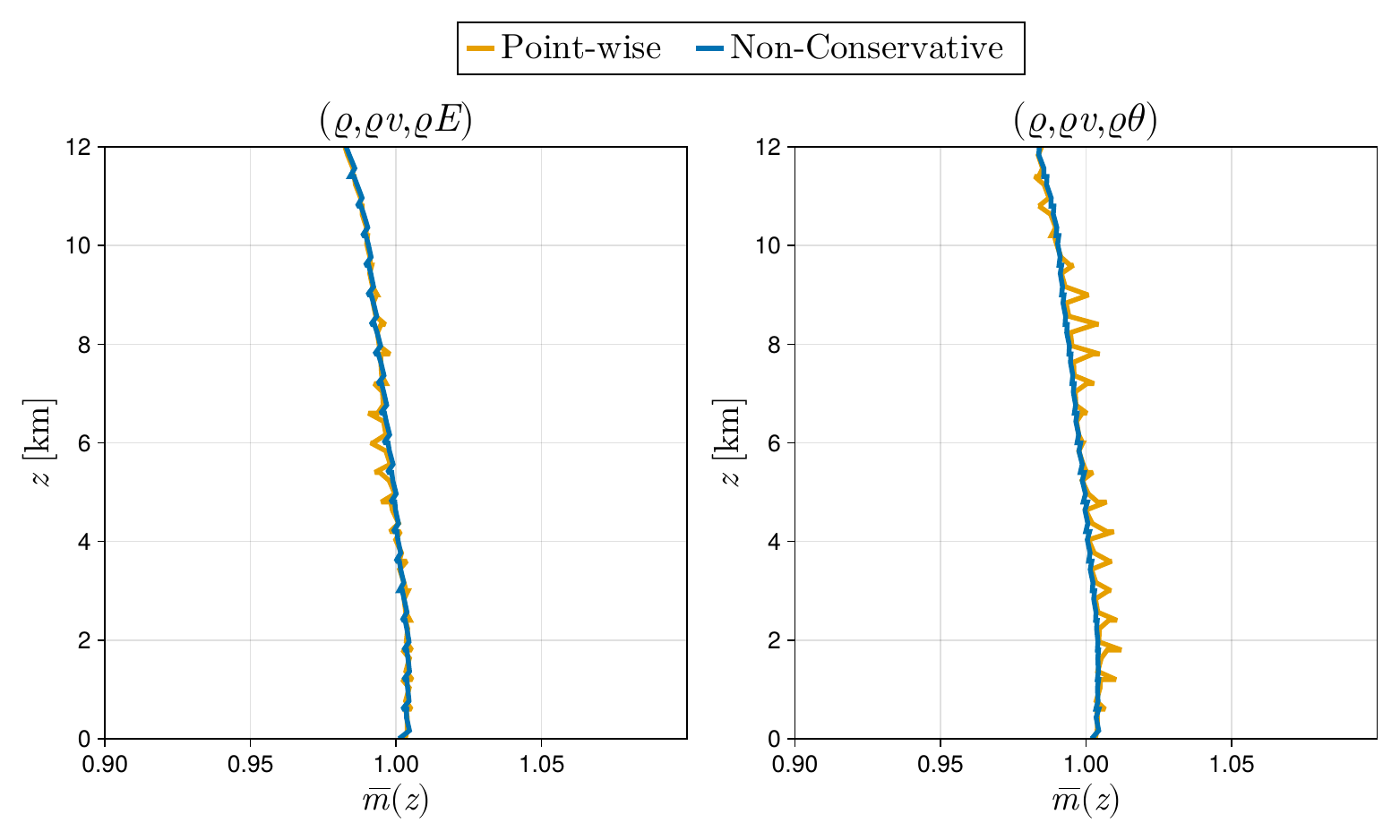}
    \caption{Comparison of the normalized momentum for the linear nonhydrostatic mountain test case for the 4 different formulations at $T = 8$h, with $200\times 50$ elements along the $x$ direction and $z$ direction, respectively, and $p = 3$.}
    \label{comparison_nonhydrostatic_8}
\end{figure}
\begin{figure}[h!]
    \centering
    \includegraphics[scale = 0.5]{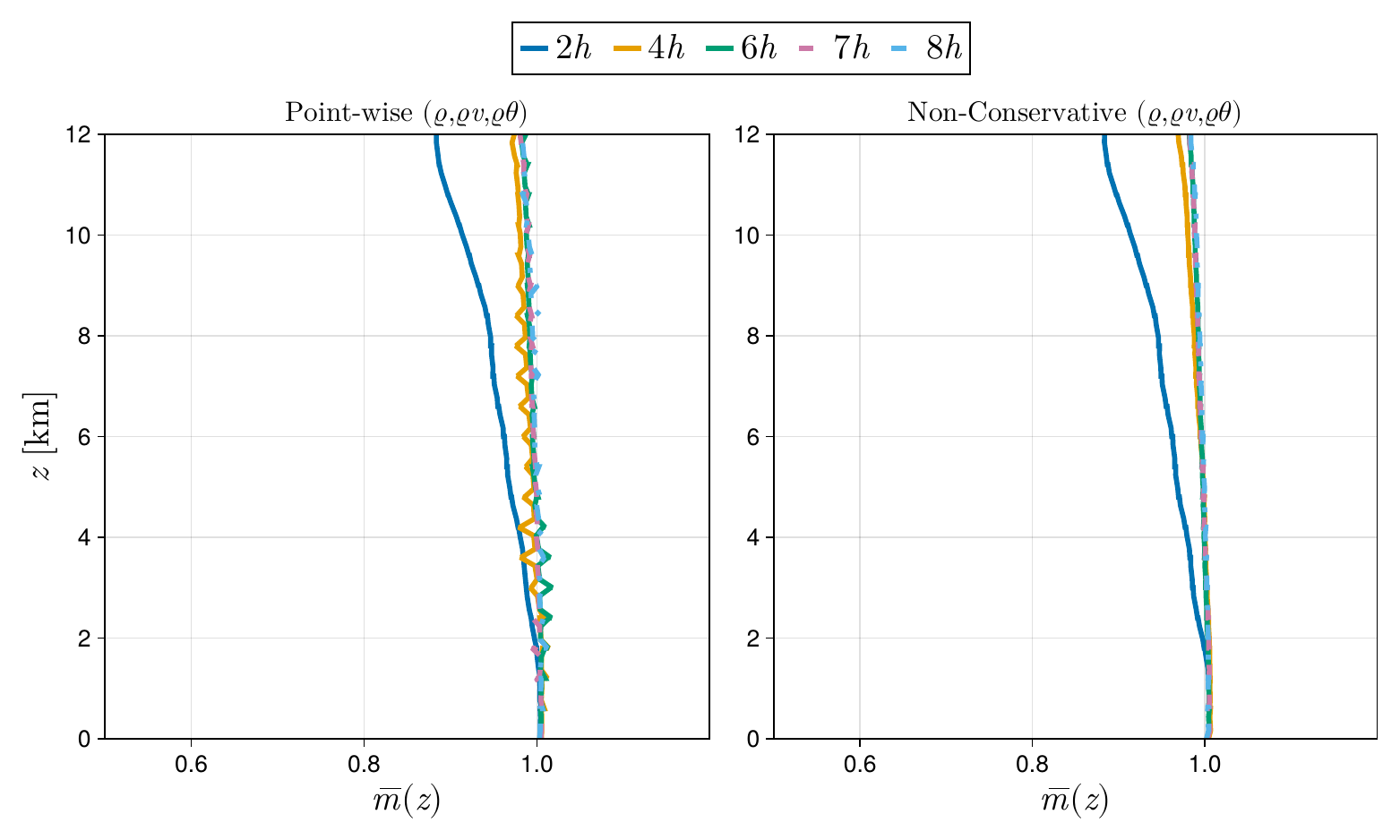}
    \caption{Evolution of the normalized momentum for the linear nonhydrostatic mountain with $200 \times 50$ elements along the $x$ and $z$ direction respectively and $p = 3$.}
    \label{Evolutionnon12}
\end{figure}

\begin{figure}[h!]
    \centering
    \includegraphics[width = \textwidth]{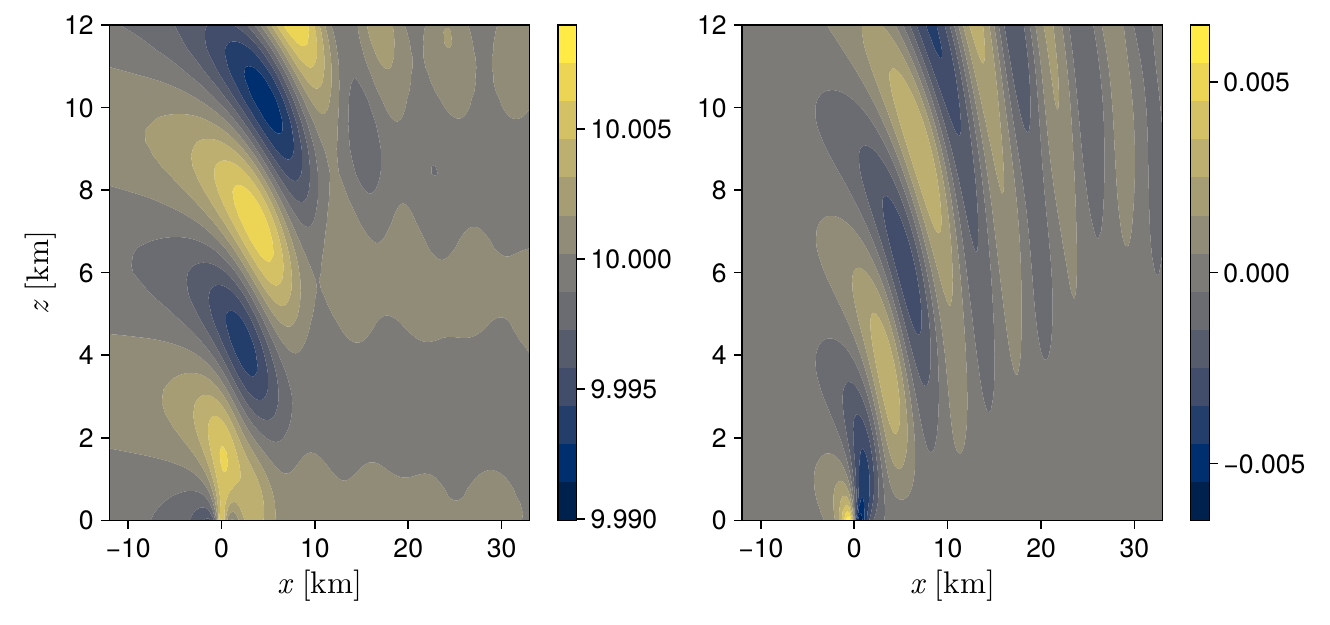}
        \caption{Contours of the horizontal (left) and vertical (right) velocity components for the linear nonhydrostatic mountain with $200 \times 50$ elements in the $x$ and $z$ directions, respectively and $p = 3$\revB{, for the compressible Euler equations using $\varrho E$ as prognostic variable and the nonconservative form of the gravity term, at the final time $T = 8$ hours.}}
    \label{Contouruwnonlinear}
\end{figure}

\subsubsection{Schär mountain}
The third benchmark test case considers the steady-state hydrostatic flow over an idealized multi-peak mountain, known as the Schär mountain \cite{ScharetAl}.
The mountain profile is defined by
\begin{equation}
    h(x) = h_c \exp\left(-\left(\frac{x}{a_c}\right)^2\right) \cos^2\left(\pi \frac{x}{\lambda_c}\right),
\end{equation}
where $h_c = \SI{250}{m}$ is the maximum mountain height, $a_c = \SI{5000}{m}$ controls the width of the Gaussian envelope, and $\lambda_c = \SI{4000}{m}$ sets the wavelength of the cosine perturbation.
The resulting orography generates five mountain peaks, symmetrically distributed around the domain center.

The computational domain is $[-\num{25000}, \num{25000}] \times [0, \num{21000}]~\mathrm{m}^2$, with a simulation time interval $t \in [0, 10]~\mathrm{h}$. No-flux boundary conditions are imposed at the bottom, while non-reflecting boundary conditions are enforced at the top and lateral boundaries using Rayleigh damping layers. Specifically, sponge layers are applied in the last \SI{8}{km} of the vertical direction and in the last \SI{5}{km} near both lateral boundaries, with a damping coefficient $\alpha = 0.03$.

The atmosphere is initialized with a uniform horizontal mean flow $\overline{u} = \SI{10}{m/s}$ and a constant Brunt–Väisälä frequency $\mathcal{N} = \SI{0.01}{s^{-1}}$. The initial thermodynamic profiles are derived under the assumption of hydrostatic balance. The Exner pressure $\overline{\pi}(z)$ is given by
\begin{equation}
    \overline{\pi}(z) = 1 + \frac{g^2}{c_p \theta_0 \mathcal{N}^2} \left( \exp\left(-\frac{\mathcal{N}^2}{g} z\right) - 1 \right),
\end{equation}
where $\theta_0 = \SI{280}{K}$ is the reference potential temperature at the surface. The pressure is then initialized as
\begin{equation}
    p = p_0 \overline{\pi}^{c_p/R},
\end{equation}
with reference pressure $p_0 = \SI{10000}{Pa}$ and gas constant $R = c_p - c_v$ for dry air. The potential temperature evolves with height as
\begin{equation}
    \theta(z) = \theta_0 \exp\left(\frac{\mathcal{N}^2}{g} z\right),
\end{equation}
and the temperature is computed as
$T(z) = \theta(z) \overline{\pi}(z)$.
The corresponding background density profile is given by the ideal gas law
\begin{equation}
    \overline{\varrho}(z) = \frac{p(z)}{R T(z)}.
\end{equation}
The initial velocity field is set uniformly as $u = \overline{u}$ and $w = 0$. \revA{We have discretized the gravity term with the constant potential
temperature well-balanced scheme~\eqref{nonconservativetheta}.}

Since in this configuration $\mathcal{N} a_c / \overline{u} > 1$, the flow is (strongly) non-hydrostatic~\cite{ScharetAl,GIRALDO20083849}. Figure~\ref{Contouruwschar} shows the contour of the horizontal and vertical velocity components at $T = \SI{5}{h}$. For this particular test case, the two formulations and the two discretization are in good agreement.

\begin{figure}[h!]
    \centering
    \includegraphics[width = \textwidth]{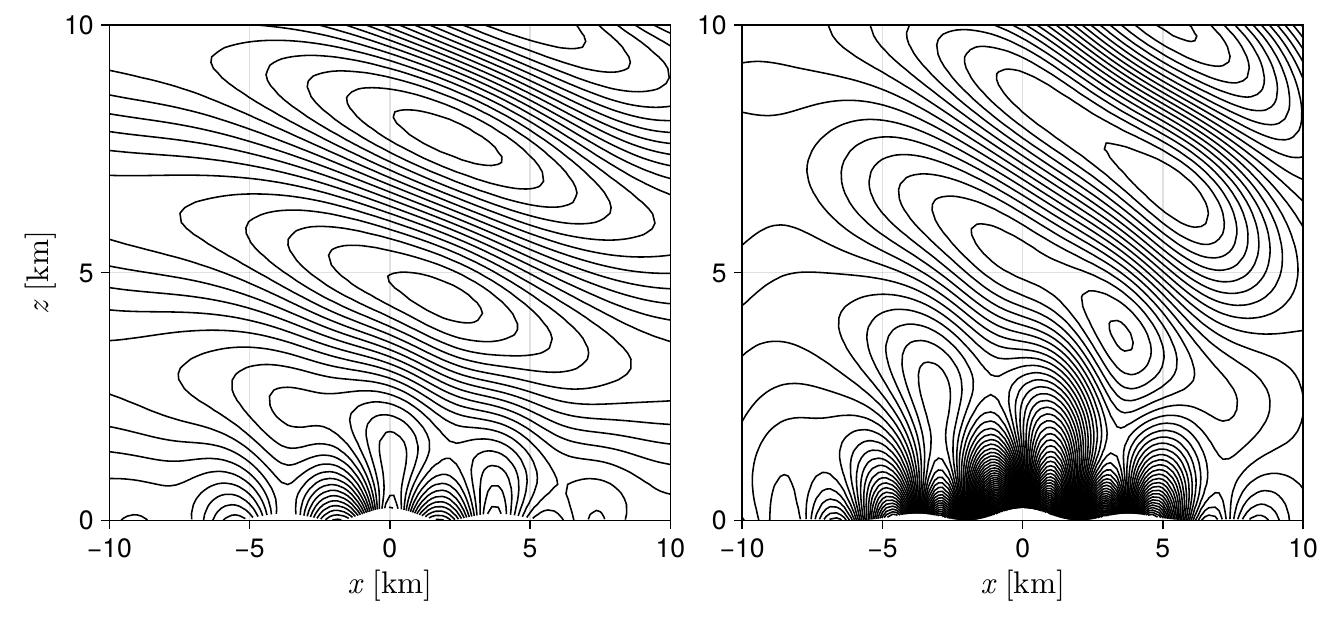}
  \caption{Contour lines of the horizontal (left) and vertical (right) velocity perturbation components for the mountain Schär test case with $100 \times 50$ elements along the $x$ and $z$ direction respectively and $p = 3$. Contour values are between $-\SI{2}{m/s}$ and $\SI{2}{m/s}$ with an interval equal to $\SI{0.2}{m/s}$ for horizontal velocity component and between $-\SI{2}{m/s}$ and $\SI{2}{m/s}$ with an interval equal to $\SI{0.05}{m/s}$ for vertical velocity component.}
    \label{Contouruwschar}
\end{figure}
\subsection{Baroclinic instability}
\label{sec:baroclinic_instability}
To further validate our formulation, we employ the baroclinic wave benchmark originally formulated by Ullrich et al.~\cite{ullrich2014} and adopted in the 2016 edition of the Dynamical Core Model Intercomparison Project (DCMIP)~\cite{ullrich2016}. This configuration has become a standard test for assessing the ability of atmospheric models to reproduce midlatitude baroclinic instability.

The initial state is a balanced, axisymmetric solution of the \revB{3D} deep-atmosphere equations, as described in Appendix~A of Ullrich et al.~\cite{ullrich2014}. A localized Gaussian perturbation is applied to the zonal wind field in the northern midlatitudes, which excites the growth of a baroclinic disturbance. In the early stages (up to about day~7) the evolution is predominantly linear, while at later times nonlinear interactions lead to steepening and eventual wave breaking.

In our study the simulation length is set to 10 days, covering both the linear growth phase and the transition to nonlinearity. The spatial discretization uses an equiangular cubed--sphere grid with $K_h = 8$ elements on the horizontal and $K_v = 4$ elements in the vertical, for a total of six panels and polynomial degree $p = 5$.

The point-wise discretization of the source term is unconditionally unstable, independently by the set of equations.
On the other hand the non-conservative formulation could run stably and without any need of filtering for long simulation time. The contour of the surface pressure after 10 days are shown in Figure~\ref{baroclinic}.

\begin{figure}[htb]
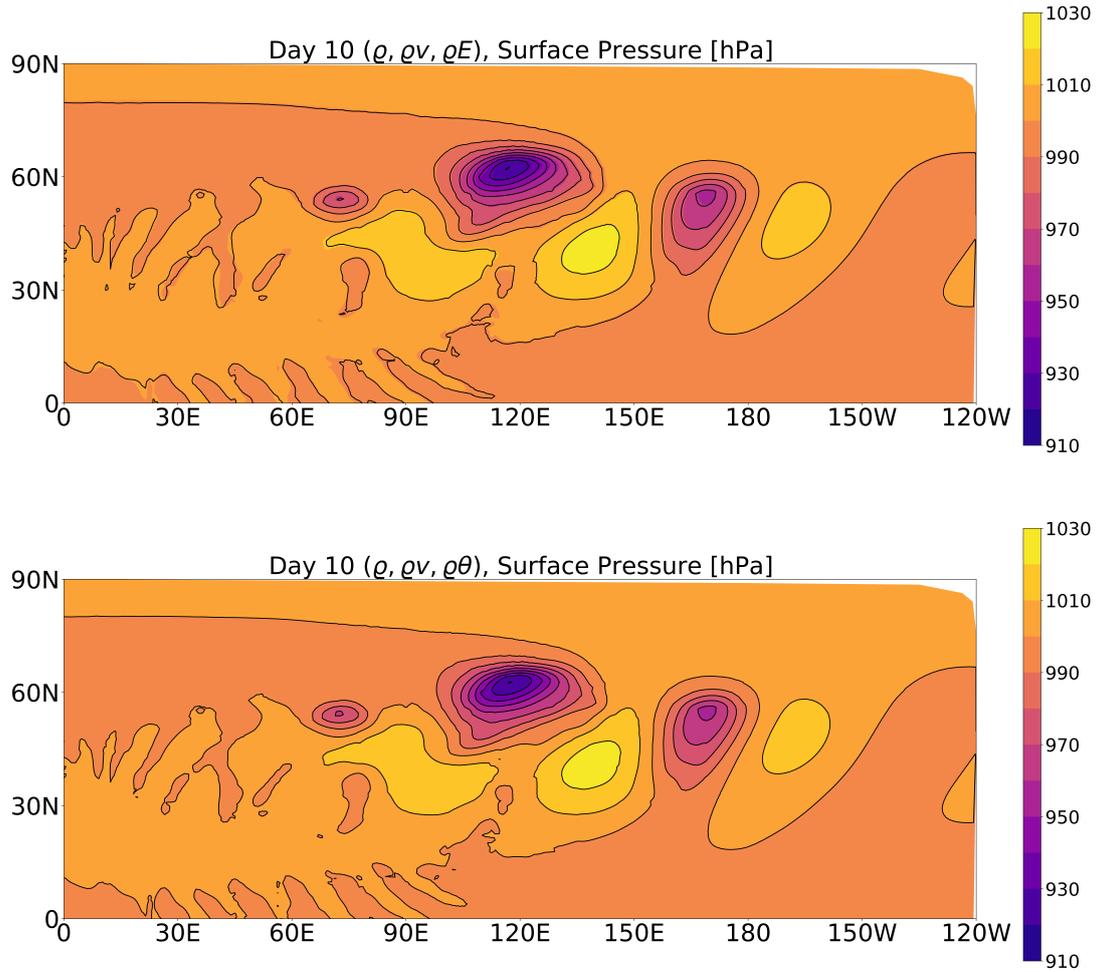

    \centering
    \begin{subfigure}{0.9\textwidth}
        \centering
        \includegraphics[width=\textwidth]{contour_pressure_euler_new_colormap.png}
        \label{baroclinic1}
    \end{subfigure}

    \vspace{0.3cm}

    \begin{subfigure}{0.9\textwidth}
        \centering
        \includegraphics[width=\textwidth]{contour_pressure_theta_new_colormap.png}
        \label{baroclinic2}
    \end{subfigure}

    \caption{Contour of the surface pressure at day 10 for the baroclinic instability test case, with polynomial degree 5.}
    \label{baroclinic}
\end{figure}

\section{Conclusions}
\label{sec:conclusions}

In this work we derived three new numerical fluxes for the compressible Euler equations, formulated with potential temperature as a primary invariant.
Depending on the choice, these fluxes can preserve thermodynamic entropy (EC), total energy (TEC), or both simultaneously (ETEC).
The EC and TEC fluxes allow a degree of freedom in the density flux. For both cases, the logarithmic mean provide the most robust formulation; the EC flux then is PEP, while the TEC flux additional satisfies the EC condition under constant pressure.
We introduced a general definition of the kinetic and potential energy preservation (KPEP) property and derived a condition for total energy conservation (TEC), when a geopotential term is included, which rely on a non-conservative discretization of the geopotential term in flux differencing form.
Furthermore, we developed a well-balanced scheme for the case of a constant background potential temperature and extended all properties to discontinous Galerkin spectral-element method (DGSEM) on arbitrary curvilinear coordinates.
The robustness and accuracy of the novel numerical fluxes and schemes have been assessed with a variety of standard atmospheric benchmark problems. The correct convergence rates were obtained, and the non-conservative approach showed up to two to four order of magnitude improvement, where the well-balancedness property was crucial to achieving such differences.
The numerical experiments demonstrated that the newly derived fluxes are both robust and accurate, performing consistently with the entropy conservative Ranocha flux for the total energy formulation.
No significant differences were observed between the two formulations, using total energy or potential temperature as prognostic variables, indicating that neither approach appears to offer a clear advantage over the other based on our test cases.
Additionally, schemes employing a non-conservative discretization of the source term consistently exhibited higher robustness and accuracy compared to the point-wise approach.

%% file: acknowledgments.tex
\section*{Acknowledgments}

This work was supported by the Max Planck Graduate Center with the
Johannes Gutenberg University of Mainz (MPGC).
MA and HR were supported by the Deutsche Forschungsgemeinschaft
(DFG, German Research Foundation, project numbers 513301895 and 528753982
as well as within the DFG priority program SPP~2410 with project number 526031774)
and the Daimler und Benz Stiftung (Daimler and Benz foundation,
project number 32-10/22).
We acknowledge support from the Mainz Institute of Multiscale Modeling (M3ODEL).

We thank Andr{\'e}s Rueda-Ram{\'i}rez for some valuable discussions on well-balanced methods and structure-preserving discretizations.

%% file: refs.bib
@misc{artiano2025structureRepro,
  title={Reproducibility repository for
         "{S}tructure-Preserving High-Order Methods for the Compressible
         {E}uler Equations in Potential Temperature Formulation for
         Atmospheric Flows"},
  author={Artiano, Marco and Knoth, Oswald and Spichtinger, Peter and Ranocha, Hendrik},
  year={2025},
  howpublished={\url{https://github.com/MarcoArtiano/2025_structure_potential_temperature}},
  doi={10.5281/zenodo.17106781}
}

@article{gmd-11-1497-2018,
authro = {Gardner, D. J. and Guerra, J. E. and Hamon, F. P. and Reynolds, D. R. and Ullrich, P. A. and Woodward, C. S.},
title = {Implicit-explicit ({IMEX}) {R}unge-{K}utta methods for non-hydrostatic atmospheric models},
journal = {Geoscientific Model Development},
volume = {11},
year = {2018},
number = {4},
pages = {1497--1515},
doi = {10.5194/gmd-11-1497-2018}
}

@article{Mark2020,
author = {Taylor, Mark A. and Guba, Oksana and Steyer, Andrew and Ullrich, Paul A. and Hall, David M. and Eldred, Christopher},
title = {An Energy Consistent Discretization of the Nonhydrostatic Equations in Primitive Variables},
journal = {Journal of Advances in Modeling Earth Systems},
volume = {12},
number = {1},
pages = {e2019MS001783},
doi = {https://doi.org/10.1029/2019MS001783},
year = {2020}
}

@article{WimmerGolo2021,
     author = {Wimmer, Golo A. and Cotter, Colin J. and Bauer, Werner},
     title = {Energy conserving {SUPG} methods for compatible finite element schemes in numerical weather prediction},
     journal = {The SMAI Journal of computational mathematics},
     pages = {267--300},
     year = {2021},
     publisher = {Soci\'et\'e de Math\'ematiques Appliqu\'ees et Industrielles},
     volume = {7},
     doi = {10.5802/smai-jcm.77},
     language = {en},
}

@article{Gassmann2013,
author = {Gassmann, Almut},
title = {A global hexagonal C-grid non-hydrostatic dynamical core (ICON-IAP) designed for energetic consistency},
journal = {Quarterly Journal of the Royal Meteorological Society},
volume = {139},
number = {670},
pages = {152-175},
doi = {https://doi.org/10.1002/qj.1960},
year = {2013}
}

@article{GiraldoIMEX2013,
author = {Giraldo, Francis and Kelly, James and Constantinescu, Emil},
year = {2013},
month = {10},
pages = {B1162-B1194},
title = {Implicit-Explicit Formulations of a Three-Dimensional Nonhydrostatic Unified Model of the Atmosphere (NUMA)},
volume = {35},
journal = {SIAM Journal on Scientific Computing},
doi = {10.1137/120876034}
}

@article{ICON2015,
author = {Zängl, Günther and Reinert, Daniel and Rípodas, Pilar and Baldauf, Michael},
title = {The ICON (ICOsahedral Non-hydrostatic) modelling framework of DWD and MPI-M: Description of the non-hydrostatic dynamical core},
journal = {Quarterly Journal of the Royal Meteorological Society},
volume = {141},
number = {687},
pages = {563-579},
doi = {10.1002/qj.2378},
year = {2015}
}

@techreport{wrf_version,
author = {Skamarock, William and Klemp, Joseph and Dudhia, Jimy and Gill, David and Liu, Zhiquan and Berner, Judith and Wang, Wei and Powers, Jordan and Duda, Michael and Barker, Dale and Huang, Xiang-Yu},
year = {2019},
month = {03},
institution = {National Center for Atmospheric Research},
title = {A Description of the Advanced Research WRF Model Version 4},
type = {NCAR Technical Note},
number = {NCAR/TN-556+STR},
}

@article {ABlendedSoundprooftoCompressibleNumericalModelforSmalltoMesoscaleAtmosphericDynamics,
      author = "Tommaso Benacchio and Warren P. O’Neill and Rupert Klein",
      title = "A Blended Soundproof-to-Compressible Numerical Model for Small- to Mesoscale Atmospheric Dynamics",
      journal = "Monthly Weather Review",
      year = "2014",
      publisher = "American Meteorological Society",
      address = "Boston MA, USA",
      volume = "142",
      number = "12",
      doi = "10.1175/MWR-D-13-00384.1",
      pages=      "4416 - 4438",
      }

@article{BISPEN2017222,
title = {Asymptotic preserving IMEX finite volume schemes for low Mach number Euler equations with gravitation},
journal = {Journal of Computational Physics},
volume = {335},
pages = {222-248},
year = {2017},
issn = {0021-9991},
doi = {10.1016/j.jcp.2017.01.020},
author = {Georgij Bispen and Mária Lukáčová-Medvid'ová and Leonid Yelash}
}

@article{Tadmor2003,
  author = {E. Tadmor},
  title = {Entropy stability theory for difference approximations of nonlinear conservation laws and related time-dependent problems},
  journal = {Acta Numerica},
  volume = {12},
  year = {2003},
  pages = {451--512},
  doi = {10.1017/S0962492902000156}
}

@article{FJORDHOLM20115587,
title = {Well-balanced and energy stable schemes for the shallow water equations with discontinuous topography},
journal = {Journal of Computational Physics},
volume = {230},
number = {14},
pages = {5587-5609},
year = {2011},
issn = {0021-9991},
doi = {10.1016/j.jcp.2011.03.042},
author = {Ulrik S. Fjordholm and Siddhartha Mishra and Eitan Tadmor},
}

@article{Ranocha2017,
  author    = {Hendrik Ranocha},
  title     = {Shallow water equations: split-form, entropy stable, well-balanced, and positivity preserving numerical methods},
  journal   = {GEM - International Journal on Geomathematics},
  year      = {2017},
  volume    = {8},
  number    = {1},
  pages     = {85--133},
  doi       = {10.1007/s13137-016-0089-9},
  issn      = {1869-2680},
}

@inbook{CarpenterBook,
author = {Mark H. Carpenter and Matteo Parsani and Eric J. Nielsen and Travis C. Fisher},
title = {Towards an Entropy Stable Spectral Element Framework for Computational Fluid Dynamics},
booktitle = {54th AIAA Aerospace Sciences Meeting},
chapter = {},
pages = {},
doi = {10.2514/6.2016-1058},
}

@article{winters2018comparative,
  title={A comparative study on polynomial dealiasing and split form
         discontinuous {G}alerkin schemes for under-resolved turbulence
         computations},
  author={Winters, Andrew R and Moura, Rodrigo C and Mengaldo, Gianmarco and
          Gassner, Gregor J and Walch, Stefanie and Peiro, Joaquim and
          Sherwin, Spencer J},
  journal={Journal of Computational Physics},
  volume={372},
  pages={1--21},
  year={2018},
  publisher={Elsevier},
  doi={10.1016/j.jcp.2018.06.016}
}

@article{Carpenter2014,
author = {Carpenter, Mark H. and Fisher, Travis C. and Nielsen, Eric J. and Frankel, Steven H.},
title = {Entropy Stable Spectral Collocation Schemes for the {N}avier--{S}tokes Equations: Discontinuous Interfaces},
journal = {SIAM Journal on Scientific Computing},
volume = {36},
number = {5},
pages = {B835-B867},
year = {2014},
doi = {10.1137/130932193}
}

@article{FISHER2013518,
title = {High-order entropy stable finite difference schemes for nonlinear conservation laws: Finite domains},
journal = {Journal of Computational Physics},
volume = {252},
pages = {518-557},
year = {2013},
issn = {0021-9991},
doi = {10.1016/j.jcp.2013.06.014},
author = {Travis C. Fisher and Mark H. Carpenter}
}

@article{LIU2025114095,
title = {Structure-preserving nodal {DG} method for the {E}uler equations with gravity: well-balanced, entropy stable, and positivity preserving},
journal = {Journal of Computational Physics},
volume = {537},
pages = {114095},
year = {2025},
issn = {0021-9991},
doi = {10.1016/j.jcp.2025.114095},
author = {Yuchang Liu and Wei Guo and Yan Jiang and Mengping Zhang}
}

@article{MENGALDO201556,
title = {Dealiasing techniques for high-order spectral element methods on regular and irregular grids},
journal = {Journal of Computational Physics},
volume = {299},
pages = {56-81},
year = {2015},
issn = {0021-9991},
doi = {10.1016/j.jcp.2015.06.032},
author = {G. Mengaldo and D. {De Grazia} and D. Moxey and P.E. Vincent and S.J. Sherwin}
}

@Article{CiCP-27-5,
author = {Yu, Jian and S., Jan, Hesthaven},
title = {A Study of Several Artificial Viscosity Models within the Discontinuous {G}alerkin Framework},
journal = {Communications in Computational Physics},
year = {2020},
volume = {27},
number = {5},
pages = {1309--1343},
issn = {1991-7120},
doi = {10.4208/cicp.OA-2019-0118},
}

@article{ULLRICH2018427,
title = {Impact and importance of hyperdiffusion on the spectral element method: A linear dispersion analysis},
journal = {Journal of Computational Physics},
volume = {375},
pages = {427-446},
year = {2018},
issn = {0021-9991},
doi = {10.1016/j.jcp.2018.06.035},
author = {Paul A. Ullrich and Daniel R. Reynolds and Jorge E. Guerra and Mark A. Taylor}
}

@article{Kopriva2006,
  author    = {David A. Kopriva},
  title     = {Metric Identities and the Discontinuous Spectral Element Method on Curvilinear Meshes},
  journal   = {Journal of Scientific Computing},
  year      = {2006},
  volume    = {26},
  number    = {3},
  pages     = {301--327},
  doi       = {10.1007/s10915-005-9070-8},
  issn      = {1573-7691}
}

@article{Tadmor1987,
  author = {E. Tadmor},
  title = {The numerical viscosity of entropy stable schemes for systems of conservation laws. {I}},
  journal = {Mathematics of Computation},
  volume = {49},
  number = {179},
  year = {1987},
  pages = {91--103},
  doi = {10.1090/S0025-5718-1987-0890255-3}
}

@article{Ranocha2018,
  author = {H. Ranocha},
  title = {Comparison of Some Entropy Conservative Numerical Fluxes for the {E}uler Equations},
  journal = {Journal of Scientific Computing},
  volume = {76},
  number = {1},
  pages = {216--242},
  year = {2018},
  doi = {10.1007/s10915-017-0618-1}
}

@article{ECChandra,
author = {Chandrashekar, Praveen},
year = {2012},
month = {09},
pages = {},
title = {Kinetic Energy Preserving and Entropy Stable Finite Volume Schemes for Compressible {E}uler and {N}avier-{S}tokes Equations},
volume = {14},
journal = {Communications in Computational Physics},
doi = {10.4208/cicp.170712.010313a}
}

@article{DERIGS2017624,
title = {A novel averaging technique for discrete entropy-stable dissipation operators for ideal {MHD}},
journal = {Journal of Computational Physics},
volume = {330},
pages = {624-632},
year = {2017},
issn = {0021-9991},
doi = {10.1016/j.jcp.2016.10.055},
author = {Dominik Derigs and Andrew R. Winters and Gregor J. Gassner and Stefanie Walch}
}

@article{ismail2009affordable,
  title={Affordable, entropy-consistent {E}uler flux functions {II}:
         {E}ntropy production at shocks},
  author={Ismail, Farzad and Roe, Philip L},
  journal={Journal of Computational Physics},
  volume={228},
  number={15},
  pages={5410--5436},
  year={2009},
  publisher={Elsevier},
  doi={10.1016/j.jcp.2009.04.021}
}

@article{Jameson2008,
  author = {A. Jameson},
  title = {Formulation of Kinetic Energy Preserving Conservative Schemes for Gas Dynamics and Direct Numerical Simulation of One-Dimensional Viscous Compressible Flow in a Shock Tube Using Entropy and Kinetic Energy Preserving Schemes},
  journal = {Journal of Scientific Computing},
  volume = {34},
  number = {2},
  pages = {188--208},
  year = {2008},
  doi = {10.1007/s10915-007-9172-6}
}

@article{Kuya2018,
  author = {Y. Kuya and K. Totani and S. Kawai},
  title = {Kinetic energy and entropy preserving schemes for compressible flows by split convective forms},
  journal = {Journal of Computational Physics},
  volume = {375},
  pages = {823--853},
  year = {2018},
  doi = {10.1016/j.jcp.2018.08.058}
}

@incollection{Ranocha2020,
  author = {H. Ranocha},
  title = {Entropy Conserving and Kinetic Energy Preserving Numerical Methods for the {E}uler Equations Using Summation-by-Parts Operators},
  booktitle = {Spectral and High Order Methods for Partial Differential Equations ICOSAHOM 2018},
  editor = {S. J. Sherwin and D. Moxey and J. Peiró and P. E. Vincent and C. Schwab},
  series = {Lecture Notes in Computational Science and Engineering},
  volume = {134},
  pages = {525--535},
  publisher = {Springer},
  address = {Cham},
  year = {2020},
  month = {Aug.},
  doi = {10.1007/978-3-030-39647-3_42}
}

@article{Ranocha2022,
  author = {H. Ranocha and G. J. Gassner},
  title = {Preventing Pressure Oscillations Does Not Fix Local Linear Stability Issues of Entropy-Based Split-Form High-Order Schemes},
  journal = {Communications on Applied Mathematics and Computation},
  volume = {4},
  number = {3},
  pages = {880--903},
  year = {2022},
  month = {Sep.},
  doi = {10.1007/s42967-021-00148-z}
}

@article{gassner2016split,
  title={Split Form Nodal Discontinuous {G}alerkin Schemes with
         Summation-By-Parts Property for the Compressible {E}uler
         Equations},
  author={Gassner, Gregor Josef and Winters, Andrew Ross and
          Kopriva, David A},
  journal={Journal of Computational Physics},
  volume={327},
  pages={39--66},
  year={2016},
  publisher={Elsevier},
  doi={10.1016/j.jcp.2016.09.013}
}

@article{Souza,
author = {Souza, A. N. and He, J. and Bischoff, T. and Waruszewski, M. and Novak, L. and Barra, V. and Gibson, T. and Sridhar, A. and Kandala, S. and Byrne, S. and Wilcox, L. C. and Kozdon, J. and Giraldo, F. X. and Knoth, O. and Marshall, J. and Ferrari, R. and Schneider, T.},
title = {The Flux-Differencing Discontinuous {G}alerkin Method Applied to an Idealized Fully Compressible Nonhydrostatic Dry Atmosphere},
journal = {Journal of Advances in Modeling Earth Systems},
volume = {15},
number = {4},
pages = {e2022MS003527},
doi = {10.1029/2022MS003527},
note = {e2022MS003527 2022MS003527},
year = {2023}
}

@article{KENNEDY20081676,
title = {Reduced aliasing formulations of the convective terms within the {N}avier–{S}tokes equations for a compressible fluid},
journal = {Journal of Computational Physics},
volume = {227},
number = {3},
pages = {1676-1700},
year = {2008},
issn = {0021-9991},
doi = {10.1016/j.jcp.2007.09.020},
author = {Christopher A. Kennedy and Andrea Gruber}
}

@article{WARUSZEWSKI2022111507,
title = {Entropy stable discontinuous {G}alerkin methods for balance laws in non-conservative form: Applications to the {E}uler equations with gravity},
journal = {Journal of Computational Physics},
volume = {468},
pages = {111507},
year = {2022},
issn = {0021-9991},
doi = {10.1016/j.jcp.2022.111507},
author = {Maciej Waruszewski and Jeremy E. Kozdon and Lucas C. Wilcox and Thomas H. Gibson and Francis X. Giraldo}
}

@article{RUEDARAMIREZ2024112607,
title = {A flux-differencing formula for split-form summation by parts discretizations of non-conservative systems: Applications to subcell limiting for magneto-hydrodynamics},
journal = {Journal of Computational Physics},
volume = {496},
pages = {112607},
year = {2024},
issn = {0021-9991},
doi = {10.1016/j.jcp.2023.112607},
author = {Andrés M. Rueda-Ramírez and Gregor J. Gassner}
}

@ARTICLE{Shima2021,
       author = {{Shima}, Nao and {Kuya}, Yuichi and {Tamaki}, Yoshiharu and {Kawai}, Soshi},
        title = "{Preventing spurious pressure oscillations in split convective form discretization for compressible flows}",
      journal = {Journal of Computational Physics},
         year = 2021,
        month = feb,
       volume = {427},
          eid = {110060},
        pages = {110060},
          doi = {10.1016/j.jcp.2020.110060}
}

@article{DEMICHELE2023112439,
title = {Asymptotically entropy-conservative and kinetic-energy preserving numerical fluxes for compressible {E}uler equations},
journal = {Journal of Computational Physics},
volume = {492},
pages = {112439},
year = {2023},
issn = {0021-9991},
doi = {10.1016/j.jcp.2023.112439},
author = {Carlo {De Michele} and Gennaro Coppola}
}

@phdthesis{ranocha2018thesis,
  title={Generalised Summation-by-Parts Operators and Entropy Stability of
         Numerical Methods for Hyperbolic Balance Laws},
  author={Ranocha, Hendrik},
  year={2018},
  month={02},
  school={TU~Braunschweig}
}

@article{baldauf2013,
author = {Baldauf, Michael and Brdar, Slavko},
title = {An analytic solution for linear gravity waves in a channel as a test for numerical models using the non-hydrostatic, compressible {E}uler equations},
journal = {Quarterly Journal of the Royal Meteorological Society},
volume = {139},
number = {677},
pages = {1977-1989},
doi = {10.1002/qj.2105},
year = {2013}
}

@article{blaise2016,
author = {Blaise, Sébastien and Lambrechts, Jonathan and Deleersnijder, Eric},
title = {A stabilization for three-dimensional discontinuous {G}alerkin discretizations applied to nonhydrostatic atmospheric simulations},
journal = {International Journal for Numerical Methods in Fluids},
volume = {81},
number = {9},
pages = {558-585},
doi = {10.1002/fld.4197},
year = {2016}
}

@article{BALDAUF2021110635,
title = {A horizontally explicit, vertically implicit ({HEVI}) discontinuous {G}alerkin scheme for the 2-dimensional {E}uler and {N}avier-{S}tokes equations using terrain-following coordinates},
journal = {Journal of Computational Physics},
volume = {446},
pages = {110635},
year = {2021},
issn = {0021-9991},
doi = {10.1016/j.jcp.2021.110635},
author = {Michael Baldauf}
}

@article{GIRALDO20083849,
title = {A study of spectral element and discontinuous {G}alerkin methods for the Navier–Stokes equations in nonhydrostatic mesoscale atmospheric modeling: Equation sets and test cases},
journal = {Journal of Computational Physics},
volume = {227},
number = {8},
pages = {3849-3877},
year = {2008},
issn = {0021-9991},
doi = {10.1016/j.jcp.2007.12.009},
author = {F.X. Giraldo and M. Restelli}
}

@article{ORLANDO2023115124,
title = {An {IMEX}-{DG} solver for atmospheric dynamics simulations with adaptive mesh refinement},
journal = {Journal of Computational and Applied Mathematics},
volume = {427},
pages = {115124},
year = {2023},
issn = {0377-0427},
doi = {10.1016/j.cam.2023.115124},
author = {Giuseppe Orlando and Tommaso Benacchio and Luca Bonaventura}
}

@article{ChandrashekarWB,
author = {Chandrashekar, Praveen and Klingenberg, Christian},
title = {A Second Order Well-Balanced Finite Volume Scheme for {E}uler Equations with Gravity},
journal = {SIAM Journal on Scientific Computing},
volume = {37},
number = {3},
pages = {B382-B402},
year = {2015},
doi = {10.1137/140984373}
}

@article{GIRFOGLIO2025106510,
  title   = {A comparative computational study of different formulations of the compressible {E}uler equations for mesoscale atmospheric flows in a finite volume framework},
  journal = {Computers \& Fluids},
  volume  = {288},
  pages   = {106510},
  year    = {2025},
  issn    = {0045-7930},
  doi     = {10.1016/j.compfluid.2024.106510},
  author  = {M. Girfoglio and A. Quaini and G. Rozza}
}

@article {SimpleTestsofaSemiImplicitSemiLagrangianModelon2DMountainWaveProblems,
      author = "Jean-Pierre  Pinty and Robert  Benoit and Evelyne  Richard and René  Laprise",
      title = "Simple Tests of a Semi-Implicit Semi-{L}agrangian Model on 2D Mountain Wave Problems",
      journal = "Monthly Weather Review",
      year = "1995",
      publisher = "American Meteorological Society",
      address = "Boston MA, USA",
      volume = "123",
      number = "10",
      doi = "10.1175/1520-0493(1995)123<3042:STOASI>2.0.CO;2",
      pages=      "3042 - 3058",
}

@article{Smith1979TheIO,
  title={The Influence of Mountains on the Atmosphere},
  author={Ronald B. Smith},
  journal={Advances in Geophysics},
  year={1979},
  volume={21},
  pages={87-230},
  doi={10.1016/S0065-2687(08)60262-9}
}

@article {AnUpperBoundaryConditionPermittingInternalGravityWaveRadiationinNumericalMesoscaleModels,
      author = "Joseph B.  Klemp and Dale R.  Durran",
      title = "An Upper Boundary Condition Permitting Internal Gravity Wave Radiation in Numerical Mesoscale Models",
      journal = "Monthly Weather Review",
      year = "1983",
      publisher = "American Meteorological Society",
      address = "Boston MA, USA",
      volume = "111",
      number = "3",
      doi = "10.1175/1520-0493(1983)111<0430:AUBCPI>2.0.CO;2",
      pages=      "430 - 444",
}

@book{godlewski2021numerical,
  author    = {Edwige Godlewski and Pierre-Arnaud Raviart},
  title     = {Numerical Approximation of Hyperbolic Systems of Conservation Laws},
  year      = {2021},
  publisher = {Springer},
  series    = {Texts in Applied Mathematics},
  volume    = {72},
  doi       = {10.1007/978-1-0716-1344-3}
}

@article{gassner2013skew,
  title={A Skew-Symmetric Discontinuous {G}alerkin Spectral Element
         Discretization and Its Relation to {SBP}-{SAT} Finite Difference
         Methods},
  author={Gassner, Gregor Josef},
  journal={SIAM Journal on Scientific Computing},
  volume={35},
  number={3},
  pages={A1233--A1253},
  year={2013},
  publisher={Society for Industrial and Applied Mathematics},
  doi={10.1137/120890144}
}

@article{GASSNER2016291,
title = {A well balanced and entropy conservative discontinuous {G}alerkin spectral element method for the shallow water equations},
journal = {Applied Mathematics and Computation},
volume = {272},
pages = {291-308},
year = {2016},
note = {Recent Advances in Numerical Methods for Hyperbolic Partial Differential Equations},
issn = {0096-3003},
doi = {10.1016/j.amc.2015.07.014},
author = {Gregor J. Gassner and Andrew R. Winters and David A. Kopriva}
}

@article{DERIGS2018420,
title = {Ideal {GLM-{MHD}}: About the entropy consistent nine-wave magnetic field divergence diminishing ideal magnetohydrodynamics equations},
journal = {Journal of Computational Physics},
volume = {364},
pages = {420-467},
year = {2018},
issn = {0021-9991},
doi = {10.1016/j.jcp.2018.03.002},
author = {Dominik Derigs and Andrew R. Winters and Gregor J. Gassner and Stefanie Walch and Marvin Bohm}
}

@article{BOHM2020108076,
title = {An entropy stable nodal discontinuous Galerkin method for the resistive MHD equations. Part I: Theory and numerical verification},
journal = {Journal of Computational Physics},
volume = {422},
pages = {108076},
year = {2020},
issn = {0021-9991},
doi = {10.1016/j.jcp.2018.06.027},
author = {Marvin Bohm and Andrew R. Winters and Gregor J. Gassner and Dominik Derigs and Florian Hindenlang and Joachim Saur}
}

@article{RUEDARAMIREZ2021110580,
title = {An entropy stable nodal discontinuous {G}alerkin method for the resistive {MHD} equations. Part II: Subcell finite volume shock capturing},
journal = {Journal of Computational Physics},
volume = {444},
pages = {110580},
year = {2021},
issn = {0021-9991},
doi = {10.1016/j.jcp.2021.110580},
author = {Andrés M. Rueda-Ramírez and Sebastian Hennemann and Florian J. Hindenlang and Andrew R. Winters and Gregor J. Gassner}
}

@Inbook{Kopriva2009,
author="Kopriva, David A.",
title="Spectral Element Methods",
bookTitle="Implementing Spectral Methods for Partial Differential Equations: Algorithms for Scientists and Engineers",
year="2009",
publisher="Springer Netherlands",
address="Dordrecht",
pages="293--354",
isbn="978-90-481-2261-5",
doi="10.1007/978-90-481-2261-5_8"
}

@incollection{winters2021dgsem,
  author    = {Winters, Andrew R. and Kopriva, David A. and Gassner, Gregor J. and Hindenlang, Florian},
  title     = {Construction of Modern Robust Nodal Discontinuous {G}alerkin Spectral Element Methods for the Compressible {N}avier--{S}tokes Equations},
  booktitle = {Efficient High-Order Discretizations for Computational Fluid Dynamics},
  pages     = {117--196},
  year      = {2021},
  publisher = {Springer},
  doi={10.1007/978-3-030-60610-7_3}
}

@article{WINTERS2017274,
title = {A uniquely defined entropy stable matrix dissipation operator for high Mach number ideal {MHD} and compressible {E}uler simulations},
journal = {Journal of Computational Physics},
volume = {332},
pages = {274-289},
year = {2017},
issn = {0021-9991},
doi = {10.1016/j.jcp.2016.12.006},
author = {Andrew R. Winters and Dominik Derigs and Gregor J. Gassner and Stefanie Walch}
}

@misc{ullrich2016,
  author       = {P.A. Ullrich and C. Jablonowski and K.A. Reed and C. Zarzycki and P.H. Lauritzen and R.D. Nair and J. Kent and A. Verlet-Banide},
  title        = {Dynamical core model intercomparison project (DCMIP2016) test case document},
  year         = {2016},
  howpublished = {\url{https://github.com/ClimateGlobalChange/DCMIP2016}}
}

@article{ullrich2014,
author = {Ullrich, Paul A. and Melvin, Thomas and Jablonowski, Christiane and Staniforth, Andrew},
title = {A proposed baroclinic wave test case for deep- and shallow-atmosphere dynamical cores},
journal = {Quarterly Journal of the Royal Meteorological Society},
volume = {140},
number = {682},
pages = {1590-1602},
doi = {10.1002/qj.2241},
year = {2014}
}

@article{Winters2020,
  author    = {Andrew R. Winters and Christof Czernik and Moritz B. Schily and Gregor J. Gassner},
  title     = {Entropy stable numerical approximations for the isothermal and polytropic {E}uler equations},
  journal   = {BIT Numerical Mathematics},
  year      = {2020},
  volume    = {60},
  number    = {3},
  pages     = {791--824},
  doi       = {10.1007/s10543-019-00789-w}
}

@article { AControlVolumeModeloftheCompressibleEulerEquationswithaVerticalLagrangianCoordinate,
      author = "Xi Chen and Natalia Andronova and Bram Van Leer and Joyce E. Penner and John P. Boyd and Christiane Jablonowski and Shian-Jiann Lin",
      title = "A Control-Volume Model of the Compressible {E}uler Equations with a Vertical {L}agrangian Coordinate",
      journal = "Monthly Weather Review",
      year = "2013",
      publisher = "American Meteorological Society",
      address = "Boston MA, USA",
      volume = "141",
      number = "7",
      doi = "10.1175/MWR-D-12-00129.1",
      pages=      "2526 - 2544",
}

@article{NUMA,
author = {Andreas Müller and Michal A Kopera and Simone Marras and Lucas C Wilcox and Tobin Isaac and Francis X Giraldo},
title ={Strong scaling for numerical weather prediction at petascale with the atmospheric model {NUMA}},

journal = {The International Journal of High Performance Computing Applications},
volume = {33},
number = {2},
pages = {411-426},
year = {2019},
doi = {10.1177/1094342018763966}
}

@article { ScharetAl,
      author = "Christoph Schär and Daniel Leuenberger and Oliver Fuhrer and Daniel Lüthi and Claude Girard",
      title = "A New Terrain-Following Vertical Coordinate Formulation for Atmospheric Prediction Models",
      journal = "Monthly Weather Review",
      year = "2002",
      publisher = "American Meteorological Society",
      address = "Boston MA, USA",
      volume = "130",
      number = "10",
      doi = "10.1175/1520-0493(2002)130<2459:ANTFVC>2.0.CO;2",
      pages=      "2459 - 2480",
}

@article { ACompressibleModelfortheSimulationofMoistMountainWaves,
      author = "Dale R.  Durran and Joseph B.  Klemp",
      title = "A Compressible Model for the Simulation of Moist Mountain Waves",
      journal = "Monthly Weather Review",
      year = "1983",
      publisher = "American Meteorological Society",
      address = "Boston MA, USA",
      volume = "111",
      number = "12",
      doi = "10.1175/1520-0493(1983)111<2341:ACMFTS>2.0.CO;2",
      pages=      "2341 - 2361",
}

@article{MeansGenerated,
 ISSN = {0025570X, 19300980},
 author = {Hongwei Chen},
 journal = {Mathematics Magazine},
 number = {5},
 pages = {397--399},
 publisher = {Mathematical Association of America},
 title = {Means Generated by an Integral},
 urldate = {2025-08-18},
 volume = {78},
 year = {2005},
 doi={10.2307/30044201}
}

@book{Gottliebs2011,
  title={Strong stability preserving {R}unge-{K}utta and multistep time
         discretizations},
  author={Gottlieb, Sigal and Ketcheson, David I and Shu, Chi-Wang},
  year={2011},
  publisher={World Scientific},
  address={Singapore}
}

@Inbook{Barth1999,
author="Barth, Timothy J.",
editor="Kr{\"o}ner, Dietmar
and Ohlberger, Mario
and Rohde, Christian",
title="Numerical Methods for Gasdynamic Systems on Unstructured Meshes",
bookTitle="An Introduction to Recent Developments in Theory and Numerics for Conservation Laws: Proceedings of the International School on Theory and Numerics for Conservation Laws, Freiburg/Littenweiler, October 20--24, 1997",
year="1999",
publisher="Springer Berlin Heidelberg",
address="Berlin, Heidelberg",
pages="195--285",
isbn="978-3-642-58535-7",
doi="10.1007/978-3-642-58535-7_5"
}

@article{sjogreen2018high,
  title={High order entropy conservative central schemes for wide ranges of
         compressible gas dynamics and {MHD} flows},
  author={Sj{\"o}green, Bj{\"o}rn and Yee, HC},
  journal={Journal of Computational Physics},
  volume={364},
  pages={153--185},
  year={2018},
  publisher={Elsevier},
  doi={10.1016/j.jcp.2018.02.003}
}

@article{ranocha2023efficient,
  title={Efficient implementation of modern entropy stable and kinetic energy
         preserving discontinuous {G}alerkin methods for conservation laws},
  author={Ranocha, Hendrik and Schlottke-Lakemper, Michael and Chan, Jesse and
          Rueda-Ramirez, Andr{\'e}s M and Winters, Andrew R and
          Hindenlang, Florian and Gassner, Gregor J},
  journal={ACM Transactions on Mathematical Software},
  year={2023},
  month={09},
  volume={49},
  issue={4},
  doi={10.1145/3625559}
}

@article{bezanson2017julia,
  title={Julia: {A} Fresh Approach to Numerical Computing},
  author={Bezanson, Jeff and Edelman, Alan and Karpinski, Stefan and
          Shah, Viral B},
  journal={SIAM Review},
  volume={59},
  number={1},
  pages={65--98},
  year={2017},
  publisher={SIAM},
  doi={10.1137/141000671}
}

@article{ranocha2022adaptive,
  title={Adaptive numerical simulations with {T}rixi.jl:
         {A} case study of {J}ulia for scientific computing},
  author={Ranocha, Hendrik and Schlottke-Lakemper, Michael and
          Winters, Andrew Ross and Faulhaber, Erik and Chan, Jesse and
          Gassner, Gregor J},
  journal={Proceedings of the JuliaCon Conferences},
  volume={1},
  number={1},
  pages={77},
  year={2022},
  month={01},
  publisher={The Open Journal},
  doi={10.21105/jcon.00077}
}

@article{schlottkelakemper2021purely,
  title={A purely hyperbolic discontinuous {G}alerkin approach for
         self-gravitating gas dynamics},
  author={Schlottke-Lakemper, Michael and Winters, Andrew R and
          Ranocha, Hendrik and Gassner, Gregor J},
  journal={Journal of Computational Physics},
  pages={110467},
  year={2021},
  month={06},
  volume={442},
  publisher={Elsevier},
  doi={10.1016/j.jcp.2021.110467}
}

@article{kraaijevanger1991contractivity,
  title={Contractivity of {R}unge-{K}utta methods},
  author={Kraaijevanger, Johannes Franciscus Bernardus Maria},
  journal={BIT Numerical Mathematics},
  volume={31},
  number={3},
  pages={482--528},
  year={1991},
  publisher={Springer},
  doi={10.1007/BF01933264}
}

@article{rackauckas2017differentialequations,
  title={{DifferentialEquations.jl} {--} {A} Performant and Feature-Rich
         Ecosystem for Solving Differential Equations in {J}ulia},
  author={Rackauckas, Christopher and Nie, Qing},
  journal={Journal of Open Research Software},
  volume={5},
  number={1},
  pages={15},
  year={2017},
  publisher={Ubiquity Press},
  doi={10.5334/jors.151}
}

@article{danisch2021makie,
  title={Makie.jl: Flexible high-performance data visualization for {J}ulia},
  author={Danisch, Simon and Krumbiegel, Julius},
  journal={Journal of Open Source Software},
  volume={6},
  number={65},
  pages={3349},
  year={2021},
  doi={10.21105/joss.03349}
}

@article{hunter2007matplotlib,
  title={Matplotlib: {A} {2D} graphics environment},
  author={Hunter, J. D.},
  journal={Computing in Science \& Engineering},
  volume={9},
  number={3},
  pages={90--95},
  year={2007},
  publisher={IEEE Computer Society},
  doi={10.1109/MCSE.2007.55}
}

@article{michel2025towards,
  title={Towards a fully well-balanced and entropy-stable scheme for the {E}uler equations with gravity: General equations of state},
  author={Michel-Dansac, Victor and Thomann, Andrea},
  journal={Computers \& Fluids},
  pages={106853},
  year={2025},
  publisher={Elsevier},
  doi={10.1016/j.compfluid.2025.106853}
}

@article{ranocha2025error,
  title={On error-based step size control for discontinuous {G}alerkin
         methods for compressible fluid dynamics},
  author={Ranocha, Hendrik and Winters, Andrew R and
          Castro, Hugo Guillermo and Dalcin, Lisandro and
          Schlottke-Lakemper, Michael and Gassner, Gregor Josef
          and Parsani, Matteo},
  journal={Communications on Applied Mathematics and Computation},
  volume={7},
  pages={3--39},
  year={2025},
  month={02},
  doi={10.1007/s42967-023-00264-y},
  eprint={2209.07037},
  eprinttype={arxiv},
  eprintclass={math.NA}
}

@article{ranocha2021optimized,
  title={Optimized {R}unge-{K}utta Methods with Automatic Step Size Control
         for Compressible Computational Fluid Dynamics},
  author={Ranocha, Hendrik and Dalcin, Lisandro and Parsani, Matteo
          and Ketcheson, David I},
  journal={Communications on Applied Mathematics and Computation},
  volume={4},
  pages={1191--1228},
  year={2021},
  month={11},
  doi={10.1007/s42967-021-00159-w},
  eprint={2104.06836},
  eprinttype={arxiv},
  eprintclass={math.NA}
}

@inproceedings{jahdali2021optimized,
  title={Optimized Explicit {R}unge-{K}utta Schemes for Entropy Stable
         Discontinuous Collocated Methods Applied to the {E}uler and
         {N}avier-{S}tokes equations},
  author={Al Jahdali, Rasha and Boukharfane, Radouan and Dalcin, Lisandro and
          Parsani, Matteo},
  booktitle={AIAA Scitech 2021 Forum},
  pages={0633},
  year={2021},
  doi={10.2514/6.2021-0633}
}
